\documentclass[11pt,a4paper]{article}
\usepackage[margin=0.75in]{geometry}
\usepackage[utf8]{inputenc}
\usepackage{cmap}
\usepackage[T1]{fontenc}
\usepackage{lmodern}
\usepackage[square, numbers]{natbib}     %[square, numbers]
\usepackage[colorlinks=true]{hyperref}
\usepackage{amssymb, amsthm, amsmath, amsfonts}
\usepackage{authblk, dsfont, mathrsfs, color, bm, enumitem, mathtools}
\usepackage[dvipsnames]{xcolor}
\hypersetup{colorlinks, linkcolor={red!80!black}, citecolor={blue!80!black}, urlcolor={ForestGreen}}
\definecolor{red}{RGB}{163, 31, 52}
\definecolor{gray}{RGB}{194, 192, 191}
\definecolor{blue}{RGB}{59, 89, 152}
\definecolor{green}{RGB}{2, 138, 67}
\definecolor{darkgreen}{rgb}{0.0, 0.52, 0.24}
\definecolor{shadecolor}{rgb}{0.9,0.9,1}

\usepackage{tabularx}
\usepackage{tabularray}

\usepackage{comment}

\usepackage{graphicx, caption, subcaption}

\theoremstyle{definition}
\newtheorem{theorem}{Theorem}[section]

\newtheorem{assumption}[theorem]{Assumption}
\newtheorem{proposition}[theorem]{Proposition}
\newtheorem{corollary}[theorem]{Corollary}
\newtheorem{lemma}[theorem]{Lemma}
\newtheorem{remark}[theorem]{Remark}
\newtheorem{example}[theorem]{Example}
\newtheorem{fact}[theorem]{Fact}

\DeclareMathOperator*{\argmax}{arg\,max}
\DeclareMathOperator*{\argmin}{arg\,min}

\DeclareMathOperator*{\diam}{diam}

\DeclareMathOperator*{\IS}{IS}
\DeclareMathOperator*{\ET}{ET}
\DeclareMathOperator*{\MT}{MT}
\DeclareMathOperator*{\MM}{MM}

\DeclareMathOperator*{\st}{s.t.}
\DeclareMathOperator*{\trace}{Tr}

\newcommand{\set}[2]{\left\{ #1 : #2 \right\}}
\newcommand{\tset}[2]{\{ #1 : #2 \}}

\newcommand{\norm}[1]{\| #1 \|}
\newcommand{\one}[1]{\mathbf 1 \left\{ #1 \right\}}
\renewcommand{\tfrac}[2]{#1/#2}
\renewcommand{\Re}{\mathbb R}
\renewcommand{\S}{\mathbb S}
\newcommand{\mc}{\mathcal}
\newcommand{\mb}{\mathbb}
\renewcommand{\d}{{\mathrm d}}
\newcommand{\defn}{:=}

\DeclarePairedDelimiterX{\inner}[2]{\langle}{\rangle}{#1, #2}
\usepackage{import}
\usepackage{tikz}
\usepackage{pgfplots}
\usetikzlibrary{calc}
\usetikzlibrary{shapes.multipart}
\usetikzlibrary{decorations.pathreplacing}
\usetikzlibrary{arrows.meta}
\usepgfplotslibrary{fillbetween}
\usetikzlibrary{patterns}
\usetikzlibrary{tikzmark}
\usetikzlibrary{fit}
\usepackage{csvsimple}
\pgfplotsset{compat=newest}

\title{Stochastic Optimization with Optimal Importance Sampling}
\author[1]{Liviu Aolaritei}
\author[2]{Bart P.G.\ Van Parys}
\author[3]{Henry Lam}
\author[1,4]{Michael I. Jordan}
\affil[1]{University of California, Berkeley, USA \protect\\ \texttt{liviu.aolaritei@berkeley.edu, jordan@cs.berkeley.edu}}
\affil[2]{CWI, Amsterdam, The Netherlands \protect\\ \texttt{bart.van.parys@cwi.nl}}
\affil[3]{Columbia University, New York, USA  \protect\\
\texttt{khl2114@columbia.edu}}
\affil[4]{Inria, Paris, France}
\date{}

\begin{document}
\maketitle

\begin{abstract}
Importance Sampling (IS) is a widely used variance reduction technique for enhancing the efficiency of Monte Carlo methods, particularly in rare-event simulation and related applications. Despite its effectiveness, the performance of IS is highly sensitive to the choice of the proposal distribution and often requires stochastic calibration. While the design and analysis of IS have been extensively studied in estimation settings, applying IS within stochastic optimization introduces a lesser-known fundamental challenge: the decision variable and the importance sampling distribution are mutually dependent, creating a circular optimization structure. This interdependence complicates both convergence analysis and variance control. In this paper, we consider the generic setting of convex stochastic optimization with linear constraints. We propose a single-loop stochastic approximation algorithm, based on a variant of Nesterov's dual averaging, that jointly updates the decision variable and the importance sampling distribution, notably without time-scale separation or nested optimization. The method is globally convergent and achieves the minimal asymptotic variance among stochastic gradient schemes, which moreover matches the performance of an oracle sampler adapted to the optimal solution and thus effectively resolves the circular optimization challenge.
\end{abstract}

\section{Introduction}
\label{sec:introduction}

Importance sampling (IS) is a powerful variance reduction technique that improves the efficiency of Monte Carlo methods. The key idea is to generate input samples from a deliberately distorted distribution, and then reweight the outputs using a likelihood ratio that accounts for the discrepancy between the original and distorted distributions. With a carefully chosen IS distribution, this approach can yield dramatic reductions in variance, particularly in rare-event simulation, often outperforming naïve Monte Carlo methods by several orders of magnitude. For foundational references, see \citep{bucklew2004introduction,  juneja2006rare,blanchet2012state}.

Despite its popularity, IS is often described as a “double-edged sword”. Its performance depends critically on the choice of the proposal distribution, which is typically sensitive to the underlying model. While a well-chosen IS distribution can lead to dramatic variance reductions, a poor choice can result in catastrophic performance degradation \cite{glynn1989importance, sadowsky2002large,l2010asymptotic,owen2000safe, hesterberg1995weighted}. A large body of work has therefore focused on how to design and calibrate effective IS distributions, using either analytical approaches based on large deviations theory or adaptive search methods that optimize IS parameters over suitable classes of distributions. The first approach, originating with \cite{siegmund1976importance}, leverages the change of measure used to establish large deviations rates \cite{dembo2016lecture} as a principled and often near-optimal importance sampling distribution, and has been applied to various systems across problem domains such as queueing \cite{dupuis2009importance,blanchet2014rare,kroese1999efficient,ridder2009importance}, finance and insurance \cite{glasserman2008fast,glasserman2005importance,asmussen1985conjugate,collamore2002importance,chen2019efficient}, reliability \cite{rubino2009markovian,heidelberger1995fast,nicola2002fast}, biology and multiscale processes \cite{sandmann2009rare,dupuis2012importance,vanden2012rare}. However, it typically requires analytically tractable models that admit a large deviations principle sufficiently explicit to be converted into efficient algorithms. The second approach is more versatile, in that it trades analytical precision for flexibility and requires less a priori model knowledge, but generally offers weaker theoretical guarantees. A common approach is the so-called cross-entropy method that casts the IS parameter search problem as a Kullback–Leibler divergence minimization against the oracle best IS distribution \cite{de2005tutorial,rubinstein2004cross}. This typically amounts to a sequential search process involving iterative updates of the IS parameters in order to control the underlying estimation variances. Besides this approach, other effective adaptive methods or enhancements to cross-entropy have been proposed \cite{egloff2010quantile, bugallo2017adaptive,botev2013markov,chan2012improved,grace2014automated,parpas2015importance}. In this work, we adopt the latter perspective, focusing on methods that are broadly applicable and emphasize practical versatility.

Specifically, we depart from the probability estimation problem that is extensively studied in the rare-event literature described above. Our main focus is the use of importance sampling (IS) for solving stochastic optimization problems that involve rare events. More precisely, we consider problems of the form $\min_{\theta\in\Theta} \mathbb E_{X \sim \mathbb P}\left[ F(\theta, X) \right]$, where $\theta$ is a decision variable in a feasible region $\Theta$ and $X$ is a random variable drawn from a distribution $\mathbb P$. In settings where the loss function $F(\theta, X)$ involves, for example, indicators of rare events, standard sampling-based optimization algorithms become inefficient due to the scarcity of informative samples.
In such cases, IS can significantly improve sample efficiency. However, while there is a substantial literature on IS for estimating probabilities and expectations, the integration of IS into decision-making and stochastic optimization has only recently begun to receive focused attention (see Section~\ref{subsec:related:work} on related work). Our work contributes to this emerging line of research. In particular, we present what is, to the best of our knowledge, the first framework that bridges state-of-the-art techniques from stochastic first-order optimization and IS parameter calibration. 

We now describe our problem setting and its underlying challenges more precisely. We consider a stochastic optimization problem where the objective takes the form $f(\theta) := \mathbb{E}_{X \sim \mathbb{P}}[F(\theta, X)]$, but the expectation is not available in closed form. Instead, we assume access to samples from the nominal distribution $\mathbb P$ as well as from a parametrized class of IS distributions. Our goal is to leverage this IS family to accelerate the optimization process. This setup presents several layers of difficulty. The first challenge stems from the sensitivity of IS performance to the choice of sampling distribution: an effective IS distribution must be tightly matched to the underlying model. In an optimization context, however, the model is inherently uncertain, since the optimal decision $\theta^\star$ is not known a priori. This creates a fundamental tension: without a good decision, we cannot identify an effective IS distribution; and without an effective IS distribution, sampling is too inefficient to reliably identify a good decision. This circular dependency has recently been highlighted and studied by \citet{he2023adaptive}, who propose selecting the IS distribution through an iterative process. \citet{he2023adaptive}, however, crucially assumes an a priori known scheme that maps model parameters into effective IS parameters. This important presumption implicitly limits applicability to tractable models that are amenable to analysis via large deviations techniques, in which case a computable mapping can be derived to obtain IS parameters from the model specifications. In general, as we have described earlier, we aim to offer a more versatile and widely applicable approach. In this case, IS calibration, even for a \emph{fixed} model, requires an auxiliary sampling-based optimization routine such as the cross-entropy method \cite{deboer2005tutorial}, direct variance minimization \cite{owen2000safe}, or their variants. As a first goal of this paper, we develop an efficient algorithm that simultaneously handles the optimization of the IS parameters and the underlying decision problem. With this goal in mind, we study our approach under two additional, substantial, generality considerations from an optimization perspective. First, we allow for constraints on the decision variable that introduce another layer of difficulty. The optimal IS distribution can change discontinuously with the active constraint set, which can undermine the stability of IS-based procedures. As a result, one is effectively required to solve two intertwined optimization problems: one for the decision $\theta$, and one for the IS parameters, and this interdependence can multiply the computational cost. Second, we aim to establish global optimality guarantees. Even if the objective is convex in $\theta$, and the IS calibration problem is well-posed, the interaction between the two leads to a lack of joint convexity, making it nontrivial to establish global convergence or optimality for the combined procedure.

\subsection{Contributions}
\label{subsec:contributions}

Given the above discussion, our main contribution is a stochastic approximation framework that integrates versatile, search-based importance sampling directly into convex stochastic optimization with linear constraints. Our core idea is to jointly update the decision variable and the importance sampling distribution using a single-loop stochastic approximation scheme based on a joint variant of Nesterov’s dual averaging (NDA). The resulting method operates without time-scale separation or nested optimization and applies to broad families of importance sampling distributions. Taken together, our work represents, to the best of our knowledge, the first attempt to integrate first-order optimization methods with importance sampling, at a level sophisticated enough to allow the attainment of several notable new guarantees. More precisely, our main technical contributions are summarized below.

\begin{enumerate}
\item \textbf{Global convergence without joint convexity.}
We establish global convergence of the joint decision and sampling iterates generated by the proposed algorithm under convexity of the objective function and standard regularity conditions on the importance sampling family, including log-convex likelihood ratios as in exponential tilting or mean translation. These guarantees hold despite the fact that the combined optimization problem over the decision and sampling parameters is not jointly convex.

\item \textbf{No time-scale separation or nested optimization.}
The proposed method resolves the coupled decision and importance sampling problem using a single-loop stochastic approximation scheme based on a joint dual averaging update. The algorithm does not require time-scale separation or inner optimization loops, in contrast to bilevel or alternating procedures commonly used in related settings. This structure leads to a streamlined analysis and a simple algorithmic implementation.

\item \textbf{Joint handling of linear constraints and importance sampling.}
Linear constraints introduce additional challenges for importance sampling calibration, since the optimal sampling distribution can vary discontinuously with the active constraint set. We show that, by leveraging a variant of Nesterov’s dual averaging within the joint stochastic approximation scheme, the active constraints are identified in finite time. This property is key to establishing global convergence and asymptotic optimality for convex stochastic optimization problems with linear constraints.

\item \textbf{Asymptotically optimal variance without prior importance sampling knowledge.}
We show that the proposed joint scheme resolves the circular dependence between decision optimization and importance sampling calibration in the asymptotic regime. In particular, the averaged decision iterates achieve minimal asymptotic variance among stochastic gradient schemes, matching the performance of an oracle that has access to the optimal importance sampling distribution in advance. Thus, in the asymptotic sense, the performance of our approach cannot be improved.
\end{enumerate}

\subsection{Related Work}
\label{subsec:related:work}

This work lies at the intersection of two large bodies of literature: stochastic optimization and importance sampling. We have already described a range of related work in the introduction. Here, we focus our discussion on contributions that directly address the intersection of these two bodies of literature and that we have not covered yet.

The use of importance sampling (IS) in stochastic optimization has primarily been studied in the context of specific objectives, such as quantile, value-at-risk (VaR), and conditional value-at-risk (CVaR) estimation. \citet{egloff2010quantile} propose a stochastic approximation scheme to asymptotically identify a minimum-variance importance sampler for quantile estimation. While they establish almost sure convergence to the desired quantile, they do not analyze the asymptotic variance of the resulting estimator. \citet{pan2020adaptive} develop an adaptive IS approach for quantile estimation using a two-layer model, where the inner layer employs a heuristic metamodel. Under this structure, their method achieves convergence to the quantile and demonstrates variance reduction relative to standard Monte Carlo, but they do not analyze the magnitude of improvement. Among the works most closely related to ours are \citet{he2023adaptive} and \citet{Bardou2009computing}, which we now discuss in more detail in terms of both their connections to and distinctions from our work.

\citet{he2023adaptive} consider adaptive IS for stochastic root-finding problems, which can arise from solving optimization problems. They raise the circularity challenge we described in the introduction, where, on one hand, good IS parameters rely on accurate model specifications, including the decision, while, on the other hand, a good decision relies on efficient IS that hinges on good parameters. To tackle this challenge, \citet{he2023adaptive} propose an adaptive approach to embed IS into sequential algorithms such as stochastic approximation, and show how appropriate configuration can achieve the same asymptotic variance as if using an IS based on a priori knowledge of the optimal solution. They refer to this property as minimax optimality. While their work, like us, directly tackles the circularity issue in IS calibration, it crucially assumes knowledge of a mapping from decisions to optimal IS parameters, which limits applicability to problems with sufficient analytical structure—typically those where large deviations theory can be used to derive this mapping. In contrast, our method performs stochastic calibration of IS parameters without requiring such mappings to be known in advance. Furthermore, our framework provides global convergence guarantees for general convex objectives with linear constraints, whereas \citet{he2023adaptive} focus on local convergence in unconstrained, predominantly univariate settings.

\citet{Bardou2009computing} study IS for VaR and CVaR estimation, proposing a stochastic approximation scheme that asymptotically identifies a minimum-variance IS distribution. Under certain convexity conditions on the IS class, they establish a central limit theorem with optimal asymptotic variance. Our work extends this framework in several key directions: (i) we consider general multivariate stochastic optimization problems rather than scalar objectives; (ii) our algorithm handles linear constraints via a recent stochastic approximation technique introduced by \citet{duchi2021asymptotic}; and (iii) we operate under bounded IS parameter spaces, which simplifies the procedure relative to that of \citet{Bardou2009computing}, who employ an auxiliary IS step to handle unbounded domains.

Beyond the IS literature, our work is also connected to broader questions in stochastic optimization. The performance of stochastic algorithms, particularly in modern large-scale settings such as deep learning \cite{bengio2012practical}, is often highly sensitive to hyperparameter tuning. This positions our approach within the emerging literature on meta-optimization \cite{hospedales2021meta}, where higher-level optimization problems govern the tuning of algorithmic components such as learning rates or sampling distributions.

Finally, we emphasize that, while our method guarantees asymptotic optimality within the chosen IS class, we do not make claims about how specific IS families accelerate convergence in practice. This is highly problem-dependent and lies beyond the scope of this work. For recent efforts in this direction, we refer to \citet{deo2023achieving, deo2024importance}, who study exponential tilting and self-structuring transformations to accelerate CVaR optimization.

Taken together, these works underscore the need for scalable, theoretically sound methods that unify IS calibration and decision-making. Our framework addresses this need through a single-loop, globally convergent approach.

\subsection{Organization and Notation}

\noindent\textbf{Organization.}
Section \ref{sec:sa:with:is} sets up the constrained convex stochastic optimization problem and reviews asymptotic optimality results for stochastic approximation methods that motivate variance as the key efficiency criterion. Section \ref{sec:importance-sampling} develops the importance sampling framework and formulates the variance-minimization objective that defines the optimal proposal distribution within a parametrized IS family. Section \ref{sec:SO:with:IS} presents the proposed adaptive importance sampling approach and the joint update scheme for simultaneously learning the decision variable and the IS parameters (see equation~\eqref{eq:NDA:SA+IS:iteration}). Section \ref{sec:conv:analysis} establishes almost sure convergence of the coupled iterates and derives central limit theorems showing asymptotic variance optimality of the averaged iterates (and associated projected-gradient optimality). Section \ref{sec:numerics} provides a numerical illustration demonstrating adaptivity and variance reduction in a rare-event regime. Technical lemmas and auxiliary results are collected in the Appendix.

\medskip

\noindent\textbf{Notation.} Throughout the paper, $\|\cdot\|$ denotes the Euclidean norm. For any matrix $A$, we write $A^\dagger$ for its Moore–Penrose pseudoinverse. The set of all symmetric positive definite matrices is denoted by $\mathbb{S}_{++}$. The set of strictly positive real numbers is written as $\mathbb{R}_{++} = \{x \in \mathbb{R}: x > 0\}$. Given a set $\mathcal{A}$, the characteristic function $\mathbf{1}_{\mathcal{A}}(x)$ equals $1$ if $x \in \mathcal{A}$ and $0$ otherwise. The indicator function $\delta_{\mathcal{A}}(x)$ equals $0$ if $x \in \mathcal{A}$ and $\infty$ otherwise. For random variables $\{X_n\}_{n \in \mathbb{N}}$ and $X$, we write $X_n \overset{\text{a.s.}}{\to} X$ to denote almost sure convergence and $X_n \overset{d}{\to} X$ for convergence in distribution.

\section{Asymptotic Optimality in Stochastic Optimization}
\label{sec:sa:with:is}

In this paper we consider the problem of efficiently computing the optimal solution of the following stochastic optimization  problem
\begin{align}
  \label{eq:so}
  \begin{split}
    \min_\theta\;\; & f(\theta):= \mathbb E_{X \sim \mathbb P}\left[ F(\theta, X) \right] \\
    \st \;\; & \theta \in \Theta:=\{\theta \in \mathbb R^s:\, A \theta \leq b\},
  \end{split}
  \tag{SO}
\end{align}
with stochastic objective function $F:\mathbb R^s \times \mathbb R^r \to \mathbb R$, technology matrix $A \in \mathbb R^{p \times s}$, budget vector $b \in \mathbb R^p$, and where $X$ is a random vector distributed according to the probability distribution $\mathbb P$ on $\mathcal X\subseteq \Re^r$. A standing assumption in the paper is that the distribution $\mathbb P$ is known, or more accurately, that we can generate samples from $\mb P$ efficiently. We denote by $\theta^\star$ an optimal solution of \eqref{eq:so}.

Stochastic optimization problems of the form \eqref{eq:so} are ubiquitous in a myriad of domains, including machine learning, stochastic control, queuing theory, portfolio selection and risk management; see \citet{schneider2007stochastic,shapiro2021lectures}.
We consider convex stochastic optimization problems which satisfy the following structural assumptions.
% Not all of our results will need every assumption, but we opt to collect all assumptions in one place for convenience.

\begin{assumption}[Stochastic Optimization]~
\label{assump:SO}
\begin{enumerate}[label=(\roman*)]
\item\label{assump:SO:convex} The objective function $f:\Theta\to\Re$ is convex.

\item \label{assump:SO:differentiable} The objective function $f:\Theta\to\Re$ is continuously differentiable.

\item \label{assump:SO:twice:differentiable} The objective function $f:\Theta\to\Re$ is twice continuously differentiable in a neighborhood of $\theta^\star$.

\item\label{assump:SO:unique} The stochastic optimization problem~\eqref{eq:so} admits a unique minimizer $\theta^\star$ with $\nabla^2 f(\theta^\star)\in \S_{++}$.
  
\item\label{assump:SO:bounded} The constraint set $\Theta$ is bounded.

\end{enumerate}
\end{assumption}

Assumption \ref{assump:SO}\ref{assump:SO:convex} ensures that \eqref{eq:so} is a convex minimization problem. A sufficient condition for Assumption \ref{assump:SO}\ref{assump:SO:convex} to hold is that the function $\theta \mapsto F(\theta, x)$ is a convex function for every $x\in \mathcal X$, and
\(
\mathbb E_{X \sim \mathbb P}\left[\,|F(\theta,X)|\right] < \infty
\)
for all $\theta\in \Theta$.
Assumption \ref{assump:SO}\ref{assump:SO:differentiable} demands that the objective function be sufficiently regular. In particular, Assumptions \ref{assump:SO}\ref{assump:SO:differentiable} and \ref{assump:SO}\ref{assump:SO:bounded} together ensure that the minimum in problem \eqref{eq:so} is achieved via Weierstrass' extreme value theorem. We remark that if $\mc X$ is a finite set then $\theta \mapsto F(\theta, X)$ being continuously differentiable for every $x\in \mathcal X$ is sufficient to ensure that Assumption \ref{assump:SO}\ref{assump:SO:differentiable} holds.
In the more general case, Assumption \ref{assump:SO}\ref{assump:SO:differentiable} holds if the functions $\{G(\theta, x) := \nabla_{\theta}F(\theta, x)\}_{x\in \mathcal X}$ are equicontinuous and $\mathbb E_{X \sim \mathbb P}\left[\norm{G(\theta,X)}\right] < \infty$ for all $\theta\in \Theta$.
Finally, Assumption~\ref{assump:SO:twice:differentiable} allows us to employ the delta method to prove asymptotic normality results.

In most practical settings, however, neither the objective function $f:\Theta\to\Re$ nor its gradient $\nabla f:\Theta\to\Re^s$ can be efficiently evaluated, as this would require high-dimensional integration. Instead, an optimizer typically only has access to stochastic gradients $G(\theta, X)$, which represent the gradient of our objective function only in expectation. Optimization algorithms which attempt to solve the problem \eqref{eq:so} using only stochastic gradients are generally known under the name \emph{stochastic approximation algorithms}, and have been studied since the seminal paper by \citet{robbins1951stochastic}. 

In what follows, we would like to argue that, under Assumption \ref{assump:SO}\ref{assump:SO:unique}, a good yardstick to measure the ability of a stochastic approximation algorithm to solve problem \eqref{eq:so} is the variance $\text{Var}_{X \sim \mathbb P}\left[G(\theta^\star,X)\right]$ at the optimal solution $\theta^\star$. Thus in later sections our goal will be to find stochastic gradients with ``small'' variance.

\subsection{Unconstrained Stochastic Optimization}
\label{subsec:unconstrained:stoch:approx}

To build intuition, we first consider unconstrained stochastic optimization problems, where $\Theta = \mathbb R^s$.
As the objective function $f$ is continuously differentiable and convex, solving the problem \eqref{eq:so} is equivalent to finding the root of its gradient function \citep[Proposition 5.4.7]{bertsekas2009convex}, i.e., problem~\eqref{eq:so} is equivalent to the first-order optimality condition
\begin{align}
  \label{eq:root-finding}
    \nabla f(\theta^\star)= 0.
\end{align}
A standard approach to solve root finding problems of the type \eqref{eq:root-finding} is to consider the \emph{Robbins-Monro stochastic approximation} (RM-SA) iteration
\begin{align}
\label{eq:SA}
    \theta_{n+1} &= \theta_n - \alpha_{n+1} K G(\theta_n, X_{n+1}),
\end{align}
with arbitrary $\theta_0 \in \mathbb R^s$, and with step-sizes $\alpha_{n}$ satisfying the classical convergence conditions $\sum_{i=1}^\infty \alpha_i = \infty$ and $\sum_{i=1}^\infty \alpha_i^2 < \infty$.
The standard choice for the step-sizes in RM-SA is $\alpha_n = \alpha/n$ for some constant $\alpha$.
Despite its simplicity, one drawback of this method is that its optimal version (having the smallest asymptotic variance \citep{polyak1992acceleration}, \cite[Chapter~VIII-3]{asmussen2007stochastic}; see below in equation \eqref{eq:PR-SA:CLT}), under Assumption \ref{assump:SO}\ref{assump:SO:unique}, requires forming $ K = \nabla^2 f(\theta^\star)^{-1}$, which, in turn, requires knowledge of the unknown optimal solution $\theta^\star$. Moreover, it has been observed that the RM-SA method is sensitive to the choice of the step-sizes, often resulting in poor practical performance (see the discussion in \cite[Section~4.5.3]{spall2005introduction}, and the elucidating example in \cite[Section~2.1]{nemirovski2009robust}).

An important improvement of the RM-SA method is the \emph{Polyak-Ruppert stochastic approximation} (PR-SA) iteration, proposed independently by Polyak \cite{polyak1990new, polyak1992acceleration} and Ruppert \cite{ruppert1988efficient}, which takes long step-sizes coupled with a subsequent averaging of the iterates,
\begin{align}
\label{eq:PR-SA}
\begin{split}
    \theta_{n+1} &= \theta_n - \alpha_{n+1} G(\theta_n, X_{n+1}),\\
    \bar{\theta}_n &= \frac{1}{n} \sum_{i=0}^{n-1} \theta_i,
\end{split}
\end{align}
for some $\theta_0 \in \mathbb R^s$. A standard choice for the step-sizes in PR-SA is $\alpha_n = \alpha/n^\gamma$, for $\gamma \in (1/2,1)$ and some constant $\alpha$ (see \cite[Assumption~3.4]{polyak1992acceleration}).
Under mild regularity conditions, in particular Assumption~\ref{assump:SO}\ref{assump:SO:unique}, \citet[Theorem~2]{polyak1992acceleration} prove that $\bar{\theta}_n \overset{\text{a.s.}}{\to} \theta^\star$ and that the following central limit theorem (CLT) holds:
\begin{align}
\label{eq:PR-SA:CLT}
    \sqrt{n}\left( \bar{\theta}_n - \theta^\star \right) \overset{d}{\to} \mathcal N \left(0, (\nabla^2  f(\theta^\star))^{-1} \text{Var}_{X \sim \mathbb P}\left[G(\theta^\star,X)\right] (\nabla^2  f(\theta^\star))^{-1}\right).
\end{align}
That is, the convergence of $\bar \theta_n$ to $\theta^\star$ takes place asymptotically at order $\mc O(1/\sqrt{n})$, while the variance $(\nabla^2  f(\theta^\star))^{-1} \text{Var}_{X \sim \mathbb P}\left[G(\theta^\star,X)\right] (\nabla^2  f(\theta^\star))^{-1}$ can be interpreted as a characterization of the convergence speed (at this order). Importantly, the asymptotic variance of the PR-SA method is \emph{optimal} among the stochastic gradient-based methods (see \cite[Chapter~8,~Section~4]{asmussen2007stochastic}); that is, the PR-SA iterates enjoy minimal (in a semidefinite sense) asymptotic variance among all estimates based on stochastic gradients.

The distance of the iterates $\bar \theta_n$ to the root $\theta^\star$ is one particular sub-optimality measure. Following the optimality condition \eqref{eq:root-finding}, another natural sub-optimality measure is the size of the gradients $\nabla {f}(\bar{\theta}_n)$ \citep{nesterov2012make}.
In fact, this optimality criterion has seen a surge of recent interest in the optimization community, since it generalizes more naturally to nonconvex problems \cite{foster2019complexity}. 
Under the same regularity assumptions as in \citet[Theorem~2]{polyak1992acceleration}, and since the gradient function is continuous (by Assumption~\ref{assump:SO}\ref{assump:SO:differentiable}), the Mann–Wald theorem ensures that $\nabla {f}(\bar{\theta}_n) \overset{\text{a.s.}}{\to} 0$. Moreover, using the optimality condition \eqref{eq:root-finding} and the fact that the gradient function is continuously differentiable in a neighborhood of $\theta^\star$ (by Assumption~\ref{assump:SO}\ref{assump:SO:twice:differentiable}), the delta method can be employed to recover the following CLT:
\begin{align}
\label{eq:gradient:asymp:normality}
    \sqrt{n}\, \nabla f(\bar{\theta}_n) \overset{d}{\to} \mathcal N \left(0, \text{Var}_{X \sim \mathbb P}\left[G(\theta^\star,X)\right]\right).
\end{align}
Following equation \eqref{eq:gradient:asymp:normality}, the quality of the iterates (as measured by the size of their gradients) in the optimization problem \eqref{eq:so} is also characterized by the variance at the optimal solution $\theta^\star$.

% ------------------------------------------------------------------------------------

\subsection{Constrained Stochastic Optimization}
\label{subsec:constrained:stoch:approx}

We now consider the general constrained stochastic optimization problem.
Following \citep[Proposition 5.4.7]{bertsekas2009convex}, the necessary and sufficient first-order optimality conditions of problem~\eqref{eq:so} under Assumption \ref{assump:SO}\ref{assump:SO:convex} and \ref{assump:SO}\ref{assump:SO:differentiable} reduce to:
\begin{equation}
  \label{eq:kkt-conditions-orig}
  \begin{split}
    &\exists \lambda^\star\geq 0,~ A \theta^\star - b \leq 0,\\
    &\nabla {f}(\theta^\star) + A^\top \lambda^\star=0, \\
    &(A \theta^\star - b)^\top \lambda^\star = 0.
  \end{split}
\end{equation}
It will be important in the sequel to partition the constraints in problem \eqref{eq:so} into active and inactive constraints.
First, for any feasible point $\theta \in \Theta$, we denote by $A_a^\theta$ and $b_a^\theta$ the active parts of the constraints, i.e., $A_{a}^\theta \theta - b_a^\theta=0$. Moreover, we denote by $A^\star_a$, and $b_a^\star$ the active parts of the constraints at the unique optimal solution $\theta^\star$. Secondly, we denote by $A_i^\theta$ and $b_i^\theta$ the inactive parts of the constraints at any $\theta \in \Theta$, i.e., $A_{i}^\theta \theta - b_i^\theta<0$, and by $A^\star_i$, and $b_i^\star$ the active parts of the constraint at $\theta^\star$. Let $\text{P}_{A^\star_a}= I - {A^\star_{a}}^\top ({A^\star_{a}} {A^\star_{a}}^\top)^\dagger {A^\star_{a}}$ be the orthogonal projector onto the null space of the active constraints $\{\theta \in \mathbb R^s:\; A^\star_{a} \theta = 0\}$, 
where $({A^\star_{a}} {A^\star_{a}}^\top)^\dagger$ denotes the Moore-Penrose inverse of the matrix ${A^\star_{a}} {A^\star_{a}}^\top$. Given the partition into active and inactive constraints, the optimality conditions \eqref{eq:kkt-conditions-orig} imply the following fact.

\begin{fact}[Optimality conditions]
\label{fact:optimality:conditions}
  We have
  \begin{equation}
  \label{eq:kkt-conditions}
  A^\star_a \theta^\star - b_a^\star=0, \quad \text{P}_{A^\star_a} \nabla {f}(\theta^\star)=0.
\end{equation}
\end{fact}
\begin{proof}
From the complementarity optimality condition in equation \eqref{eq:kkt-conditions-orig} it follows immediately that the dual variables associated with the inactive components vanish, i.e., $\lambda^\star_i=0$.
Thus, we must have that $\theta^\star$ is optimal in the unconstrained problem
$\min_\theta\, {f}(\theta_a + \text{P}_{A^\star_a}\theta) = \min_\theta \set{{f}(\theta)}{A^\star_a\theta-b^\star_a=0}$ for an arbitrary $\theta_a$ satisfying $A^\star_{a} \theta_a - b^\star_a=0$, as $(\theta^\star, \lambda_a^\star)$ can be easily verified to satisfy its optimality conditions $A^\star_a \theta^\star - b_a^\star=0$ and $\nabla {f}(\theta^\star) + {A^\star_a}^\top \lambda_a^\star=0$.
\end{proof}

The previous fact implies that we are looking for a point $\theta^\star$ in the affine subspace associated with the active constraints, for which the gradient is orthogonal to the active constraints; see Figure~\ref{fig:so}. This geometric observation motivates all active set methods \cite{nocedal1999numerical}, with the simplex method as the primary example.

\begin{figure}[t]
  \centering
  \begin{subfigure}[t]{0.43\textwidth}
    \centering
    \includegraphics[width=0.9\textwidth]{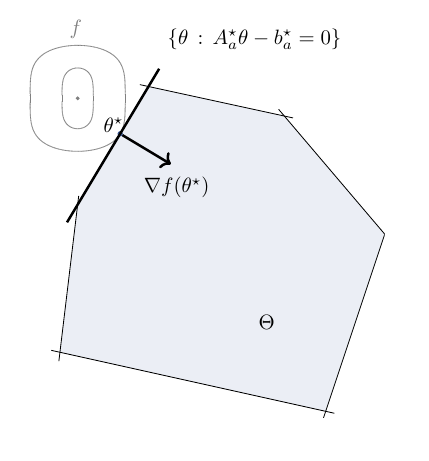}
    \subcaption{Problem \eqref{eq:so}.}
    \label{fig:so}
  \end{subfigure}%
  \hspace{5em}
  \begin{subfigure}[t]{0.43\textwidth}
    \centering
    \includegraphics[width=0.9\textwidth]{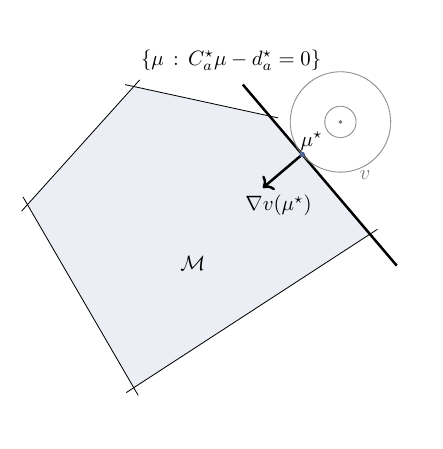}
    \subcaption{Problem \eqref{eq:optimal:IS:2}.}
    \label{fig:is}
  \end{subfigure}
  \caption{Two constrained stochastic optimization problem. The minimum $\theta^\star$ in \eqref{eq:so} is characterized as the minimum restricted to the active constraint set $\set{\theta}{A^\star_a\theta-b^\star_a=0}$. Likewise, the minimum $\mu^\star$ in \eqref{eq:optimal:IS:2} is characterized as the minimum restricted to the active constraint set $\set{\mu}{C^\star_a\mu-d^\star_a=0}$. }
  \label{fig:constrained-so}
\end{figure}

In the presence of constraints, the performance analysis of stochastic approximation algorithms is however much more delicate.
Through a local asymptotic minimax argument, \citet[Theorems~1 and 4]{duchi2021asymptotic} show that any sequence of approximations $\{\bar \theta_n\}_{n \in \mathbb N}$ which satisfies the CLT
\begin{align}
\label{eq:LB:CLT}
    \sqrt{n}\left( \bar{\theta}_n - \theta^\star \right) \overset{d}{\to} \mathcal N \left(0, \text{Q}^{\dagger}\, \text{Var}_{X \sim \mathbb P}\left[G(\theta^\star,X)\right] \text{Q}^{\dagger} \right),
\end{align}
with $\text{Q}:=\text{P}_{A^\star_{a}}\nabla^2  f(\theta^\star) \text{P}_{A^\star_{a}}$,
enjoys in fact optimal asymptotic convergence. It can be remarked here that the performance lower bound in equation \eqref{eq:LB:CLT} coincides with the asymptotic variance in \eqref{eq:PR-SA:CLT} of the classical PR-SA applied to the unconstrained problem $\min_\theta\, {f}(\theta_a + \text{P}_{A^\star_a}\theta)$, for an arbitrary $\theta_a$ satisfying $A^\star_{a} \theta_a - b^\star_a=0$.
However, the classical PR-SA applied to the unconstrained problem is not a viable algorithm as the active constraints are not known a priori. 
To solve this issue, \citet{duchi2021asymptotic} introduce the following variant of Nesterov's dual averaging (NDA) \cite{nesterov2009primal} iterates
\begin{align}
\label{eq:NDA}
\begin{split}
    \theta_{n+1} &= \argmin_{\theta \in \Theta} \left\{ \inner{\textstyle\sum_{k=0}^n \alpha_{k+1} G(\theta_k, X_{k+1})}{\theta} +\frac{1}{2} \| \theta - \theta_{0} \|^2\right\},\\
    \bar{\theta}_n &= \textstyle\frac{1}{n} \sum_{i=0}^{n-1} \theta_i,
\end{split}
\end{align}
with $\theta_0 \in \mathbb R^s$ and $\alpha_n = \alpha/n^\gamma$, for $\gamma \in (1/2,1)$ and some constant $\alpha>0$.
Notice that if the stochastic optimization problem~\eqref{eq:so} is unconstrained, then iteration~\eqref{eq:NDA} simply reduces to the classical PR-SA.
Hence, the iterates $\bar \theta_n$ in equation~\eqref{eq:NDA} can be considered a direct generalization of the classical PR-SA to the constrained setting.
Under standard regularity assumptions (similar in nature to those in \citet{polyak1992acceleration}) and under the condition that the constraint set $\Theta$ is polytopic (as in problem~\eqref{eq:so}), \citet{duchi2021asymptotic} show that the averaged iterates satisfy the almost sure convergence $\bar{\theta}_n \overset{\text{a.s.}}{\to} \theta^\star$, as well as the CLT \eqref{eq:LB:CLT}. Remarkably, the proposed iteration identifies the active constraints in finite time. 
That is, there exists some (random) finite $N \in \mathbb N$ such that $A^\star_{a} \theta_n - b^\star_a=0$ and $A^\star_{i} \theta_n - b_i<0$ for all $n \geq N$, and the iterations are identical to those of the PR-SA method in the affine subspace $\{\theta :\; A^\star_{a} \theta - b^\star_a=0\}$ for all $n \geq N$. 

\begin{remark}[Projected SA]
\label{remark:projectedSA}
Perhaps surprisingly, the standard projected versions of the RM-SA or PR-SA iterations
\begin{align*}
    \theta_{n+1} &= \argmin_{\theta \in \Theta} \left\{ \inner{\alpha_{n+1} G(\theta_n, X_{n+1})}{\theta} +\frac{1}{2} \| \theta - \theta_{n} \|^2\right\}
\end{align*}
(with subsequent averaging for PR-SA) fails to identify the active constraints in finite time. Indeed, there are many instances where the iterates do not satisfy the active constraints with constant non-zero probability at each iteration, and consequently jump off the constraint infinitely often (see \cite{lee2012manifold,duchi2021asymptotic} for details). Nonetheless, \citet{davis2023asymptotic} show that this standard projected version also achieves the guarantee \eqref{eq:LB:CLT} with minimal asymptotic variance. Moreover, unlike NDA, the optimality of projected RM-SA or PR-SA iterations is not restricted to polytopic constraint sets $\Theta$. As will become clear later, the finite-time identification of the active constraints is a key step in the convergence analysis of our method. For this reason, in this paper we will adopt the NDA iteration \eqref{eq:NDA}.\hfill $\clubsuit$
\end{remark}

Similarly to the unconstrained case and in view of the optimality conditions (\ref{eq:kkt-conditions}), an alternative sub-optimality metric is to consider the size of the gradient component in the nullspace associated with the active constraints, i.e., $\text{P}_{A^\star_{a}} \nabla {f}(\bar{\theta}_n)$, as well as the size of the residuals $A^\star_a\bar\theta_n-b^\star_a$.

\begin{lemma}[Projected gradient CLT]
\label{lemma:kkt-residual}
Let Assumptions~\ref{assump:SO}\ref{assump:SO:twice:differentiable}-\ref{assump:SO:unique} be satisfied. Then, any iterate sequence $\bar \theta_n$ enjoying the CLT \eqref{eq:LB:CLT} must also satisfy the CLT
\begin{align}
\label{eq:NDA:CLT}
    \sqrt{n}\, \text{P}_{A^\star_{a}}  \nabla f(\bar{\theta}_n) \overset{d}{\to} \mathcal N \left(0, \text{Var}_{X \sim \mathbb P}\left[\text{P}_{A^\star_{a}} G(\theta^\star,X) \right] \right).
\end{align}
\end{lemma}
\begin{proof}
Since the gradient function $\nabla f$ is continuously differentiable in a neighborhood of $\theta^\star$ (by Assumption~\ref{assump:SO}\ref{assump:SO:twice:differentiable}) and using the facts $\nabla^2 f(\theta^\star) \in \mb S_{++}$ and $\text{P}_{A^\star_{a}} \nabla f(\theta^\star) = 0$, the delta method can be employed to recover the following CLT:
\[
  \sqrt{n}\, \text{P}_{A^\star_{a}}  \nabla f(\bar{\theta}_n) \overset{d}{\to} \mathcal N \left(0, \text{Q} \text{Q}^{\dagger} \text{Var}_{X \sim \mathbb P}\left[G(\theta^\star,X) \right] \text{Q}^{\dagger} \text{Q} \right),
\]
with $\text{Q}:=\text{P}_{A^\star_{a}}\nabla^2  f(\theta^\star) \text{P}_{A^\star_{a}}$. We will now show that $\text{Q}\text{Q}^\dagger = \text{P}_{A^\star_{a}}$. We start by stating three facts. First, since $\text{P}_{A^\star_a}$ is an orthogonal projection matrix, we have that $\text{P}_{A^\star_a} \text{Q}^{\dagger} = \text{Q}^\dagger = \text{Q}^\dagger \text{P}_{A^\star_a}$. Secondly, from the definition of the Moore-Penrose inverse, we have that $\text{Q} \text{Q}^{\dagger} \text{Q} = \text{Q}$. Thirdly, observe that $\text{Im}(\text{Q}) = \text{Im}(\text{P}_{A^\star_a})$. Indeed, it is clear that $\text{Im}(\text{Q}) \subseteq \text{Im}(\text{P}_{A^\star_a})$. We will now prove the converse inclusion. Since both $\text{P}_{A^\star_{a}}$ and $\text{Q}$ are symmetric, by the fundamental theorem of linear algebra it is enough to prove that $\text{Ker}(\text{Q}) \subseteq \text{Ker}(\text{P}_{A^\star_a})$. For any $v \in \text{Ker}(\text{Q})$, we have that $v^\top \text{Q} v = v^\top \text{P}_{A^\star_{a}}\nabla^2  f(\theta^\star) \text{P}_{A^\star_{a}} v = 0$ which immediately implies that $\text{P}_{A^\star_{a}} v = 0$ using the fact that $\nabla^2  f(\theta^\star)$ is positive definite (by Assumption~\ref{assump:SO}\ref{assump:SO:unique}). 

We are now ready to prove that $\text{Q}\text{Q}^\dagger = \text{P}_{A^\star_{a}}$. For any $v \in \mathbb R^s$, let $v'  \in \mathbb R^s$ be such that $\text{P}_{A^\star_a} v = \text{Q} v'$. Then, the previous three facts guarantee that the chain of equalities $\text{Q} \text{Q}^{\dagger} v = \text{Q} \text{Q}^{\dagger} \text{P}_{A^\star_a} v = \text{Q} \text{Q}^{\dagger} \text{Q} v' = \text{Q} v' = \text{P}_{A^\star_a}v$ holds. This shows that $\text{Q} \text{Q}^{\dagger} v = \text{P}_{A^\star_{a}} v$ for all $v \in \mathbb R^s$, which implies that $\text{Q}\text{Q}^\dagger = \text{P}_{A^\star_{a}}$. The equality $\text{Q}^\dagger\text{Q} = \text{P}_{A^\star_{a}}$ can be proven analogously. 

Finally, \eqref{eq:NDA:CLT} follows from the equality $\text{Var}_{X \sim \mathbb P}\left[\text{P}_{A^\star_{a}} G(\theta^\star,X) \right] = \text{P}_{A^\star_{a}} \text{Var}_{X \sim \mathbb P}\left[G(\theta^\star,X) \right] \text{P}_{A^\star_{a}}$. This concludes the proof.
\end{proof}

Under suitable uniform integrability conditions, the CLTs~\eqref{eq:LB:CLT} and \eqref{eq:NDA:CLT} immediately imply that the following two limits (convergence of second moments) hold:
\begin{equation}
\label{eq:trace-identity}
    \lim_{n\to\infty} n \mathbb E \left[ ||A^\star_a\bar\theta_n-b^\star_a||_2^2 \right]=0, \quad \lim_{n\to\infty} n \mathbb E \left[ ||\text{P}_{A^\star_{a}} \nabla f(\bar \theta_n)||_2^2 \right] = \trace\left( \text{Var}_{X \sim \mathbb P}\left[\text{P}_{A^\star_{a}} G(\theta^\star,X)\right] \right).
\end{equation}
These limits highlight that any stochastic approximation iteration which hopes to attain the performance lower bound in equation~\eqref{eq:LB:CLT} must ``quickly'' identify the active constraints, in the sense that the residual norm $\|A^\star_a\bar\theta_n-b_a\|$ must decay to zero (in expectation) faster than $1/\sqrt{n}$. However, the gradient component in the nullspace associated with the active constraints, i.e., $\text{P}_{A^\star_{a}} \nabla {f}(\bar{\theta}_n)$, may decay at order $1/\sqrt{n}$. Consequently, an optimizer aiming for iterates with small (asymptotic) gradient norm should prefer stochastic gradients for which the trace of the projected variance $\trace\left(\text{Var}_{X \sim \mathbb P}\left[\text{P}_{A^\star_{a}} G(\theta^\star,X)\right]\right)=\trace\left(\text{P}_{A^\star_{a}} \text{Var}_{X \sim \mathbb P}\left[G(\theta^\star,X)\right] \text{P}_{A^\star_{a}}\right)$ is small. This observation will be instrumental in the next section, where importance sampling techniques will be introduced with the aim of directly reducing this quantity. \qed

\begin{remark}[Sample average approximation]
\label{remark:SAA}
An alternative approach to the stochastic approximation algorithms presented so far is the sample average approximation (SAA)
\begin{align}
  \label{eq:saa}
    \theta_n = \argmin_{\theta \in \Theta} \frac{1}{n} \sum_{i=1}^n F(\theta,X_i).
\end{align}
The SAA procedure is asymptotically optimal as well, enjoying the CLT \eqref{eq:LB:CLT} (see \cite[Theorem~3.3]{shapiro1989asymptotic}). However, for modern large scale problems, online stochastic gradient methods are generally preferred due to their superior computational efficiency. Indeed, equation \eqref{eq:saa} still requires the solution of an optimization problems whose complexity grows with increasing sample size. \hfill $\clubsuit$
\end{remark}

% ------------------------------------------------------------------------------------
% ------------------------------------------------------------------------------------
% ------------------------------------------------------------------------------------

\section{Importance Sampling}
\label{sec:importance-sampling}

As discussed earlier, stochastic gradient methods are based on the definition 
\(
G(\theta, X) \defn \nabla_{\theta}F(\theta,X),
\)
so that $\nabla f(\theta) = \mathbb E_{X \sim \mathbb P}\left[ G(\theta, X) \right]$.
However, such stochastic gradients may have very large variance, leading to practical inefficiency.
This is often the case in optimization problems involving rare events, where the expected loss to be minimized may be determined by extreme events that occur infrequently, but that are associated with very large costs.
With naïve Monte-Carlo simulation, such events are infrequently observed, and hence the associated stochastic gradient estimators display high variance, and consequently poor practical performance.

Importance sampling (IS) is a variance reduction technique \citep{glasserman2004monte} which can be employed in the selection of better stochastic gradients. Such technique is based on the following observation:
\begin{align}
\label{eq:IS}   
    \nabla  f(\theta) = \mathbb E_{X \sim \mathbb P}\left[ G(\theta,X)  \right] = \mathbb E_{X^{(\IS)} \sim \mathbb P_{\IS}}\left[ G_{\IS}(\theta, X^{(\IS)}) := G(\theta,X^{(\IS)}) \frac{\d \mb P}{\d \mb P_{\IS}}(X^{(\IS)}) \right],
\end{align}
where ${\mathrm{d} \mathbb P}/{\mathrm{d} \mathbb P_{\IS}}:\mc X\to\Re_+$ represents the Radon-Nicodym derivative of the original distribution $\mathbb P$ with respect to the importance sampling distribution $\mathbb P_{\IS}$. We remark that this derivative exists as long as $\mathbb P\ll \mathbb P_{\IS}$ \cite{nikodym1930sur}, and unlike \citet{Bardou2009computing} we do not require the two distributions to be absolutely continuous with respect to the Lebesgue measure (i.e., admit a density function  on $\Re^r$).
In the particular case where the support set $\mc X$ of $\mathbb P$ is finite (e.g., in empirical risk minimization problems), the absolute continuity requirement reduces to the condition $\mathbb P_{\IS}(x)>0$ for all $x\in \mc X$. 
In view of Lemma \ref{lemma:kkt-residual} and Equation (\ref{eq:trace-identity}), an IS distribution $\mathbb P_{\IS}$ is good if 
\[
  \trace \left(\text{Var}_{X^{(\IS)} \sim \mathbb P_{\IS}}[\text{P}_{A^\star_{a}} G_{\IS}(\theta^\star,X^{(\IS)})] \right) < \trace \left(\text{Var}_{X \sim \mathbb P}\left[\text{P}_{A^\star_{a}} G(\theta^\star,X)\right]\right).
\]
Let $\mc P_{\IS}$ denote the probability simplex of probability measures on $\mc X$ with respect to which $\mathbb P$ is absolutely continuous. Then, the optimal IS distributon can be identified as
\begin{equation*}
%  \label{eq:opt:IS:char}
  \begin{split}
    \mathbb P_\star \in \arg\min & ~\trace\left(\text{Var}_{X^{(\IS)} \sim \mathbb P_{\IS}}\left[\text{P}_{A^\star_{a}} G(\theta^\star,X^{(\IS)}) \frac{\d \mb P}{\d \mb P_{\IS}}(X^{(\IS)})\right]\right) \\
    \st & ~\,\mb P_{\IS} \in \mc P_{\IS}.
  \end{split}
\end{equation*}
Simple calculations reveal that the optimal IS distribution can be explicitly characterized through
\begin{equation}
\label{eq:opt:IS}
    \mathbb P_\star(\mc E) \defn \textstyle\tfrac{\int_{\mc E}\norm{\text{P}_{A^\star_{a}} G(\theta^\star, x)}^2 \, \mathrm{d} \mathbb P(x)}{\int \norm{\text{P}_{A^\star_{a}} G(\theta^\star, x)}^2 \, \mathrm{d} \mathbb P(x) } 
\end{equation}
for every measurable event set $\mc E\subseteq \mc X$.
However, for an IS distribution $\mathbb{P}_{\IS}$ to be useful, it is crucial that it can be sampled from efficiently, and that the associated likelihood ratio ${\mathrm{d} \mathbb{P}}/{\mathrm{d} \mathbb{P}_{\IS}}$ is known. Clearly, this is not the case for the optimal distribution $\mathbb P_\star$ in Equation~(\ref{eq:opt:IS}), even if the distribution $\mathbb P$ is discrete and supported on a finite number of point, since the active constraints and the optimal solution are not known.
To address this shortcoming, we will therefore consider a restricted family of IS distributions $\mathbb P_\mu$, parametrized by a parameter $\mu$ which lives in a closed and convex set $\mathcal M:=\{\mu \in \mathbb R^m:\, C \mu \leq d\}$, for some technology matrix $C$ and budget vector $d$. We denote the associated likelihood ratio by
\begin{align*}
    \ell(x, \mu)\defn {\mathrm{d} \mathbb P}/{\mathrm{d} \mathbb P_{\mu}}(x),
\end{align*}
and stochastic gradients by $G_{\mu}(\theta, x) \defn \ell(x, \mu) G(\theta, x)$. Moreover, for the IS class to be well behaved, we impose the following assumptions on the function $\ell$, which are similar to those considered in \citep{Bardou2009computing}.

\begin{assumption}[Importance Sampling I]~
\label{assump:IS:class}
\begin{enumerate}[label=(\roman*)]
    \item \label{assump:IS:logconcave} For any $x \in \mathbb R^r$, the function $\mu \to \ell(x,\mu)$ is logarithmically convex.
    \item \label{assump:IS:differntiable} For any $x \in \mathbb R^r$, the function $\mu \to \ell(x,\mu)$ is differentiable. 
\end{enumerate}
\end{assumption}

\noindent We now briefly discuss three important IS classes for which Assumption \ref{assump:IS:class} holds.

\subsection{Exponential Tilting}
\label{sec:example:ECM}
Perhaps the most important IS distribution class is the exponential tilting (ET), defined as
\begin{equation}
\label{eq:exponential:IS}
    \ell_{\ET}(x,\mu) = \exp(-\mu^\top x + \phi(\mu)),
\end{equation}
with $\phi$ being the associated cumulant-generating function defined as $\phi(\mu) = \log \mathbb E_{X \sim \mathbb P}[ \exp(\mu^\top X) ]$.
Note that Assumption \ref{assump:IS:class}\ref{assump:IS:logconcave} is immediately satisfied given that the cumulant-generating function is a convex function.
This family of IS distributions is consistent with a natural exponential family
\[
  \mc P_{\ET} = \set{\mb P_\mu\in \mc P_{\IS}}{\exists \mu\in \mc M, ~\mb P_\mu(\mc E) = \int_{\mc E}  \exp(\mu^\top X - \phi(\mu)) \d \mb P, ~~\forall \mc E\subseteq \mc X}.
\]

In certain cases, the IS distribution $\mathbb P_\mu$ belongs to the same parametric family as $\mathbb P$. This is the case for example when the original density belongs to the exponential family of distributions (e.g., normal, exponential, Poisson, chi-squared, etc.), simplifying random variable generation during Monte Carlo simulations. The exponential tilting has also proven to be fundamental in the context of rare-event simulation and large deviation theory, where often it is the unique efficient simulation distribution choice (see \cite[Chapter~$5.2$]{bucklew2004introduction}). The following lemma states that Assumption \ref{assump:IS:class}\ref{assump:IS:differntiable} is also satisfied. Its proof follows immediately from Lebesgue's dominated convergence, and is therefore omitted.

\begin{lemma}[Exponential tilting]
\label{lemma:ECM:assump:IS}
Let the cumulant-generating function satisfy $\phi(\mu)<\infty$ for all $\mu \in \mathbb R^m$. 
Then the function $\ell_{\ET}$ satisfies Assumption~\ref{assump:IS}\ref{assump:IS:differntiable} with gradient
\begin{align*}
    \nabla_\mu \ell_{\ET}(x,\mu) = (\nabla \phi(\mu)-X)e^{-\mu^\top X + \phi(\mu)},
\end{align*}
where $\nabla_\mu \phi(\mu) = \mathbb E_{X \sim \mathbb P} \left[ X e^{\,\mu^\top X} \right]/\mathbb E_{X \sim \mathbb P} \left[e^{\,\mu^\top X} \right]$.
\end{lemma}

The following example serves as an illustrative case to demonstrate how exponential tilting within an importance sampling scheme can transform an estimator with exponentially large asymptotic variance into one with a uniformly bounded variance. While the setting is intentionally simple, it highlights the potential of such techniques and motivates the more general methodology developed in this paper.

\begin{example}[Normal Quantile Estimation]
\label{example:nqe}
Let $X$ be a standard normal random variable. Its $\alpha$-th quantile $\theta^\star$ can be characterized as the minimum of the stochastic optimization problem
\(
 \min_{\theta\in \mathbb R}\,f(\theta)=\mathbb E_{X\sim \mathcal N(0, 1)}\left[ \alpha \theta + \max\{X - \theta, 0\} \right].
\) 
This follows immediately from the optimality condition $\nabla f(\theta^\star)=0$, which reduces to $\alpha - \mathbb E_{X\sim \mathcal N(0, 1)}\left[ G(\theta, X) = \one{X \geq \theta} \right] = 0$.
We remark that the Hessian at the optimal solution is given as $\nabla^2 f(\theta^\star) = p(\theta^\star)$, where $p(x) = \tfrac{\exp(-x^2/2)}{\sqrt{2\pi}}$. 
Furthermore, we have that $\text{Var}_{X \sim \mathbb P}\left[G(\theta^\star,X)\right] = \alpha(1-\alpha)$. Assume now that $\alpha < 1/2$, and hence $\theta^\star>0$ and $\text{Var}_{X \sim \mathbb P}\left[G(\theta^\star,X)\right]\geq \alpha/2$.
Following equation \eqref{eq:PR-SA:CLT}, the PR-SA iteration scheme based on standard stochastic gradients would achieve the CLT,
$\sqrt{n}( \bar{\theta}_n - \theta^\star ) \overset{d}{\to} \mathcal N \left(0, \sigma^2  \right),$
with asymptotic variance $$\sigma^2 \geq \sqrt{2\pi}\exp(\theta^{\star 2}/2)/(2(\theta^\star + 1/\theta^\star)),$$ which follows from the standard normal tail inequality $\mathbb P[X\geq \theta^\star] \geq \exp(-\theta^{\star 2}/2)/((\theta^\star + 1/\theta^\star)\sqrt{2\pi})$, for $\theta^\star>0$. Clearly, when $\alpha$ is small, the quantile $\theta^\star$ is large and the asymptotic variance of any estimator based on standard stochastic gradients is exponentially large. 

Consider now an IS scheme in which $X^{(\mu)} \sim \mathbb P_\mu$, where the distribution $\mathbb P_\mu=\mathcal N(\mu, 1)$ is associated with $\ell_\text{ET}(x, \mu) ={\mathrm{d} \mathbb P}/{\mathrm{d} \mathbb P_\mu}(x)= \exp(-\mu x + \mu^2/2)$, which is consistent with an exponential tilting with natural parametrization.
Indeed, as the natural sufficient statistic is $S(x)=x$, Equation~(\ref{eq:exponential:IS}) reduces to
\(
  \ell_{\ET}(x,\mu) = \exp(-\mu x + \phi(\mu))
\)
with $\phi(\mu) = \log \mathbb E_{X \sim \mathbb P}[ \exp(\mu^\top X) ] = \mu^2/2$. The variance of the stochastic gradients can be bounded by
\begin{align*}
 \text{Var}_{X^{(\mu)} \sim \mathbb P_\mu} &\left[G_\mu(\theta^\star,X^{(\mu)})\right]  = \mathbb E_{X^{(\mu)} \sim \mathbb P_\mu}\left[G_\mu(\theta^\star,X^{(\mu)})^2\right] - \nabla f(\theta^\star)^2
= \mathbb E_{X \sim \mathbb P}\left[G(\theta^\star,X)^2 \ell_{\text{ET}}(x, \mu)\right]\\
& = \mathbb E_{X \sim \mathbb P}\left[\one{X\geq \theta^\star} \ell_{\text{ET}}(x, \mu)\right]
= \int_{\theta^\star}^\infty \exp(-\mu x + \mu^2/2) \exp(-x^2/2)/\sqrt{2\pi} \; \text{d} x\\
& =  \exp(\mu^2/2) \int_{\theta^\star}^\infty \exp(-\mu x) \tfrac{\exp(-x^2/2)}{\sqrt{2\pi}} \; \text{d} x\\
& \leq \exp(\mu^2/2) \int_{\theta^\star}^\infty \exp(-\mu \theta^\star) \tfrac{\exp(-x^2/2)}{\sqrt{2\pi}} \; \text{d} x\\
& = \exp(\mu^2/2-\mu \theta^\star) \int_{\theta^\star}^\infty \tfrac{\exp(-x^2/2)}{\sqrt{2\pi}} \; \text{d} x \leq \frac{1}{2} \exp(-\theta^{\star 2}) \exp((\theta^\star-\mu)^2/2).
\end{align*}
The first inequality follows from $\mu\geq 0$ and the fact that the exponential is an increasing function.
The final inequality follows from the standard tail inequality
$\mathbb P[X\geq \theta^\star] \leq \exp(-\theta^{\star 2}/2)/2$.
Following equation \eqref{eq:PR-SA:CLT}, the PR-SA iteration scheme based on the IS stochastic gradients would achieve the CLT
\(
\sqrt{n}\left( \bar{\theta}_n - \theta^\star \right) \overset{d}{\to} \mathcal N \left(0, \sigma(\mu)^2  \right)
\)
with
\(
\sigma^2(\mu) \leq  \exp((\theta^\star-\mu)^2/2)/2.
\)
In particular, we remark that the best importance sampler in the considered IS family has a bounded variance independent of $\alpha$, i.e.,
\(
\min_{\mu\geq 0}\,\sigma(\mu)^2  = \sigma(\theta^\star)^2 \leq 1/2.
\) \hfill $\clubsuit$
\end{example}

% ------------------------------------------------------------------------------------
% ------------------------------------------------------------------------------------
% ------------------------------------------------------------------------------------

\subsection{Mean Translation}
\label{sec:example:MT}

Another important IS distribution class is represented by the translation family of a log-concave base distribution.
In this context, the base probability measure $\mb P$ is assumed to admit a density function with respect to the Lebesgue measure $\mb L$ on $\mc X = \Re^r$, which is strictly positive, i.e., $p(x)>0$ for all $x\in \Re^r$. We require the base distribution $\mb P$ to be log-concave, i.e., its density function $p(x) \defn \tfrac{\d \mb P}{\d \mb L}(x)$ is such that $\Delta(x) := -\log(p(x))$ is a convex function.
Then, the associated IS distributions $\mathbb P_\mu$ are characterized by the translated densities $p_\mu(x) \defn \tfrac{\d \mb P_\mu}{\d \mb L}(x) = p(x-\mu)$, and together define the translation family $\mc P_{\MT}\defn \set{\mathbb P_\mu}{\mu\in \mathbb R^r}$. Sampling from $\mathbb P_\mu$  having access to a samples from the base measure $\mb P$ is trivial and hence this family is practical whenever
\begin{equation}
  \label{eq:mt:derivative}
  \ell_{\MT}(x, \mu) \defn \frac{\d \mb P}{\d \mb P_\mu}(x) = \frac{p(x)}{p_\mu(x)} = \frac{\exp(-\Delta(x))}{\exp(-\Delta(x-\mu))}=\exp(-(\Delta(x)-\Delta(x-\mu)))
\end{equation}
can be evaluated efficiently. Such translations of the density function are generally used to place more probability mass in a rare event region, and have been successfully employed in the context of digital communication systems \cite{bucklew2004introduction}.

It is clear from Equation (\ref{eq:mt:derivative}) and the log-concavity of the base probability measure $\mb P$ that the derivative function $\ell_{\MT}$ associated with this class of IS distributions satisfies Assumption \ref{assump:IS:class}\ref{assump:IS:logconcave}.
It remains to verify Assumption \ref{assump:IS:class}\ref{assump:IS:differntiable}. The following well known result is stated without proof.

\begin{lemma}[Mean translation]
\label{lemma:MT:assump:IS}
Let $\nabla p(x)$ exist for all $x\in \Re^r$. Then the function $\ell_{\MT}$ satisfies Assumption~\ref{assump:IS:class}\ref{assump:IS:differntiable} with gradient given as
\[
  \nabla_\mu \ell_{\MT}(x,\mu) = \exp(-(\Delta(x)-\Delta(x-\mu)))\frac{\nabla p(x-\mu)}{p(x-\mu)}.
\]
\end{lemma}

Similarly to Example~\ref{example:nqe}, the following example illustrates how mean translation in importance sampling can reduce exponentially large asymptotic variance to a uniformly bounded level.

\begin{example}[Exponential Quantile Estimation]
Assume that $X$ is instead a standard exponential random variable, whose distribution we denote by $\mathcal E(0,1)$. Its $\alpha$-th quantile $\theta^\star$ can be characterized as the minimum to the stochastic optimization problem
\(
\min_{\theta\in \mathbb R}\,f(\theta)=\mathbb E_{X\sim \mathcal E(0, 1)}\left[ \alpha X + \max(X - \theta, 0) \right].
\)
The Hessian at the optimal solution is given by $\nabla^2 f(\theta^\star) = p(\theta^\star)$, with $p(x) = \tfrac{\exp(-|x|)}{2}$. 
Again, we have that $\text{Var}_{X \sim \mathbb P}\left[G(\theta^\star,X)\right] = \alpha(1-\alpha)$. Assume now that $\alpha < 1/2$, and hence $\theta^\star>0$ and $\text{Var}_{X \sim \mathbb P}\left[G(\theta^\star,X)\right]\geq \alpha/2$.
Following equation \eqref{eq:PR-SA:CLT}, a standard optimal iteration scheme based on standard stochastic gradients would achieve the CLT
\(
\sqrt{n}\left( \bar{\theta}_n - \theta^\star \right) \overset{d}{\to} \mathcal N \left(0, \sigma^2  \right),
\)
with asymptotic variance $$\sigma^2 \geq \frac{1}{2} \exp(-\theta^\star) 4 \exp(2|\theta^\star|)=2\exp(\theta^\star)$$ where we employed the identity $$\mathbb P[X\geq \theta^\star] = \int^\infty_{\theta^\star}  \exp(-x) \; \text{d} x =\exp(-\theta^\star)$$ for $\theta^\star>0$. Clearly, when $\alpha$ tends to zero the quantile $\theta^\star$ tends to infinity and the asymptotic variance of any estimator based on standard stochastic gradients explodes.

Consider now an IS scheme in which $X^{(\mu)} \sim \mathbb P_\mu$ with distribution $\mathbb P_\mu=\mathcal E(\mu, 1)$. Consequently, we have that $\ell_{\MT}(x, \mu) ={\mathrm{d} \mathbb P}/{\mathrm{d} \mathbb P_\mu}(x)= \exp(-(|x|-|x-\mu|))$. The variance of the stochastic gradients can be bounded by
\begin{align*}
\text{Var}_{X^{(\mu)} \sim \mathbb P_\mu}\left[G_\mu(\theta^\star,X^{(\mu)})\right] & = \mathbb E_{X^{(\mu)} \sim \mathbb P_\mu}\left[G_\mu(\theta^\star,X^{(\mu)})^2\right] - \nabla f(\theta^\star)^2
=  \mathbb E_{X \sim \mathbb P}\left[G(\theta^\star,X)^2 \ell(x, \mu)\right]\\
& = \mathbb E_{X \sim \mathbb P}\left[\one{X\geq \theta^\star} \ell(x, \mu)\right]
= \int_{\theta^\star}^\infty \exp(-2|x|+|x-\mu|)/2 \; \text{d} x\\
& = \int_{\theta^\star}^\mu  \exp(-3x+\mu)/2 \; \text{d} x + \int^\infty_\mu\exp(-x-\mu)/2 \; \text{d} x\\
&\leq  \exp(\mu)/2 \int_{\theta^\star}^\infty  \exp(-3x) \; \text{d} x + \exp(-2\mu)/2\\
& = \frac{1}{2}\left(\exp(\mu-3\theta^\star)/3 + \exp(-2\mu)\right).
\end{align*}
Then, following equation \eqref{eq:PR-SA:CLT}, a standard optimal iteration scheme based on standard stochastic gradients would achieve
\(
\sqrt{n}\left( \bar{\theta}_n - \theta^\star \right) \overset{d}{\to} \mathcal N \left(0, \sigma(\mu)^2  \right)
\)
with
\(
\sigma^2(\mu) \leq 2\left(\exp(\mu-\theta^\star)/3 + \exp(2(\theta^\star-\mu))\right).
\)
In particular, we remark that the best importance sampler in the considered IS family has a bounded variance independent of $\alpha$, i.e.,
\(
\min_{\mu\geq 0}\,\sigma(\mu)^2  = \sigma(\theta^\star+\log(6)/3) = \sqrt[3]{6}.
\) \hfill $\clubsuit$
\end{example}

\subsection{Mixture Models}

Finally, consider a setting in which the decision-maker has access to $I\in \mathbb N$ distinct importance samplers. That is, we have
$G_i(\theta, X) = G(\theta, X) \ell_i(X)$ so that, for all $i\in[I]$, we have $\nabla f(\theta) = \mathbb E_{\mathbb P_i}\left[G_i(\theta, X)\right]$ and where we assume that we can sample from each distribution $\mb P_i$ efficiently. Given any $\mu\in\Re^I_+$ such that $\sum_{i=1}^I\mu_i=1$, we can consider the mixture distribution
\(
  \mb P_\mu = \sum_{i=1}^I \mu_i\mathbb P_i
\)
from which we can sample efficiently given access to samples from each of the IS distributions $\mb P_i$, for $i\in [I]$. Such mixture models are considered for instance in \cite[Section 5.2.2]{bucklew2004introduction} and are successfully employed to define efficient important samplers in the context of a large deviation principle. 

We consider in this setting the IS family $\mc P_{\MM} \defn \set{\mb P_\mu}{\exists \mu\in \mc M}$ with IS parameter set $\mathcal M = \tset{\mu\in \mathbb R_+^I}{\sum_{i=1}^I \mu_i=1}$ and associated likelihood ratio
\begin{equation}
  \label{eq:mm:derivative}
  \ell_{\MM}(x, \mu) = \left(\textstyle\sum_{i=1}^I\mu_i \ell_i(x)^{-1}\right)^{-1}.
\end{equation}
It is trivial to verify from equation (\ref{eq:mm:derivative}) that the function $\ell_{\MM}$ associated with this class of IS distributions satisfies Assumption \ref{assump:IS:class}\ref{assump:IS:logconcave}.
It finally remains to verify Assumption \ref{assump:IS:class}\ref{assump:IS:differntiable}.

\begin{lemma}[Mixture models]
\label{lemma:MM:assump:IS}
The function $\ell_{\MM}$ satisfies Assumption~\ref{assump:IS:class}\ref{assump:IS:differntiable} with gradient given as
\[
  \nabla_\mu \ell_{\MM}(x,\mu) =
  -
  \begin{pmatrix}
    \ell_1(x)^{-1}\\
    \vdots\\
    \ell_I(x)^{-1}
  \end{pmatrix}
  \left(\textstyle\sum_{i=1}^I\mu_i \ell_i(x)^{-1}\right)^{-2}.   
\]
\end{lemma}

With this foundation in place, we now explore how to adaptively select importance sampling distributions within stochastic approximation algorithms to minimize asymptotic variance.

% ------------------------------------------------------------------------------------

\section{Adaptive Importance Sampling}
\label{sec:SO:with:IS}

Given only black-box access to the optimization problem~\eqref{eq:so}, the optimizer has access only to samples from $G(\theta, X)$, where $G(\theta, X) \defn \nabla_{\theta} F(\theta, X)$ and $\nabla f(\theta) = \mathbb{E}_{X \sim \mathbb{P}}\!\left[ G(\theta, X) \right]$. Under this information structure, the NDA iteration proposed in~\cite{duchi2021asymptotic} is optimal in the sense discussed in Section~\ref{sec:sa:with:is}. In this paper we go beyond the black-box model and assume that sampling from an IS class is a viable option.

Following the discussion after Lemma~\ref{lemma:kkt-residual}, if we have access to the optimal IS parameter $\mu^\star$ which minimizes $ \trace\left(\text{Var}_{X^{(\mu)} \sim \mathbb P_\mu}\left[\text{P}_{A^\star_{a}} G_\mu(\theta^\star, X^{(\mu)})\right]\right)$ over $\mathcal M$, then an NDA procedure based on the stochastic gradients $G_{\mu^\star}(\theta,X^{(\mu^\star)})$ would output a sequence of iterates which reduce the residual as fast as possible,
\begin{equation*}
%\label{eq:second:moment:limit}
    \lim_{n\to\infty} n \mathbb E \left[ ||\text{P}_{A^\star_{a}} \nabla  f(\bar \theta_n)||_2^2 \right] = \min_{\mu\in \mathcal M} \trace\left(\text{Var}_{X^{(\mu)} \sim \mathbb P_\mu}\left[\text{P}_{A^\star_{a}} G_\mu(\theta^\star, X^{(\mu)})\right]\right).
\end{equation*}
The objective $\trace\left(\text{Var}_{X^{(\mu)} \sim \mathbb P_\mu}\left[\text{P}_{A^\star_{a}} G_\mu(\theta^\star, X^{(\mu)})\right]\right)$ on the right-hand side plays a central role in our analysis and admits a particularly convenient structure. Indeed, using the standard identity
\begin{align*}
  & \text{Var}_{X^{(\mu)} \sim \mathbb P_\mu}\left[\text{P}_{A^\theta_{a}} G_\mu(\theta, X^{(\mu)})\right] 
  = \text{P}_{A^\theta_{a}} \mathbb E_{X \sim \mathbb P}\left[ G(\theta, X) G(\theta, X)^\top  \ell(X, \mu)\right] \text{P}_{A^\theta_{a}} - \text{P}_{A^\theta_{a}} \nabla  f(\theta) \nabla  f(\theta)^\top \text{P}_{A^\theta_{a}},
\end{align*}
and recalling the optimality condition $\text{P}_{A^\star_{a}} \nabla  f(\theta^\star) = 0$ from equation \eqref{eq:kkt-conditions}, we see that this objective function simplifies to an expectation of the form $\mathbb E_{X \sim \mathbb P}\left[V(\theta^\star, \mu, X)\right]$, where 
\begin{align*}
    V(\theta, \mu, X):=\|\text{P}_{A_{a}^\theta}G(\theta,X)\|^2\ell(X,\mu),
\end{align*}
for all $\theta\in \Theta$ and $\mu \in \mathcal M$.
Hence, the optimal IS parameter $\mu^\star$ can be equivalently recovered as the minimum of the following stochastic optimization problem:
\begin{align}
  \label{eq:optimal:IS:2}
  \begin{split}
    \min_{\mu} & ~v(\theta^\star,\mu):=\mathbb E_{X \sim \mathbb P}\left[V(\theta^\star, \mu, X)\right]\\
    \st & ~\mu \in \mathcal M:=\{\mu \in \mathbb R^m:\, C \mu \leq d\}.
  \end{split}
  \tag{IS}
\end{align}
There are two main challenges that arise in tackling this problem:

\medskip
\noindent\textbf{Challenge I.} \emph{The first major challenge is the fact that the optimization problem \eqref{eq:optimal:IS:2} which characterizes the optimal importance sampler is itself a stochastic optimization problem.}

\medskip
Akin to the structural assumptions imposed on problem \eqref{eq:so}, we will impose the following structural assumptions on the \eqref{eq:optimal:IS:2} problem as well.

\begin{assumption}[Importance sampling II]~
\label{assump:IS}
\begin{enumerate}[label=(\roman*)]
    \item \label{ass:IS:unique} Problem~\eqref{eq:optimal:IS:2} admits an unique minimizer $\mu^\star$ with $\nabla^2_\mu v(\theta^\star,\mu)\in \mb S_{++}$.
    
    \item The set $\mathcal M$ is bounded.
    
    \item \label{ass:IS:dominated} For any $\theta \in \Theta$ and $\mu \in \mathcal M$, we have:
    \begin{itemize}
        \item[(iii.1)] \label{ass:IS:dominated:1} $\mathbb E_{X \sim \mathbb P}\left[\,\|G(\theta,X)\|^2\ell(X,\mu)\right] \leq G^2_M < \infty$;

        \item[(iii.2)] \label{ass:IS:dominated:2} $\mathbb E_{X \sim \mathbb P}\left[\,\|\text{P}_{A_{a}^\theta} G(\theta,X)\|^4\|\nabla_\mu \ell(X,\mu)\|^2\right] \leq H^2_M < \infty$.
    \end{itemize}
\end{enumerate}
\end{assumption}

Using the notation introduced in equation \eqref{eq:IS} and below in equation \eqref{eq:gradient:V}, the two conditions in Assumption~\ref{assump:IS}\ref{ass:IS:dominated} are equivalent to $\mathbb E_{X^{(\mu)} \sim \mathbb P_\mu}\left[\,\|G_\mu(\theta,X^{(\mu)})\|^2\right] \leq G^2_M < \infty$ and $\mathbb E_{X \sim \mathbb P}\left[\,\|H(\theta, \mu, X)\|^2\right] \leq H^2_M < \infty$, respectively. Such a second-moment boundedness assumption on the gradients is standard in the analysis of stochastic approximation algorithms; see, e.g., \cite{nemirovski2009robust}. In our setting, this condition is often satisfied due to the compactness of the sets $\Theta$ and $\mathcal M$, for many common choices of the distribution $\mathbb P$. As an example, it can be easily checked that this is the case if $\mathbb P$ is Gaussian and the parameter $\mu$ comes from the exponential tilting IS class.

\begin{lemma}[Convexity and differentiability of $v$]
  \label{lemma:conv:diff:V}
  Let Assumptions~\ref{assump:IS:class}\ref{assump:IS:logconcave}, \ref{assump:IS:class}\ref{assump:IS:differntiable} and \ref{assump:IS}\ref{ass:IS:dominated} be satisfied. Then, for every $\theta$, the function $v(\theta, \mu)\defn \mathbb E_{X \sim \mathbb P}\left[V(\theta, \mu, X)\right]$ is convex and differentiable in $\mu$, with gradient
  \begin{align}
    \label{eq:gradient:V}
    \nabla_\mu v(\theta, \mu)  = \mathbb E_{X \sim \mathbb P}\left[H(\theta, \mu, X):=\|\text{P}_{A_{a}^\theta}G(\theta,X)\|^2\nabla_\mu \ell(X,\mu)\right].
  \end{align}
\end{lemma}
\begin{proof}
  Fix $x \in \mathbb R^r$. Then, from Assumption~\ref{assump:IS:class}\ref{assump:IS:logconcave} we have that $\mu \to \tfrac{\d\mb P}{\d \mb P_\mu}\, (x)$ is log-convex, $\mu \to \log \tfrac{\d\mb P}{\d \mb P_\mu}(x) $ is convex, and therefore via Young's inequality we have that $\mu \to \ell(x, \mu)= \tfrac{\d\mb P}{\d \mb P_\mu}(x)$ is convex. Then, as integration preserves convexity, we also have that $v$ is convex. The differentiability of $v$ and the expression of the gradient follow immediately from Assumptions~\ref{assump:IS:class}\ref{assump:IS:differntiable} and \ref{assump:IS}\ref{ass:IS:dominated}. In particular, the condition $ \mathbb E_{X \sim \mathbb P}\left[\,\|\text{P}_{A_{a}^\theta}G(\theta,X)\|^2\ell(X,\mu)\right] \leq G_M < \infty$ guarantees that Lebesgue's dominated convergence can be applied to exchange expectation and differentiation.
\end{proof}

We may again decompose the technology matrix $C$ and budget vector $d$ into its {active} and {inactive} components as 
\(
C^\star_{a} \mu^\star - d^\star_a=0
\)
and
\(
C^\star_{i} \mu^\star -d^\star_i < 0,
\)
where $\mu^\star$ is the unique optimal solution in problem \eqref{eq:optimal:IS:2}. Similarly to the optimality conditions in Fact~\ref{fact:optimality:conditions} for problem \eqref{eq:so}, the optimal importance sampler parameter $\mu^\star$ is characterized by the optimality conditions
\[
  C^\star_a \mu^\star - d^\star_a=0, \quad \text{P}_{C^\star_a} \nabla {v}(\mu^\star)=0,
\]
with $\text{P}_{C^\star_a}= I - {C^\star_{a}}^\top ({C^\star_{a}} {C^\star_{a}}^\top)^\dagger {C^\star_{a}}$ being the orthogonal projector onto the null space of the active constraints $\{\mu \in \mathbb R^m:\; C^\star_{a} \mu = 0\}$.
In other words, we are looking for a point $\mu^\star$ in the affine subspace associated with our active constraints, for which the gradient is orthogonal to the active constraints; see also Figure \ref{fig:is}.

Lemma \ref{lemma:conv:diff:V} in principle allows us to solve the \eqref{eq:optimal:IS:2} problem using any of the stochastic approximation methods discussed in Section \ref{sec:sa:with:is}.
In particular, an NDA sequence $\{\bar \mu_n\}_{n \in \mathbb N}$ based on the stochastic gradient $H(\theta^\star, \mu, X)$ defined in \eqref{eq:gradient:V} satisfies (under appropriate conditions) the CLT
\begin{equation*}
  \sqrt{n}\left( \bar{\mu}_n - \mu^\star \right) \overset{d}{\to} \mathcal N \left(0, \text{R}^{\dagger}\, \text{Var}_{X \sim \mathbb P}\left[H(\theta^\star, \mu^\star, X)\right] \text{R}^{\dagger} \right),
\end{equation*}
with $\text{R}:=\text{P}_{C^\star_{a}}\nabla^2  v(\theta^\star,\mu^\star) \text{P}_{C^\star_{a}}$.

\medskip
\noindent\textbf{Challenge II.} \emph{The second and perhaps more fundamental challenge is the fact that in order to find the minimizer $\mu^\star$ in problem \eqref{eq:optimal:IS:2}, we require knowledge of the minimizer $\theta^\star$ in problem (\ref{eq:so}).}

\medskip
To address this conundrum, we consider the following joint NDA iteration instead:
\begin{align}
\label{eq:NDA:SA+IS:iteration}
\begin{split}
    \begin{bmatrix}\theta_{n+1} \\  \mu_{n+1}\end{bmatrix} &= \argmin_{(\theta, \mu) \in \Theta\times \mathcal M} \left\{ \inner{\sum_{k=0}^n \alpha_{k+1} \begin{bmatrix} G_{\mu_k}(\theta_k,X_{k+1}^{(\mu_k)})
 \\ H(\theta_k,\mu_k,X_{k+1}) \end{bmatrix}}{\begin{bmatrix}\theta \\ \mu \end{bmatrix}} +\frac{1}{2} \left\| \begin{bmatrix}\theta - \theta_{0} \\ \mu - \mu_{0} \end{bmatrix} \right\|^2\right\}\\
    \bar{\theta}_n &= \frac 1n \sum_{i=0}^{n-1} \theta_{i}, \quad
    \bar{\mu}_n = \frac 1n \sum_{i=0}^{n-1} \mu_{i},
\end{split}
\end{align}
with step size $\alpha_n = \alpha/n^\gamma$, for $\gamma \in (1/2,1)$ and some constant $\alpha>0$. In Section~\ref{sec:conv:analysis} we will study the convergence properties of the suggested iteration procedure \eqref{eq:NDA:SA+IS:iteration} and show that such an adaptive IS scheme guarantees the optimal asymptotic variance.

% ------------------------------------------------------------------------------------

\subsection{Secondary Importance Sampling}
\label{subsec:secondary:IS}

The reader may have noticed that the stochastic optimization problem~\eqref{eq:optimal:IS:2} is of a similar nature to problem~\eqref{eq:so}. In fact, Lemma \ref{lemma:MM:assump:IS} ensures that problem~(\ref{eq:optimal:IS:2}) satisfies Assumptions \ref{assump:SO}\ref{assump:SO:convex} and \ref{assump:SO}\ref{assump:SO:differentiable}, whereas \ref{assump:SO}\ref{assump:SO:bounded} follows from the assumed compactness of our IS parameter set $\mc M$.
Consequently, since we assume in Assumption \ref{assump:IS}\ref{ass:IS:unique} that the optimal IS parameter $\mu^\star$ is a unique minimizer in (\ref{eq:optimal:IS:2}) with $\nabla^2 v(\mu^\star)\in \mb S_{++}$, then problem (\ref{eq:optimal:IS:2}) is precisely of the same nature as problem (\ref{eq:so}).

Given this observation, it is natural to consider a secondary importance sampling procedure. Indeed, observe that the gradient of problem (\ref{eq:optimal:IS:2}) can be alternatively characterized as
\begin{align}
\label{eq:H_nu}
    \nabla_\mu \mathbb E_{X \sim \mathbb P} [V(\theta,\mu,X)] = \mathbb E_{X^{(\nu)} \sim \mathbb P_{\nu}}\left[H_{\nu}(\theta, \mu, X^{(\nu)})\defn\|\text{P}_{A_{a}^\theta} G(\theta,X^{(\nu)})\|^2\nabla_\mu \ell(X^{(\nu)},\mu) \kappa(X^{(\nu)},\nu)\right],
\end{align}
where the second expectation is with respect to a secondary IS distribution $\mathbb P_\nu$, with parameter $\nu \in \mathcal V$, and associated Radon-Nikodym derivative function $\kappa(x,\nu) = \tfrac{\d \mb P}{\d \mb P_\nu}(x)$.
This observation suggests that one can indeed use the following adaptive IS scheme
\begin{align}
\label{eq:NDA:SA+IS:iteration:extension}
\begin{split}
    \begin{bmatrix}\theta_{n+1} \\  \mu_{n+1}\end{bmatrix} &= \argmin_{(\theta, \mu) \in \Theta\times \mathcal M} \left\{ \inner{\sum_{k=0}^n \alpha_{k+1} \begin{bmatrix} G_{\mu_k}(\theta_k,X_{k+1}^{(\mu_k)})
 \\ H_{\nu_k}(\theta_k,\mu_k,X^{(\nu_k)}_{k+1}) \end{bmatrix}}{\begin{bmatrix}\theta \\ \mu \end{bmatrix}} +\frac{1}{2} \left\| \begin{bmatrix}\theta - \theta_{0} \\ \mu - \mu_{0} \end{bmatrix} \right\|^2\right\}\\
    \bar{\theta}_n &= \frac 1n \sum_{i=0}^{n-1} \theta_{i}, \quad
    \bar{\mu}_n = \frac 1n \sum_{i=0}^{n-1} \mu_{i},
\end{split}
\end{align}
where the superscript $(\nu_k)$ in $X_{k+1}^{(\nu_k)}$ denotes the fact that the sample is distributed according to $\mathbb P_{\nu_k}$.
In what follows we discuss three choices for the secondary importance sampling parameters $\nu_k$ sorted by increasing complexity.

\begin{enumerate}[label=(\roman*)]
    \item \label{point:lemaire} First, in particular settings, a natural choice of secondary IS parameters can be available. In the context of exponential tilting  discussed in Section \ref{sec:example:ECM}, with $\mc M=-\mc M$, \citet{lemaire2010unconstrained} suggest using the same exponential tilting class for the secondary IS class together with the simple choice $\nu_k=-\mu_k$, so that
    \begin{align*}
    \begin{split}
        H_{\nu_k}(\theta_k,\mu_k,X^{\nu_k}_{k+1}) = (\nabla \phi(\mu_k) - X^{(-\mu_k)}_{k+1})\|G(\theta_k, X^{(-\mu_k)}_{k+1})\|^2\, e^{\phi(\mu_k) + \phi(-\mu_k)}.
    \end{split}
    \end{align*}
    This choice has been exploited in \cite{lemaire2010unconstrained} to bound and control the growth of the stochastic gradient in order to ensure the almost sure convergence of the adaptive IS procedure even when the set $\mc M$ is not bounded.

    \item \label{point:he} Secondly, as suggested by \cite{he2023adaptive}, we may assume that there exists a readily available good importance sampler in the secondary importance sampling class $\set{\mb P_\nu}{\nu \in \mc V}$ for estimating the gradient $ \nabla_\mu \mathbb E_{X \sim \mathbb P} [V(\theta,\mu,X)]$, for a given $\theta$ and $\mu$. That is, we may assume to have access to a mapping $I$ so that $\nu = I(\theta, \mu)$ is associated with an asymptotic variance $\text{Var}_{X^{(\nu)} \sim \mathbb P_\nu}\left[H_{\nu}(\theta, \mu, X^{(\nu)})\right]$ that is sufficiently small. In certain cases, large deviations theory can be used to derive such a mapping.

    \item Clearly, we may apply the same reasoning developed thus far to advocate for choosing $\nu_k = \nu^\star$, where $\nu^\star$ minimizes the trace of the asymptotic variance of the secondary importance sampler, i.e.,
    \begin{align*}
        \nu^\star = \argmin_{\nu \in \mathcal V}\; \text{Tr}\left(\text{Var}_{X^{(\nu)} \sim \mathbb P_\nu}\left[\|G(\theta^\star,X)\|^2\nabla_\mu \ell(X^{(\nu)},\mu^\star) \kappa(X^{(\nu)},\nu)\right]\right).
    \end{align*}
    However, this again leads to the same conundrum previously encountered: identifying $\nu^\star$ requires knowledge of both $\theta^\star$ and $\mu^\star$, which are unavailable a priori. As such, one may recursively apply the approach developed in this paper as many times as necessary, ultimately terminating with one of the heuristic procedures described in point~\ref{point:lemaire} or~\ref{point:he}.
\end{enumerate}

\noindent Secondary importance sampling lies beyond the scope of this paper. While we do not pursue this direction further, the theoretical results developed in the next section extend to all three strategies discussed above.

% ------------------------------------------------------------------------------------
% ------------------------------------------------------------------------------------
% ------------------------------------------------------------------------------------

\section{Convergence Analysis}
\label{sec:conv:analysis}

In what follows, we prove that the iterates $({\theta}_n, {\mu}_n)$ defined in~\eqref{eq:NDA:SA+IS:iteration} satisfy the almost sure convergence $({\theta}_n, {\mu}_n) \overset{\text{a.s.}}{\to} (\theta^\star,\mu^\star)$, where $\theta^\star$ is the solution to problem~\eqref{eq:so} and $\mu^\star$ is the solution to problem~\eqref{eq:optimal:IS:2}. Moreover, we establish that the averaged iterates $\bar{\theta}_n$ satisfy the central limit theorem (CLT)
\begin{align}
\label{eq:CLT1:conv}
    \sqrt{n}\left( \bar{\theta}_n - \theta^\star \right) \overset{d}{\to} \mathcal N \left(0, \text{Q}^{\dagger}\, \text{Var}_{X^{(\mu^\star)} \sim \mathbb P_{\mu^\star}}\left[G_{\mu^\star}(\theta^\star, X^{(\mu^\star)})\right] \text{Q}^{\dagger} \right),
\end{align}
with $\text{Q}:=\text{P}_{A^\star_{a}}\nabla^2  f(\theta^\star) \text{P}_{A^\star_{a}}$. By Lemma~\ref{lemma:kkt-residual}, this further implies the CLT
\begin{align}
\label{eq:CLT2:conv}
    \sqrt{n}\, \text{P}_{A^\star_{a}} \nabla f(\bar{\theta}_n)  \overset{d}{\to} \mathcal N \left(0, \text{Var}_{X^{(\mu^\star)} \sim \mathbb P_{\mu^\star}}\left[\text{P}_{A^\star_{a}} G_{\mu^\star}(\theta^\star, X^{(\mu^\star)})\right] \right),
\end{align}
with $\mu^\star = \argmin_{\mu \in \mathcal M} \trace\left(\text{Var}_{X^{(\mu)} \sim \mathbb P_\mu}\left[\text{P}_{A^\star_{a}} G_\mu(\theta^\star, X^{(\mu)})\right]\right)$. For ease of notation, throughout this section we further simplify the notation in \eqref{eq:NDA:SA+IS:iteration} to
\begin{align*}
%\label{eq:simplified:SOwithIS}
    \begin{bmatrix}\theta_{n+1} \\ \mu_{n+1} \end{bmatrix} = \argmin_{(\theta, \mu) \in \Theta\times \mathcal M} \left\{ \inner{\sum_{k=0}^n \alpha_{k+1} \begin{bmatrix}G_k \\ H_k \end{bmatrix}}{\begin{bmatrix}\theta \\ \mu \end{bmatrix}} +\frac{1}{2} \left\| \begin{bmatrix}\theta \\ \mu \end{bmatrix} \right\|^2\right\},
\end{align*}
where, without loss of generality, we consider $\theta_0 = 0$ and $\mu_0 = 0$. Moreover, whenever clear from the context, we will drop the subscripts when defining expectations and variances (i.e., we write $\mathbb E[\cdot]$ instead of $\mathbb E_{X \sim \mathbb P}[\cdot]$). Finally, in this section we let $p_1$ and $p_2$ denote the dimensions of the vectors $b_a^\star$ and $d_a^\star$, respectively.

% ------------------------------------------------------------------------------------

\subsection{Almost Sure Convergence}

We start by studying the almost sure convergence of the sequence $\{(\theta_n,\mu_n)\}_{n \in \mathbb N}$. For this, we require the following regularity assumptions.

\begin{assumption}[Regularity assumptions~I]~
\label{assump:as:conv}
\begin{enumerate}[label=(\roman*)]
    \item \label{assump:as:conv:1} There exists $c_1>0$ such that for all $\theta\in \Theta$ and $\mu \in \mathcal M$:
    \begin{itemize}
        \item[(i.1)] \label{assump:as:conv:1.1} ${f}(\theta) - {f}(\theta^\star) \geq c_1 \|\theta - \theta^\star\|^2$;

        \item[(i.2)] \label{assump:as:conv:1.2} $v(\theta^\star,\mu) - v(\theta^\star,\mu^\star) \geq c_1 \|\mu - \mu^\star\|^2$.
    \end{itemize}

    \item \label{assump:as:conv:2} There exists $c_2, c_3 < \infty$ such that for all $\theta\in \Theta$ and $\mu \in \mathcal M$:
    \begin{itemize}
        \item[(ii.1)] \label{assump:as:conv:2.1} $\|\nabla {f}(\theta) - \nabla {f}(\theta^\star)\| \leq c_2 \|\theta-\theta^\star\|$;

        \item[(ii.2)] \label{assump:as:conv:2.2} $\|\nabla_\mu v(\theta^\star,\mu) - \nabla_\mu v(\theta^\star,\mu^\star)\| \leq c_2 \|\mu-\mu^\star\|$;
    \end{itemize}
  \item \label{assump:as:conv:3} $\|\nabla_\mu v(\theta,\mu) - \nabla_\mu v(\theta^\star,\mu)\| \leq c_3 \|\theta-\theta^\star\|^2$ for all $\mu \in \mathcal M$ and all $\theta \in \Theta$ satisfying $A_a^\star \theta = b_a^\star$,  $A_i^\star \theta < b_i^\star$.
    \item \label{assump:as:conv:4} $-\nabla {f}(\theta^\star) = A_a^{\star \top} \lambda$, for some $\lambda \in \mathbb R_{++}^{p_1}$.
\end{enumerate}
\end{assumption}

Assumptions~\ref{assump:as:conv}\ref{assump:as:conv:1}-\ref{assump:as:conv:3} are natural extensions of the standard gradient regularity conditions which allow us to prove the convergence of the PR-SA and NDA methods. 
Moreover, Assumption~\ref{assump:as:conv}\ref{assump:as:conv:4} will allow us to prove that the sequence of iterates $\{\theta_n\}_{n \in \mathbb N}$ identifies the active constraints in \eqref{eq:so} in finite time, i.e., there exists some (random) finite $N$ such that $A_{a}^\star \theta_n = b_a^\star$ and $A_{i}^\star \theta_n < b_i^\star$ for all $n \geq N$ (see Lemma~\ref{lemma:as:conv:theta}). Although this does not have a direct impact on the proof of Theorem~\ref{thm:as:conv}, where we will explicitly use only Assumptions~\ref{assump:as:conv}\ref{assump:as:conv:1}-\ref{assump:as:conv:3}, the active constraints identification ensured by Assumption~\ref{assump:as:conv}\ref{assump:as:conv:4} is fundamental in practice to prove that Assumption~\ref{assump:as:conv}\ref{assump:as:conv:3} is satisfied. For more detail on this, see Remark~\ref{remark:as:conv:identification:active} below.
We are now ready to state the first main result of this paper.

\begin{theorem}[Almost sure convergence]
\label{thm:as:conv}
Let Assumptions~\ref{assump:SO}, \ref{assump:IS} and \ref{assump:as:conv} be satisfied. Then,
\begin{align}
\label{eq:as:conv:theta:mu}
    \begin{bmatrix}\theta_n \\ \mu_n \end{bmatrix} \overset{\text{a.s.}}{\to} \begin{bmatrix}\theta^\star \\ \mu^\star \end{bmatrix},
\end{align}
where $\theta^\star$ is the optimal solution in \eqref{eq:so}, and $\mu^\star$ is the optimal IS parameter in \eqref{eq:optimal:IS:2}.
\end{theorem}

\begin{proof}
The proof builds upon Lemma~\ref{lemma:robbins:siegmund} with $R_{n+1}$ defined as
\begin{align}
\label{eq:R_{n+1}:expression}
    R_{n+1} := \inner{\sum_{k=0}^n \alpha_{k+1} \begin{bmatrix}G_k \\ H_k \end{bmatrix} + \begin{bmatrix}\theta_{n+1} \\ \mu_{n+1}\end{bmatrix}}{\begin{bmatrix}\theta^\star-\theta_{n+1} \\ \mu^\star-\mu_{n+1}\end{bmatrix}} + \frac{1}{2}\left\| \begin{bmatrix}\theta_{n+1}-\theta^\star \\ \mu_{n+1}-\mu^\star \end{bmatrix} \right\|^2.
\end{align}
Since the iterate $(\theta_{n+1},\mu_{n+1})$ is the optimal solution in the optimization problem \eqref{eq:NDA:SA+IS:iteration}, its first-order optimality condition guarantees that
\begin{align*}
    \inner{\sum_{k=0}^n \alpha_{k+1} \begin{bmatrix}G_k \\ H_k \end{bmatrix} + \begin{bmatrix}\theta_{n+1} \\ \mu_{n+1}\end{bmatrix}}{\begin{bmatrix}\theta-\theta_{n+1} \\ \mu-\mu_{n+1}\end{bmatrix}} \geq 0,
\end{align*}
for all $\theta\in \Theta$ and $\mu \in \mathcal M$.
In particular, this holds true for $\theta^\star \in \Theta$ and $\mu^\star \in \mathcal M$, showing that $R_{n+1} \geq 0$.
Standard algebraic manipulations show that $R_{n+1}$ can be rewritten as
\begin{align*}
    R_{n+1} &= \inner{\sum_{k=0}^n \alpha_{k+1} \begin{bmatrix}G_k \\ H_k \end{bmatrix}}{\begin{bmatrix}\theta^\star \\ \mu^\star \end{bmatrix}} + \frac{1}{2}\left\|\begin{bmatrix}\theta^\star \\ \mu^\star \end{bmatrix}\right\|^2 + \inner{-\sum_{k=0}^n \alpha_{k+1} \begin{bmatrix}G_k \\ H_k \end{bmatrix}}{\begin{bmatrix}\theta_{n+1} \\ \mu_{n+1}\end{bmatrix}} - \frac{1}{2}\left\|\begin{bmatrix}\theta_{n+1} \\ \mu_{n+1}\end{bmatrix}\right\|^2\\
    &= \inner{\sum_{k=0}^n \alpha_{k+1} \begin{bmatrix}G_k \\ H_k \end{bmatrix}}{\begin{bmatrix}\theta^\star \\ \mu^\star \end{bmatrix}} + \frac{1}{2}\left\|\begin{bmatrix}\theta^\star \\ \mu^\star \end{bmatrix}\right\|^2 + \max_{\substack{\theta \in \Theta \\ \mu \in \mathcal M}} \left\{ \inner{-\sum_{k=0}^n \alpha_{k+1} \begin{bmatrix}G_k \\ H_k \end{bmatrix}}{\begin{bmatrix}\theta \\ \mu\end{bmatrix}} - \frac{1}{2}\left\|\begin{bmatrix}\theta \\ \mu\end{bmatrix}\right\|^2 \right\}.
\end{align*}
Since the objective function in the above maximization is separable in $\theta$ and $\mu$, we have that the joint maximization over $\theta\in \Theta$ and $\mu \in \mathcal M$ in the previous equation  is equivalent to the sum
\begin{align*}
    \max_{\theta \in \mathbb R^s} \left\{ \inner{-\sum_{k=0}^n \alpha_{k+1} G_k}{\theta} - \frac{1}{2}\|\theta\|^2 - \delta_{\Theta}(\theta)\right\} + \max_{\mu \in \mathbb R^m} \left\{ \inner{-\sum_{k=0}^n \alpha_{k+1} H_k}{\mu} - \frac{1}{2}\|\mu\|^2 - \delta_{\mathcal M}(\mu)\right\},
\end{align*}
with $\delta_\Theta, \delta_{\mathcal M}$ being the indicator functions of the sets $\Theta$ and $\mathcal M$, respectively. With the aim of upper-bounding $R_{n+1}$, we proceed by finding upper bounds on these two maximization problems.

\medskip

\emph{Term 1:} We start with the maximization over $\theta$. Defining $\ell_1(\theta):= \frac{1}{2}\|\theta\|^2 + \delta_{\Theta}(\theta)$, we have that the maximum over $\theta$ is precisely the convex conjugate $\ell_1^*(-\sum_{k=0}^n \alpha_{k+1} G_k)$. As $\min_{\theta \in \mathbb R^s} \{ \inner{\sum_{k=0}^n \alpha_{k+1} G_k}{\theta} + \frac{1}{2}\|\theta\|^2 + \delta_{\Theta}(\theta)\}$ is the Moreau envelope of $\sum_{k=0}^n \alpha_{k+1} G_k + \delta_{\Theta}(\theta)$ evaluated at zero, we have that the gradient $\nabla \ell_1^*(-\sum_{k=0}^n \alpha_{k+1} G_k)$ is equal to proximal map evaluated at zero \cite[Theorem~2.26]{rockafellar2009variational}, i.e.,
\begin{align*}
    \nabla \ell_1^*\left(-\sum_{k=0}^n \alpha_{k+1} G_k\right) = \argmax_{\theta \in \mathbb R^s} \left\{ \inner{\sum_{k=0}^n \alpha_{k+1} G_k}{\theta} + \frac{1}{2}\|\theta\|^2 + \delta_{\Theta}(\theta) \right\}=\theta_{n+1}.
\end{align*}
Moreover, since $\ell_1$ is $1$-strongly convex, we have that $\ell_1^*$ is $1$-smooth. Therefore, we can upper-bound $\ell_1^*(-\sum_{k=0}^n \alpha_{k+1} G_k)$ as
\begin{align*}
    \ell_1^*\left( -\sum_{k=0}^n \alpha_{k+1} G_k \right) \leq \ell_1^*\left( -\sum_{k=0}^{n-1} \alpha_{k+1} G_k \right) - \alpha_{n+1} \inner{G_n}{\theta_n} + \frac{1}{2} \|\alpha_{n+1} G_n\|^2.
\end{align*}

\medskip

\emph{Term 2:} The maximization over $\mu$ can be dealt with similarly. Defining $\ell_2(\mu):= \frac{1}{2}\|\mu\|^2 + \delta_{\mathcal M}(\mu)$, we can upper-bound $\ell_2^*(-\sum_{k=0}^n \alpha_{k+1} H_k)$ as
\begin{align*}
    \ell_2^*\left( -\sum_{k=0}^n \alpha_{k+1} H_k \right) \leq \ell_2^*\left( -\sum_{k=0}^{n-1} \alpha_{k+1} H_k \right) - \alpha_{n+1} \inner{H_n}{\mu_n} + \frac{1}{2} \|\alpha_{n+1} H_n\|^2.
\end{align*}
Introducing these two upper bounds into the expression of $R_{n+1}$, after few algebraic manipulations we obtain
\begin{align*}
    R_{n+1} \leq R_n - \alpha_{n+1}\inner{\begin{bmatrix}G_n \\ H_n \end{bmatrix}}{\begin{bmatrix}\theta_{n}-\theta^\star \\ \mu_{n}-\mu^\star \end{bmatrix}} + \frac{\alpha_{n+1}^2}{2}\left\|\begin{bmatrix}G_n \\ H_n \end{bmatrix}\right\|^2.
\end{align*}
Since $R_n$ is adapted to the filtration $\mathcal F_n = \sigma(X_k^{(\mu_{k-1})}, X_k |\, k \leq n)$, we can take the conditional expectation $\mathbb E[\cdot | \mathcal F_n]$ on both sides and obtain
\begin{align}
\label{eq:R_n:RobbinsSiegmund:theta:mu}
    \mathbb E[R_{n+1} | \mathcal F_n] \leq R_n - \alpha_{n+1}\inner{\begin{bmatrix}\nabla {f}(\theta_n) \\ \nabla_\mu v(\theta_n,\mu_n)\end{bmatrix}}{\begin{bmatrix}\theta_{n}-\theta^\star \\ \mu_{n}-\mu^\star \end{bmatrix}} + \frac{\alpha_{n+1}^2}{2} \begin{bmatrix} \mathbb E [\|G_n\|^2| \mathcal F_n ] \\ \mathbb E[\|H_n\|^2| \mathcal F_n] \end{bmatrix},
\end{align}
where we have used the facts
\begin{align*}
    \mathbb E[G_n | \mathcal F_n] &= \mathbb E_{X^{(\mu_n)} \sim \mathbb P_{\mu_n}}\left[ G(\theta_n, X^{(\mu_n)})  \ell(X^{(\mu_n)},\mu_n) \right] = \nabla {f}(\theta_n),\\
    \mathbb E[H_n | \mathcal F_n] &= \mathbb E_{X \sim \mathbb P}\left[\,\|\text{P}_{A_{a}^{\theta_n}} G(\theta_n,X)\|^2\nabla_\mu \ell(X,\mu_n)\right]=\nabla_\mu v(\theta_n,\mu_n).
\end{align*}
We will now show that inequality \eqref{eq:R_n:RobbinsSiegmund:theta:mu} can be brought into the form required by Lemma~\ref{lemma:robbins:siegmund}. For this, we start by rewriting the second term on the right-hand side in \eqref{eq:R_n:RobbinsSiegmund:theta:mu} as
\begin{align*}
    -\alpha_{n+1}\inner{\begin{bmatrix}\nabla {f}(\theta_n) \\ \nabla_\mu v(\theta^\star,\mu_n)\end{bmatrix}}{\begin{bmatrix}\theta_{n}-\theta^\star \\ \mu_{n}-\mu^\star \end{bmatrix}} - \alpha_{n+1}\inner{ \nabla_\mu v(\theta_n,\mu_n)-\nabla_\mu v(\theta^\star,\mu_n)}{ \mu_{n}-\mu^\star},
\end{align*}
which we then upper-bound using the Cauchy-Schwarz inequality and the boundedness of $\mathcal M$ by
\begin{align*}
    -\alpha_{n+1}\inner{\begin{bmatrix}\nabla {f}(\theta_n) \\ \nabla_\mu v(\theta^\star,\mu_n)\end{bmatrix}}{\begin{bmatrix}\theta_{n}-\theta^\star \\ \mu_{n}-\mu^\star \end{bmatrix}} + \diam(\mc M) \alpha_{n+1}\left\|{\nabla_\mu v(\theta_n,\mu_n)-\nabla_\mu v(\theta^\star,\mu_n)}\right\|,
\end{align*}
with $\diam(\mc M) \defn \max_{\mu, \mu'\in \mc M}\norm{\mu-\mu'}<\infty$.
Observe now that the term $\left\|{\nabla_\mu v(\theta_n,\mu_n)-\nabla_\mu v(\theta^\star,\mu_n)}\right\|$ can be in general upper bounded as
\begin{align*}
    \left\|{\nabla_\mu v(\theta_n,\mu_n)-\nabla_\mu v(\theta^\star,\mu_n)}\right\|
    &= \left\|\mathbb E_{X \sim \mathbb P} \left[ | \|\text{P}_{A_{a}^{\theta_n}}G(\theta_n,X)\|^2 - \|\text{P}_{A^\star_{a}}G(\theta^\star,X)\|^2 \nabla_\mu \ell(X,\mu_n) \right] \right\| \\
    &\leq \mathbb E_{X \sim \mathbb P} \left[ \left( \|\text{P}_{A_{a}^{\theta_n}}G(\theta_n,X)\|^2 + \|\text{P}_{A^\star_{a}}G(\theta^\star,X)\|^2 \right) \|\nabla_\mu \ell(X,\mu_n)\| \right] \\
    &\leq \mathbb E_{X \sim \mathbb P} \left[ \left( \|\text{P}_{A_{a}^{\theta_n}}G(\theta_n,X)\|^4 + \|\text{P}_{A^\star_{a}}G(\theta^\star,X)\|^4 \right) \|\nabla_\mu \ell(X,\mu_n)\|^2 \right]\\
    &\leq 2H_M^2.
\end{align*}
The first equality is a consequence of the expressions of $\nabla_\mu v(\theta_n,\mu_n)$ and $\nabla_\mu v(\theta^\star,\mu_n)$ (recall equation (\ref{eq:gradient:V})), whereas the first inequality is trivial. The second inequality follows from H\"older's inequality, which allows us to use Assumption \ref{assump:IS}\ref{ass:IS:dominated} to recover the final bound. 

Moreover, on the event $\mc A := \{A^\star_a\theta_n = b^\star_a, \;A_i^\star\theta_n<b_i\}$, Assumption \ref{assump:as:conv}\ref{assump:as:conv:3} guarantees that the upper bound $\|\nabla_\mu v(\theta_n,\mu_n) - \nabla_\mu v(\theta^\star,\mu_n)\| \leq c_3 \|\theta_n-\theta^\star\|^2$ holds. Putting everything together, the second term in \eqref{eq:R_n:RobbinsSiegmund:theta:mu} is upper bounded by
\begin{align*}
    -\alpha_{n+1}\inner{\begin{bmatrix}\nabla {f}(\theta_n) \\ \nabla_\mu v(\theta^\star,\mu_n)\end{bmatrix}}{\begin{bmatrix}\theta_{n}-\theta^\star \\ \mu_{n}-\mu^\star \end{bmatrix}} + \diam(\mc M) \alpha_{n+1}\left( 2H_M^2 \one{\mc A^c}  + c_3 \|\theta_n-\theta^\star\|^2\one{\mc A}  \right),
\end{align*}
where $\mc A^c$ denotes the complement of $\mc A$. 
Introducing this into \eqref{eq:R_n:RobbinsSiegmund:theta:mu}, and using the bounds in Assumption~\ref{assump:IS}\ref{ass:IS:dominated} for the last term in \eqref{eq:R_n:RobbinsSiegmund:theta:mu}, we obtain
\begin{align}
\label{eq:R_n:RobbinsSiegmund:theta:mu:1}
\begin{split}
  \mathbb E[R_{n+1} | \mathcal F_n] \leq & R_n -\alpha_{n+1}\inner{\begin{bmatrix}\nabla {f}(\theta_n) \\ \nabla_\mu v(\theta^\star,\mu_n)\end{bmatrix}}{\begin{bmatrix}\theta_{n}-\theta^\star \\ \mu_{n}-\mu^\star \end{bmatrix}} + \frac{\sqrt{G_M^4+H_M^4}}{2} \;\alpha_{n+1}^2 \\[0.5em]
  & \hspace{8em}+ \diam(\mc M) \alpha_{n+1}\left( 2H_M^2 \one{\mc A^c}  + c_3 \|\theta_n-\theta^\star\|^2 \right).
\end{split}
\end{align}
Now, notice that
\begin{align}
\label{eq:R_n:C_n}
    \alpha_{n+1}\inner{\begin{bmatrix}\nabla {f}(\theta_n) \\ \nabla_\mu v(\theta^\star,\mu_n)\end{bmatrix}}{\begin{bmatrix}\theta_{n}-\theta^\star \\ \mu_{n}-\mu^\star \end{bmatrix}} \geq 0.
\end{align}
This follows automatically from the first-order optimality conditions in \eqref{eq:so} and \eqref{eq:optimal:IS:2}, which guarantee that $\inner{\nabla {f}(\theta_n)}{\theta_{n}-\theta^\star} \geq 0$ and $\inner{\nabla_\mu v(\theta^\star,\mu_n)}{\mu_{n}-\mu^\star} \geq 0$. Moreover, letting $c_4:={\sqrt{(G_M^4+H_m^4)}}/2$ for ease of notation, we have that 
\begin{align}
\label{eq:R_n:B_n}
\begin{split}
   \sum_{n=0}^\infty & c_4 \alpha_{n+1}^2 + \diam(\mc M) \alpha_{n+1}\left( 2H_M^2 \one{\mc A^c}  + c_3 \|\theta_n-\theta^\star\|^2  \right) \\
  &= \sum_{n=0}^\infty c_4 \alpha_{n+1}^2  + \diam(\mc M) 2H_M^2\sum_{n=0}^N  \alpha_{n+1}  + \diam(\mc M) c_3 \sum_{n=0}^\infty \alpha_{n+1} \|\theta_n-\theta^\star\|^2  \\
  &\leq \sum_{n=0}^\infty c_4 \alpha_{n+1}^2  + \diam(\mc M) 2H_M^2(N+\sum_{n=0}^\infty  \alpha^2_{n+1})  + \diam(\mc M) c_3 \sum_{n=0}^\infty \alpha_{n+1} \|\theta_n-\theta^\star\|^2 < \infty
\end{split}
\end{align}
almost surely. The first equality follows from assertion \ref{lemma:as:conv:theta:5} in Lemma~\ref{lemma:as:conv:theta} (with $N$ defined there). The first inequality follows from the two bounds $\alpha_{n+1} \leq 1+\alpha_{n+1}^2$ and $\sum_{n=1}^N  \alpha^2_{n+1} \leq \sum_{n=1}^\infty  \alpha^2_{n+1})$ (as the step-sizes $\alpha_n$ are positive). The final inequality follows from the step-size condition $\sum_{n=1}^\infty\alpha_{n}^2 < \infty$, from assertion \ref{lemma:as:conv:theta:1} in Lemma~\ref{lemma:as:conv:theta}, and from the fact that $N$ is almost surely finite by assertion \ref{lemma:as:conv:theta:5} in Lemma~\ref{lemma:as:conv:theta}.
Now notice that \eqref{eq:R_n:RobbinsSiegmund:theta:mu:1} is exactly as in Lemma~\ref{lemma:robbins:siegmund}, with $A_n = 0$, $B_n$ as in \eqref{eq:R_n:B_n}, and $C_n$ as in \eqref{eq:R_n:C_n}. Therefore, there is a random variable $R_\infty < \infty$ such that $R_n \overset{\text{a.s.}}{\to} R_\infty$, and with probability one we have that
\begin{align}
\label{eq:sum:C_n}
    \sum_{n=0}^\infty \alpha_{n+1}\inner{\begin{bmatrix}\nabla {f}(\theta_n) \\ \nabla_\mu v(\theta^\star,\mu_n)\end{bmatrix}}{\begin{bmatrix}\theta_{n}-\theta^\star \\ \mu_{n}-\mu^\star \end{bmatrix}} < \infty.
\end{align}
In what follows, we will prove that $R_\infty = 0$, which will guarantee the desired convergence \eqref{eq:as:conv:theta:mu}. From \eqref{eq:sum:C_n} and Assumption~\ref{assump:as:conv}\ref{assump:as:conv:1}, we have that

\begin{align}
\label{eq:sum:alpha:theta^2:mu^2}
\begin{split}
    \sum_{n=0}^\infty \alpha_{n+1} \left\|\begin{bmatrix}\theta_{n}-\theta^\star \\ \mu_{n}-\mu^\star \end{bmatrix} \right\|^2 &\leq 
    \frac{1}{c_1}\sum_{n=0}^\infty \alpha_{n+1} \left({f}(\theta_n) - {f}(\theta^\star)\right) + \frac{1}{c_1}\sum_{n=0}^\infty \alpha_{n+1} \left(v(\theta^\star,\mu_n) - v(\theta^\star,\mu^\star)\right)
    \\ &\leq \frac{1}{c_1} \sum_{n=0}^\infty \alpha_{n+1}\inner{\nabla {f}(\theta_n)}{\theta_{n}-\theta^\star} + \frac{1}{c_1} \sum_{n=0}^\infty \alpha_{n+1}\inner{\nabla_\mu v(\theta^\star,\mu_n)}{\mu_{n}-\mu^\star} \\ &< \infty,
    \end{split}
\end{align}
where the second inequality follows from the convexity of $f(\cdot)$ and $v(\theta^\star,\cdot)$. 

We now define $b_{n+1} := \sum_{k=0}^n \alpha_{k+1}$. Since $\alpha_n = \alpha/n^\gamma$, for $\gamma \in (1/2,1)$, it can be shown that $\sum_{k=0}^\infty \left(\alpha_{k+1}/b_{k+1}\right) = \infty$. Moreover, from \eqref{eq:sum:alpha:theta^2:mu^2}, we know that 
\begin{align*}
    \sum_{n=0}^\infty \frac{\alpha_{n+1}}{b_{n+1}} \left(b_{n+1}\left\|\begin{bmatrix}\theta_{n}-\theta^\star \\ \mu_{n}-\mu^\star \end{bmatrix} \right\|^2\right) < \infty.
\end{align*}
Therefore, there exists a subsequence $\{(\theta_{n_i}, \mu_{n_i})\}_{i \in \mathbb N}$ for which, with probability one,
\begin{align}
\label{eq:conv:on:subsequence}
    \lim_{i \to \infty} b_{n_i+1}\left\|\begin{bmatrix}\theta_{n_i}-\theta^\star \\ \mu_{n_i}-\mu^\star \end{bmatrix} \right\|^2 = 0.
\end{align}
We are now ready to prove that $R_\infty = 0$. We start by bounding $R_{n+1}$ (defined in \eqref{eq:R_{n+1}:expression}) as
\begin{align}
\label{eq:bounds:R_n+1}
    R_{n+1} \leq \inner{\sum_{k=0}^n \alpha_{k+1} \begin{bmatrix}G_k \\ H_k \end{bmatrix}}{\begin{bmatrix}\theta^\star-\theta_{n+1} \\ \mu^\star-\mu_{n+1}\end{bmatrix}}
    + \left\|{\begin{bmatrix}\theta_{n+1} \\ \mu_{n+1}\end{bmatrix}}\right\|\left\|{\begin{bmatrix}\theta^\star-\theta_{n+1} \\ \mu^\star-\mu_{n+1}\end{bmatrix}}\right\|
    + \frac{1}{2}\left\| \begin{bmatrix}\theta_{n+1}-\theta^\star \\ \mu_{n+1}-\mu^\star \end{bmatrix} \right\|^2.
\end{align}
By restricting our attention to the subsequence $\{(\theta_{n_i}, \mu_{n_i})\}_{i \in \mathbb N}$, and by using \eqref{eq:conv:on:subsequence} and the compactness of $\Theta$ and $\mathcal M$, we have that
\begin{align}
\label{eq:bounds:R_n+1:last:two:terms}
    0\leq \left\|{\begin{bmatrix}\theta_{n_i+1} \\ \mu_{n_i+1}\end{bmatrix}}\right\|\left\|{\begin{bmatrix}\theta^\star-\theta_{n_i+1} \\ \mu^\star-\mu_{n_i+1}\end{bmatrix}}\right\|
    + \frac{1}{2}\left\| \begin{bmatrix}\theta_{n_i+1}-\theta^\star \\ \mu_{n_i+1}-\mu^\star \end{bmatrix} \right\|^2 \overset{\text{a.s.}}{\to} 0.
\end{align}
We now focus on the first term on the right-hand side in \eqref{eq:bounds:R_n+1}, which we rewrite as
\begin{align*}
    \inner{\sum_{k=0}^n \alpha_{k+1} \left(\begin{bmatrix}G_k - \nabla {f}(\theta_k) \\ H_k - \nabla_\mu v(\theta_k,\mu_k)\end{bmatrix} + \begin{bmatrix}\nabla {f}(\theta_k) - \nabla {f}(\theta^\star) \\ \nabla_\mu v(\theta_k,\mu_k) - \nabla_\mu v(\theta^\star,\mu^\star)\end{bmatrix} + \begin{bmatrix}\nabla{f}(\theta^\star) \\ \nabla_\mu v(\theta^\star,\mu^\star)\end{bmatrix}\right)}{\begin{bmatrix}\theta^\star-\theta_{n+1} \\ \mu^\star-\mu_{n+1}\end{bmatrix}}.
\end{align*}
By restricting our attention to the subsequence $\{(\theta_{n_i}, \mu_{n_i})\}_{i \in \mathbb N}$, and by using \eqref{eq:conv:on:subsequence}, we have that
\begin{align*}
    \frac{1}{\sqrt{b_{n_i+1}}} \left\| \sum_{k=0}^{n_i} \alpha_{k+1}\begin{bmatrix}G_k - \nabla {f}(\theta_k) \\ H_k - \nabla_\mu v(\theta_k,\mu_k)\end{bmatrix} \right\| \sqrt{b_{n_i+1}} \left\| \begin{bmatrix}\theta^\star-\theta_{n_i+1} \\ \mu^\star-\mu_{n_i+1}\end{bmatrix} \right\| \overset{\text{a.s.}}{\to} 0,
\end{align*}
using Lemma~\ref{lemma:noise:convergence} and \eqref{eq:conv:on:subsequence}. Moreover, 
\begin{align*}
    \left\| \sum_{k=0}^{n_i} \alpha_{k+1}\begin{bmatrix}\nabla {f}(\theta_k) - \nabla {f}(\theta^\star) \\ \nabla_\mu v(\theta_k,\mu_k) - \nabla_\mu v(\theta^\star,\mu^\star)\end{bmatrix} \right\| \left\| \begin{bmatrix}\theta^\star-\theta_{n_i+1} \\ \mu^\star-\mu_{n_i+1}\end{bmatrix} \right\| \leq \sqrt{C} \sqrt{b_{n_i+1}} \left\| \begin{bmatrix}\theta^\star-\theta_{n_i+1} \\ \mu^\star-\mu_{n_i+1}\end{bmatrix} \right\|\overset{\text{a.s.}}{\to} 0,
\end{align*}
where the inequality follows from Lemma~\ref{lemma:gradients:error:bound} and the convergence follows from \eqref{eq:conv:on:subsequence}. Finally, 
\begin{align*}
    \inner{\sum_{k=0}^n \alpha_{k+1} \begin{bmatrix}\nabla{f}(\theta^\star) \\ \nabla_\mu v(\theta^\star,\mu^\star)\end{bmatrix}}{\begin{bmatrix}\theta^\star-\theta_{n_i+1} \\ \mu^\star-\mu_{n_i+1}\end{bmatrix}} \leq 0
\end{align*}
follows from the first-order optimality conditions for $(\theta^\star,\mu^\star)$ in \eqref{eq:so} and \eqref{eq:optimal:IS:2}. The last three displays imply that, with probability one,
\begin{align}
\label{eq:bounds:R_n+1:first:term}
    \limsup_{i \to \infty}\; \inner{\sum_{k=0}^{n_i} \alpha_{k+1} \begin{bmatrix}G_k \\ H_k \end{bmatrix}}{\begin{bmatrix}\theta^\star-\theta_{n_i+1} \\ \mu^\star-\mu_{n_i+1}\end{bmatrix}} \leq 0.
\end{align}
Now, from \eqref{eq:bounds:R_n+1:last:two:terms} and \eqref{eq:bounds:R_n+1:first:term}, we have that $R_{n_i} \overset{\text{a.s.}}{\to} 0$, and since $R_n \overset{\text{a.s.}}{\to} R_\infty$, we obtain $R_\infty = 0$. This guarantees the desired convergence \eqref{eq:as:conv:theta:mu}, and concludes the proof of Theorem~\ref{thm:as:conv}.
\end{proof}

Proving the almost sure converge of iterates in SA algorithms is by now a relatively standard procedure, which generally employs the use of results such as Theorem 1 in \citet{robbins1971convergence} (see Lemma~\ref{lemma:robbins:siegmund}). Interestingly, the joint NDA iteration \eqref{eq:NDA:SA+IS:iteration} introduces a technical challenge that is generally not present in the existing proofs in the literature. We detail this in the following remark.

\begin{remark}[No time-scale separation]
\label{remark:as:conv:identification:active}
As reported in Lemma~\ref{lemma:robbins:siegmund}, the Robbins and Siegmund theorem requires the existence of four sequences of \emph{nonnegative} random variables $R_n$, $A_n$, $B_n$, $C_n$ which are adapted to a filtration $\mathcal F_n$, and which satisfy the ``almost supermartingale'' property
\begin{align*}
    \mathbb E[R_{n+1}| \mathcal F_{n}] \leq (1+A_n)R_n + B_n - C_n.
\end{align*}
For the joint NDA iteration \eqref{eq:NDA:SA+IS:iteration}, the standard choice in the literature for $R_n$ (see, e.g., \cite[Theorem~2]{duchi2021asymptotic}) leads to the following term as a natural candidate for $C_n$:
\begin{align}
\label{eq:as:conv:identification:active:1}
    \alpha_{n+1}\inner{\begin{bmatrix}\nabla {f}(\theta_n) \\ \nabla_\mu v(\theta_n,\mu_n)\end{bmatrix}}{\begin{bmatrix}\theta_{n}-\theta^\star \\ \mu_{n}-\mu^\star \end{bmatrix}}.
\end{align}
However, in contradistinction to the standard literature, in our case this term is not guaranteed to be nonnegative if $\theta_n \neq \theta^\star$ (for $\theta_n = \theta^\star$ the nonnegativity follows from the first-order optimality conditions in \eqref{eq:so}). In practical terms, the condition $\theta_n = \theta^\star$ requires the iterates $\theta_n$ to have converged to the optimal $\theta^\star$ when the iterates $\mu_n$ are potentially very far from the optimal IS parameter $\mu^\star$. Clearly, if such a time-scale separation were to hold, the entire IS scheme would be useless in practice. Importantly, in Theorem~\ref{thm:as:conv} we show that such time-scale separation is not necessary as long as Assumptions~\ref{assump:as:conv}\ref{assump:as:conv:3}-\ref{assump:as:conv:4} hold and the set $\mathcal M$ is bounded with $\diam(\mc M) \defn \max_{\mu, \mu'\in \mc M}\norm{\mu-\mu'}<\infty$. The key steps in proving this build on the observation that the term \eqref{eq:as:conv:identification:active:1} can be lower bounded by the following sum:
\begin{align*}
    \alpha_{n+1}\inner{\begin{bmatrix}\nabla {f}(\theta_n) \\ \nabla_\mu v(\theta^\star,\mu_n)\end{bmatrix}}{\begin{bmatrix}\theta_{n}-\theta^\star \\ \mu_{n}-\mu^\star \end{bmatrix}} - \diam(\mc M) \alpha_{n+1}\left\|{\nabla_\mu v(\theta_n,\mu_n)-\nabla_\mu v(\theta^\star,\mu_n)}\right\|,
\end{align*}
where the first term is nonnegative (and therefore is a proper candidate for $C_n$) and the second term can be upper bounded by a ``well-behaved'' (i.e., summable) term under Assumptions~\ref{assump:as:conv}\ref{assump:as:conv:3}-\ref{assump:as:conv:4}. \hfill $\clubsuit$
\end{remark}

% ------------------------------------------------------------------------------------

\subsection{Identification of Active and Inactive Constraints}

We will now proceed to prove the two CLTs \eqref{eq:CLT1:conv} and \eqref{eq:CLT2:conv} which show that the proposed coupled NDA iteration \eqref{eq:NDA:SA+IS:iteration} yield the minimal asymptotic variance in the IS class. As a fundamental step in proving this, we fill first show that that the sequence $\{(\theta_n,\mu_n)\}_{n \in \mathbb N}$ identifies the active constraints in \eqref{eq:so} and \eqref{eq:optimal:IS:2} in finite time. For this, we require the following regularity assumptions. Recall that $A^\star_{a} \theta^\star = b^\star_a$ denotes the active constraints in \eqref{eq:so} and $C^\star_a \mu^\star = d^\star_a$ denote the active constraints in \eqref{eq:optimal:IS:2}.

\begin{assumption}[Regularity assumptions~II]~
\label{assump:finite:time:ident}
\begin{enumerate}[label=(\roman*)]
\item \label{assump:finite:time:ident:1} $-\nabla {f}(\theta^\star) = {A_a^\star}^\top \lambda_1$, for some $\lambda_1 \in \mathbb R^{p_1}_{++}$.  
\item \label{assump:finite:time:ident:2} $-\nabla_\mu v(\theta^\star,\mu^\star) = {C_a^\star}^\top \lambda_2$, for some $\lambda_2 \in \mathbb R^{p_2}_{++}$.
\end{enumerate}
\end{assumption}

Assumptions~\ref{assump:finite:time:ident}\ref{assump:finite:time:ident:1}-\ref{assump:finite:time:ident:2} are standard constraint qualifications for constrained optimization problems, with clear geometric meaning. Specifically, Assumptions~\ref{assump:finite:time:ident}\ref{assump:finite:time:ident:1} requires that $-\nabla {f}(\theta^\star)$ should belong to the \emph{relative interior} of the normal cone of $\Theta$ at $\theta^\star$. A similar intuition holds for Assumptions~\ref{assump:finite:time:ident}\ref{assump:finite:time:ident:2}. Finally, notice that Assumption~\ref{assump:finite:time:ident}\ref{assump:finite:time:ident:1} is the same as Assumption~\ref{assump:as:conv}\ref{assump:as:conv:4}. For clarity and to highlight the symmetry between the \eqref{eq:so} and \eqref{eq:optimal:IS:2} problems, we have stated it here as well. 

\begin{proposition}[Finite-time identification of active and inactive constraints]
\label{prop:fin:time:ident}
Let Assumptions~\ref{assump:SO}, \ref{assump:IS}, \ref{assump:as:conv}, and \ref{assump:finite:time:ident} be satisfied, and let $b_{n+1} = \sum_{k=0}^n \alpha_{k+1}$. Then, there exists some random finite $N$ such that $A_{a}^\star \theta_n = b_a^\star$, $A_{i}^\star \theta_n < b_i^\star$ and $C_{a}^\star \mu_n = d_a^\star$, $C_{i}^\star \mu_n < d_i^\star$, for all $n \geq N$.
\end{proposition}
\begin{proof}
First notice that the identification of the inactive constraints follows immediately from Theorem~\ref{thm:as:conv}. Indeed, since $(\theta_n,\mu_n) \overset{\text{a.s.}}{\to} (\theta^\star,\mu^\star)$, and since $A_{i}^\star \theta^\star < b_i^\star$ and $C_{i}^\star \mu^\star < d_i^\star$, there exists some random finite $N$ such that $A_{i}^\star \theta_n < b_i^\star$ and $C_{i}^\star \mu_n < d_i^\star$, for all $n \geq N$.

We will now show that the procedure~\eqref{eq:NDA:SA+IS:iteration} identifies the active constraints. This proof is a straightforward application of Lemma~\ref{lemma:duchi:4.2}, as explained in what follows. Following a similar reasoning as in \cite[Theorem~3]{duchi2021asymptotic}, we start by rewriting iteration~\eqref{eq:NDA:SA+IS:iteration} as
\begin{align*}
    \begin{bmatrix}\theta_{n+1} \\ \mu_{n+1} \end{bmatrix} = \argmin_{\substack{\theta \in \Theta \\ \mu \in \mathcal M}} \left\{ \inner{\begin{bmatrix}g_1 \\ g_2\end{bmatrix}}{\begin{bmatrix}\theta \\ \mu \end{bmatrix}} + \inner{\begin{bmatrix}v_n \\ w_n \end{bmatrix}}{\begin{bmatrix}\theta \\ \mu \end{bmatrix}} +\frac{1}{2 b_{n+1}} \left\| \begin{bmatrix}\theta \\ \mu \end{bmatrix} \right\|^2\right\},
\end{align*}
with 
\begin{align*}
    \begin{bmatrix}g_1 \\ g_2 \end{bmatrix} = \begin{bmatrix}\nabla {f}(\theta^\star) \\ \nabla_\mu v(\theta^\star,\mu^\star)\end{bmatrix} \quad \text{and} \quad  
    \begin{bmatrix}v_n \\ w_n \end{bmatrix} = \frac{1}{b_{n+1}} \left(\sum_{k=0}^n \alpha_{k+1} \begin{bmatrix}G_k \\ H_k \end{bmatrix} - b_{n+1} \begin{bmatrix}\nabla {f}(\theta^\star) \\ \nabla_\mu v(\theta^\star,\mu^\star)\end{bmatrix} \right).
\end{align*}
Now, from the KKT conditions for $(\theta^\star,\mu^\star)$ we have that there exist $\lambda_1 \in \mathbb R_{+}^{p_1}$ and $\lambda_2 \in \mathbb R_{+}^{p_2}$ such that $\nabla {f}(\theta^\star) + {A_a^\star}^\top \lambda_1 = 0$ and $\nabla_\mu v(\theta^\star,\mu^\star) + {C_a^\star}^\top \lambda_2 = 0$. Moreover, using Assumption~\ref{assump:finite:time:ident} we know that $\lambda_1, \lambda_2$ can be chosen strictly positive. Therefore, $g_1 = -{A_a^\star}^\top \lambda_1$ and $g_2 = - {C_a^\star}^\top \lambda_2$, for some $\lambda_1 \in \mathbb R_{++}^{p_1}$ and $\lambda_2 \in \mathbb R_{++}^{p_2}$. Additionally, using Lemma~\ref{lemma:conv:gradients} and the fact that $b_{n} \to \infty$, we have that $1/b_{n+1} \to 0$ and $(v_n,w_n) \to 0$ almost surely as $n \to \infty$. The result now follows from Lemma~\ref{lemma:duchi:4.2}.
\end{proof}

% ------------------------------------------------------------------------------------

\subsection{Asymptotic Normality}

Armed with the almost sure convergence in Theorem~\ref{eq:as:conv:theta:mu} and the finite-time identification of the active and inactive constraints in Proposition~\ref{prop:fin:time:ident}, we are now ready to prove the second main result of this paper. For this, we require the following regularity assumptions.

\begin{assumption}[Regularity assumptions~III]~
\label{assump:asymp:norm}
\begin{enumerate}[label=(\roman*)]
    \item \label{assump:asymp:norm:1}  There exist $c,\varepsilon >0$ such that for all $\theta \in \Theta \cap \{\theta:\; \|\theta-\theta^\star\| \leq \varepsilon\}$,
    \begin{align*}
        \left\|\nabla {f}(\theta) -\nabla {f}(\theta^\star) - \nabla^2 {f}(\theta^\star) (\theta-\theta^\star)\right\| \leq c \left\|\theta-\theta^\star \right\|^2.
    \end{align*}

    \item \label{assump:asymp:norm:2} There exist $c,\varepsilon >0$ such that for all $\mu \in \mathcal M \cap \{\mu:\; \|\mu-\mu^\star\| \leq \varepsilon\}$,
    \begin{align*}
        \left\|\nabla_\mu v(\theta^\star,\mu) -\nabla_\mu v(\theta^\star,\mu^\star) - \nabla_\mu^2 v(\theta^\star,\mu^\star) (\mu-\mu^\star)\right\| \leq c \left\|\mu-\mu^\star \right\|^2.
    \end{align*}

    \item \label{assump:asymp:norm:3} There exists $\rho >0$ such that for all $x \in \text{Ker}(A^\star_a)$ and 
    $y \in \text{Ker}(C^\star_a)$,
    \begin{align*}
        x^\top \nabla^2 f(\theta^\star) x \geq \rho \|x\|^2 \quad \text{and} \quad y^\top \nabla_\mu^2 v(\theta^\star,\mu^\star) y \geq \rho \|y\|^2.
    \end{align*}
\end{enumerate}
\end{assumption}

Assumption~\ref{assump:asymp:norm}\ref{assump:asymp:norm:1} is a standard second-order regularity assumption needed to prove the asymptotic normality of SA algorithms \cite{polyak1992acceleration}. In our case, since we 
are dealing with a joint SA procedure, Assumption~\ref{assump:asymp:norm}\ref{assump:asymp:norm:2} is a natural symmetric requirement for the $\mu_n$ iterates. Assumption~\ref{assump:asymp:norm}\ref{assump:asymp:norm:3} is a standard restricted strong convexity assumption needed to prove the asymptotic normality of both the SAA and SA procedures \cite{shapiro1989asymptotic,duchi2021asymptotic}.

Additionally, as in any asymptotic normality study, we need to impose an assumption on the asymptotic negligibility of the martingale difference process. For this, we introduce the following notation. As in the proof of Theorem~\ref{thm:as:conv}, we consider the filtration $\mathcal F_n := \sigma(X_k^{(\mu_{k-1})}, X_k |\, k \leq n)$. Moreover, for ease of notation, we define the noise vector
\begin{align}
\label{eq:def:xi_k}
    \xi_k := \begin{bmatrix}G_k \\ H_k \end{bmatrix} - \begin{bmatrix}\nabla{f}(\theta_k) \\ \nabla_\mu v(\theta_k,\mu_k)\end{bmatrix}.
\end{align}
Notice that $\{\xi_k\}_{k \in \mathbb N}$ is a martingale difference process adapted to the filtration $\{\mathcal F_{k+1}\}_{k \in \mathbb N}$. We are now ready to state the last set of regularity assumptions.

\begin{assumption}[Regularity assumptions~IV]~
\label{assump:asymp:norm:unif:int}
\begin{enumerate}[label=(\roman*)]
    \item \label{assump:asymp:norm:1:1} For all $t \in \mathbb R^{s+m}$ and $\varepsilon >0$,
    \begin{align*}
        \frac{1}{n}\sum_{k=0}^{n-1} \mathbb E \left[\left(t^\top \xi_k \right)^2 \,\mathbf{1}_{\left\{|t^\top \xi_k| \, > \, \varepsilon \sqrt{n}  \right\}} | \mathcal F_{k}\right] 
        \overset{P}{\to} 
        0.
    \end{align*}

    \item \label{assump:asymp:norm:1:2} For all $t \in \mathbb R^{s+m}$,
    \begin{align}
    \label{eq:assump:asymp:norm:1:2}
        t^\top \left(\frac{1}{n} \sum_{k=0}^{n-1}  \mathbb E \left[ 
        \xi_k \xi_k^\top | \mathcal F_k \right]\right) t 
        \overset{P}{\to} 
        t^\top \begin{bmatrix}
        \text{Var}_{X^{(\mu^\star)} \sim \mathbb P_{\mu^\star}}\left[ G_{\mu^\star}(\theta^\star,X^{(\mu^\star)})\right] & 0 \\ 0 & \text{Var}_{X \sim \mathbb P}\left[ H(\theta^\star, \mu^\star, X)\right]
    \end{bmatrix} t.
    \end{align}
\end{enumerate}
\end{assumption}

Assumption~\ref{assump:asymp:norm:unif:int} is a common assumption needed to ensure that the noise sequence $\{\xi_k\}_{k \in \mathbb N}$ satisfies the CLT $\frac{1}{\sqrt{n}} \sum_{k=0}^{n-1} \xi_k \overset{d}{\to} \mathcal N(0,\Sigma)$, for some $\Sigma \geq 0$. In particular, Assumption~\ref{assump:asymp:norm:unif:int}\ref{assump:asymp:norm:1:1} is the standard conditional Lindeberg condition needed to prove martingale CLTs (see \cite[Chapter~3]{hall2014martingale}). Moreover, Assumption~\ref{assump:asymp:norm:unif:int}\ref{assump:asymp:norm:1:2} guarantees that $\Sigma$ is a block-diagonal matrix with $\text{Var}_{X^{(\mu^\star)} \sim \mathbb P_{\mu^\star}}\left[ G_{\mu^\star}(\theta^\star,X^{(\mu^\star)})\right]$ and $\text{Var}_{X \sim \mathbb P}\left[ H(\theta^\star, \mu^\star, X)\right]$ on the diagonal. 

\begin{remark}[Decomposition of Assumption~\ref{assump:asymp:norm:unif:int}\ref{assump:asymp:norm:1:2}]
\label{assump:decomp}
Due to the independence of $X_{k+1}^{(\mu_k)} \sim \mathbb P_{\mu_k}$ and $X_k \sim \mathbb P$, $\mathbb E \left[\xi_k \xi_k^\top | \mathcal F_k \right]$ is automatically a block-diagonal matrix, with $\text{Var}_{X_{k+1}^{(\mu_{k})} \sim \mathbb P_{\mu_k}}\left[ G_{\mu_k}(\theta_k,X_{k+1}^{(\mu_k)}) | \mathcal F_k\right]$ and $\text{Var}_{X_{k+1} \sim \mathbb P}\left[ H(\theta_k, \mu_k, X_{k+1})| \mathcal F_k\right]$ on the diagonal, where we have used the full expression of the stochastic gradients (instead of just $G_k, H_k$) for clarity. Therefore, the convergence in Assumption~\ref{assump:asymp:norm:unif:int}\ref{assump:asymp:norm:1:2} can be equivalently restated as the following two convergences:
\begin{align*}
    t_1^\top \left(\frac{1}{n} \sum_{k=0}^{n-1} \text{Var}_{X_{k+1}^{(\mu_{k})} \sim \mathbb P_{\mu_k}}\left[ G_{\mu_k}(\theta_k,X_{k+1}^{(\mu_k)}) | \mathcal F_k\right] \right) t_1 
    &\overset{P}{\to} 
    t_1^\top 
    \text{Var}_{X^{(\mu^\star)} \sim \mathbb P_{\mu^\star}}\left[ G_{\mu^\star}(\theta^\star,X^{(\mu^\star)})\right] t_1, \\
    t_2^\top \left(\frac{1}{n} \sum_{k=0}^{n-1} \text{Var}_{X_{k+1} \sim \mathbb P}\left[ H(\theta_k, \mu_k, X_{k+1})| \mathcal F_k\right]\right) t_2 
    &\overset{P}{\to} 
    t_2^\top  \text{Var}_{X \sim \mathbb P}\left[ H(\theta^\star, \mu^\star, X)\right] t_2,
\end{align*}
for all $t_1 \in \mathbb R^s$ and $t_2 \in \mathbb R^m$. \hfill $\clubsuit$
\end{remark}

Before stating the main result of this section, we recall the definition of the averaged iterates $\bar{\theta}_n = n^{-1} \sum_{k=0}^{n-1} \theta_k$ and $\bar{\mu}_n = n^{-1} \sum_{k=0}^{n-1} \mu_k$, and define the covariance matrix
\begin{align}
\label{eq:Sigma12}
    \begin{bmatrix}
        \Sigma_G^\star & 0 \\ 0 & \Sigma_H^\star
    \end{bmatrix}
    \defn\begin{bmatrix}
        \text{Q}^\dagger \text{Var}_{X^{(\mu^\star)} \sim \mathbb P_{\mu^\star}}\left[ G_{\mu^\star}(\theta^\star,X^{(\mu^\star)})\right] \text{Q}^\dagger & 0 \\ 0 & \text{R}^\dagger \text{Var}_{X \sim \mathbb P}\left[ H(\theta^\star, \mu^\star, X)\right] \text{R}^\dagger
    \end{bmatrix}
\end{align}
with $\text{Q} \defn \text{P}_{A^\star_{a}}\nabla^2  f(\theta^\star) \text{P}_{A^\star_{a}}$ and $\text{R} \defn \text{P}_{C^\star_{a}}\nabla_\mu^2 v(\theta^\star,\mu^\star) \text{P}_{C^\star_{a}}$. 

\begin{theorem}[Asymptotic optimality~I]
\label{thm:asymp:norm}
Let Assumptions~\ref{assump:SO}, \ref{assump:IS}, \ref{assump:as:conv}, \ref{assump:finite:time:ident}, \ref{assump:asymp:norm}, and \ref{assump:asymp:norm:unif:int} be satisfied. Then,
\begin{align*}
    \sqrt{n}\, \begin{bmatrix}
        \bar{\theta}_n - \theta^\star \\ \bar{\mu}_n - \mu^\star
    \end{bmatrix}  \overset{d}{\to} \mathcal N \left( \begin{bmatrix} 0 \\ 0 \end{bmatrix}, \begin{bmatrix}
        \Sigma_G^\star & 0 \\ 0 & \Sigma_H^\star
    \end{bmatrix}\right),
\end{align*}
with matrices $\Sigma_G^\star$ and $\Sigma_H^\star$ defined in \eqref{eq:Sigma12}, and where $\theta^\star$ is the optimal solution in \eqref{eq:so} and $\mu^\star$ is the optimal solution in \eqref{eq:optimal:IS:2}.
\end{theorem}
\begin{proof}
The proof builds upon Lemma~\ref{lemma:duchi:CLT} in the Appendix. From the KKT conditions for $(\theta_{n+1},\mu_{n+1})$ in \eqref{eq:NDA:SA+IS:iteration}, we have that there exist $\lambda_{A^\star_a,n}, \lambda_{C^\star_a,n}, \lambda_{A^\star_i,n}, \lambda_{C^\star_i,n} \geq 0$ such that
\begin{align*}
    \begin{bmatrix} \theta_{n+1} \\ \mu_{n+1} \end{bmatrix} + \sum_{k=0}^n \alpha_{k+1} \begin{bmatrix}G_k \\ H_k \end{bmatrix} + \begin{bmatrix} A_a^{\star \top} \lambda_{A^\star_a,n} \\ C_a^{\star\top} \lambda_{C^\star_a,n}\end{bmatrix}  + \begin{bmatrix} A_i^{\star\top} \lambda_{A^\star_i,n} \\ C_i^{\star\top} \lambda_{C^\star_i,n} \end{bmatrix}  = 0.
\end{align*}
Therefore, 
\begin{align*}
    \begin{bmatrix} \theta_{n+1} \\ \mu_{n+1} \end{bmatrix} = \begin{bmatrix} \theta_{n} \\ \mu_{n} \end{bmatrix} - \alpha_{n+1} \begin{bmatrix}G_n \\ H_n \end{bmatrix} + \begin{bmatrix} A_a^{\star\top} (\lambda_{A^\star_a,n-1} - \lambda_{A^\star_a,n}) \\ C_a^{\star\top} (\lambda_{C^\star_a,n-1} - \lambda_{C^\star_a,n})\end{bmatrix}  + \begin{bmatrix} A_i^{\star\top} (\lambda_{A^\star_i,n-1} - \lambda_{A^\star_i,n}) \\ C_i^{\star\top} (\lambda_{C^\star_i,n-1} - \lambda_{C^\star_i,n}) \end{bmatrix}.
\end{align*}
Subtracting $(\theta^\star,\mu^\star)$ and pre-multiplying by the block diagonal matrix $\text{diag}(\text{P}_{A^\star_a},\text{P}_{C^\star_a})$, we obtain
\begin{align}
\label{eq:proof:asymp:norm:iter:1}
    \begin{bmatrix} \text{P}_{A^\star_a}(\theta_{n+1} - \theta^\star) \\ \text{P}_{C^\star_a}(\mu_{n+1}-\mu^\star) \end{bmatrix} = \begin{bmatrix} \text{P}_{A^\star_a}(\theta_{n} - \theta^\star) \\ \text{P}_{C^\star_a}(\mu_{n}-\mu^\star) \end{bmatrix} - \alpha_{n+1} \begin{bmatrix} \text{P}_{A^\star_a} G_n \\ \text{P}_{C^\star_a} H_n \end{bmatrix} + \begin{bmatrix} \text{P}_{A^\star_a} A_i^{\star\top} (\lambda_{A^\star_i,n-1} - \lambda_{A^\star_i,n}) \\ \text{P}_{C^\star_a} C_i^{\star\top} (\lambda_{C^\star_i,n-1} - \lambda_{C^\star_i,n}) \end{bmatrix},
\end{align}
where we have used the fact that the fact that $\text{P}_{A_a^\star} {A_a^\star}^\top = 0$ and $\text{P}_{C_a^\star} {C_a^\star}^\top = 0$. We now define
\begin{align*}
    \Delta_{n} &:= \begin{bmatrix} \text{P}_{A^\star_a}(\theta_{n} - \theta^\star) \\ \text{P}_{C^\star_a}(\mu_{n}-\mu^\star) \end{bmatrix}, 
    \quad
    H := \begin{bmatrix} \nabla^2 {f}(\theta^\star) & 0\\ 0 & \nabla_\mu^2 v(\theta^\star,\mu^\star)\end{bmatrix}, 
    \quad
    \xi_n := \begin{bmatrix}G_n - \nabla f(\theta_n) \\ H_n - \nabla_\mu v(\theta_n,\mu_n)\end{bmatrix}
     \\
    \zeta_n &:= \begin{bmatrix} \nabla f(\theta_n) - \nabla f(\theta^\star) - \nabla^2 f(\theta^\star)(\theta_n - \theta^\star) \\  \nabla_\mu v(\theta_n,\mu_n) - \nabla_\mu v(\theta^\star,\mu^\star) - \nabla_\mu^2 v(\theta^\star,\mu^\star)(\mu_n - \mu^\star) \end{bmatrix}, \quad
    \text{P} = \begin{bmatrix} \text{P}_{A^\star_a} & 0 \\ 0 & \text{P}_{C^\star_a} \end{bmatrix}
     \\
    \epsilon_n &:= \begin{bmatrix} \text{P}_{A^\star_a} A_i^{\star\top} (\lambda_{A^\star_i,n-1} - \lambda_{A^\star_i,n}) \\ \text{P}_{C^\star_a} C_i^{\star\top} (\lambda_{C^\star_i,n-1} - \lambda_{C^\star_i,n}) \end{bmatrix} - \alpha_{n+1} \begin{bmatrix} \text{P}_{A^\star_a}\nabla^2 {f}(\theta^\star)(\text{I}- \text{P}_{A^\star_a}) (\theta_{n} - \theta^\star) \\ \text{P}_{C^\star_a} \nabla_\mu^2 v(\theta^\star,\mu^\star)(\text{I}- \text{P}_{C^\star_a})(\mu_{n}-\mu^\star) \end{bmatrix}.
\end{align*}
Using again the facts $\text{P}_{A_a^\star} {A_a^\star}^\top = 0$, $\text{P}_{C_a^\star} {C_a^\star}^\top = 0$ and the optimality conditions $\text{P}_{A_a^\star} \nabla f(\theta^\star) = 0$ and $\text{P}_{C_a^\star} \nabla_\mu v(\theta^\star,\mu^\star) = 0$, it can be easily checked that iteration \eqref{eq:proof:asymp:norm:iter:1} can be rewritten as
\begin{align}
\label{eq:iteration:Delta}
    \Delta_{n+1} = \Delta_n - \alpha_{n+1} \text{P} H \text{P} \Delta_n - \alpha_{n+1} \text{P} (\xi_n + \zeta_n) + \epsilon_n.
\end{align}
Now, using the finite-time identification of the active constraints in Proposition~\ref{prop:fin:time:ident}, we have that $\text{P}_{A^\star_a}(\theta_{n} - \theta^\star) = \theta_{n} - \theta^\star$ and $\text{P}_{C^\star_a}(\mu_{n}-\mu^\star) = \mu_{n}-\mu^\star$ with probability one for large enough $N$. Consequently, iteration~\eqref{eq:iteration:Delta} is in the form required by Lemma~\ref{lemma:duchi:CLT}. Therefore, in order to conclude the desired result, we need to verify that the Assumptions~\ref{assump:generic:CLT:1} and \ref{assump:generic:CLT:2} required by Lemma~\ref{lemma:duchi:CLT} are satisfied. 

We start with Assumption~\ref{assump:generic:CLT:1}. First, from Assumption~\ref{assump:asymp:norm}\ref{assump:asymp:norm:3}, we know that there exists $\rho > 0$ such that for all $w \in \mathcal T$, with
\begin{align*}
    \mathcal T := \left\{ w \in \mathbb R^{s+m}:\; \begin{bmatrix} A^\star_a & 0 \\ 0 & C^\star_a \end{bmatrix} w = 0 \right\},
\end{align*}
we have $w^\top H w \geq \rho \|w\|^2$. Recall that $G_n, H_n$ are adapted to the filtration $\mathcal F_{n+1} = \sigma(X_k^{(\mu_{k-1})}, X_k |\, k \leq n+1)$. Therefore, using Assumption~\ref{assump:IS}\ref{ass:IS:dominated}, we have that $\mathbb E[\|G_n\|^2 | \mathcal F_n] \leq G_M^2$ and $\mathbb E[\|H_n\|^2 | \mathcal F_n] \leq H_M^2$. Additionally, using Jensen's inequality, we have that $\|\nabla {f}(\theta_n)\|^2 \leq G_M^2$ and $\|\nabla_\mu v(\theta_n,\mu_n)\|^2 \leq H_M^2$. Therefore, for all $n \in \mathbb N$, $\mathbb E \left[ \|\xi_n\|^2 | \mathcal F_n\right] \leq 2 (G_M^2 + H_M^2)$. Finally, from Lemma~\ref{lemma:CLT:gradient:error} we know that 
\begin{align*}
    \frac{1}{\sqrt{n}} \sum_{k=0}^{n-1} \xi_k \overset{d}{\to} \mathcal N \left( \begin{bmatrix} 0 \\ 0 \end{bmatrix}, \begin{bmatrix}
        \text{Var}_{X^{(\mu^\star)} \sim \mathbb P_{\mu^\star}}\left[ G_{\mu^\star}(\theta^\star,X^{(\mu^\star)})\right] & 0 \\ 0 & \text{Var}_{X \sim \mathbb P}\left[ H(\theta^\star, \mu^\star, X)\right]
    \end{bmatrix}\right).
\end{align*}

We now focus on Assumption~\ref{assump:generic:CLT:2}, and start by showing that $n^{-1/2} \sum_{k=0}^{n-1} \|\text{P} \zeta_k\| \overset{\text{a.s.}}{\to} 0$. We first rewrite $\zeta_n$ as
\begin{align*}
    \zeta_n &:= \begin{bmatrix} \nabla {f}(\theta_n) - \nabla {f}(\theta^\star) - \nabla^2 {f}(\theta^\star)(\theta_n - \theta^\star) \\  \nabla_\mu v(\theta^\star,\mu_n) - \nabla_\mu v(\theta^\star,\mu^\star) - \nabla_\mu^2 v(\theta^\star,\mu^\star)(\mu_n - \mu^\star) \end{bmatrix} + \begin{bmatrix} 0 \\  \nabla_\mu v(\theta_n,\mu_n) - \nabla_\mu v(\theta^\star,\mu_n) \end{bmatrix}.
\end{align*}
Then, using Assumption~\ref{assump:asymp:norm}\ref{assump:asymp:norm:1}-\ref{assump:asymp:norm:2}, it can be immediately established that
\begin{align*}
    \left\| \begin{bmatrix} \nabla {f}(\theta_n) - \nabla {f}(\theta^\star) - \nabla^2 {f}(\theta^\star)(\theta_n - \theta^\star) \\  \nabla_\mu v(\theta^\star,\mu_n) - \nabla_\mu v(\theta^\star,\mu^\star) - \nabla_\mu^2 v(\theta^\star,\mu^\star)(\mu_n - \mu^\star) \end{bmatrix} \right\| \leq c \left\| \begin{bmatrix}
        \theta_n - \theta^\star \\ \mu_n - \mu^\star
    \end{bmatrix} \right\|^2.
\end{align*}
Moreover, using Assumption~\ref{assump:as:conv}\ref{assump:as:conv:3}, we have that
\begin{align*}
    \left\| \begin{bmatrix} 0 \\  \nabla_\mu v(\theta_n,\mu_n) - \nabla_\mu v(\theta^\star,\mu_n) \end{bmatrix} \right\| \leq c_3 \|\theta_n -\theta^\star \|^2 \leq c_3 \left\| \begin{bmatrix}
        \theta_n - \theta^\star \\ \mu_n - \mu^\star
    \end{bmatrix} \right\|^2,
\end{align*}
where the first inequality holds with probability one for large enough $n$ using Proposition~\ref{prop:fin:time:ident}. Using these bounds and the fact that $\|\text{P} \zeta_n\| \leq \|\zeta_n\|$ (since the projection operator is non-expansive), we obtain
\begin{align*}
    0 \leq \frac{1}{\sqrt{n}} \sum_{k=0}^{n-1} \|\text{P} \zeta_k\| \leq (c+c_3) \frac{1}{\sqrt{n}} \sum_{k=0}^{n-1} \left\| \begin{bmatrix}
        \theta_n - \theta^\star \\ \mu_n - \mu^\star
    \end{bmatrix} \right\|^2 \overset{\text{a.s.}}{\to} 0,
\end{align*}
where the convergence follows from Lemma~\ref{lemma:assump:conv:CLT}. 

Finally, in order to conclude the proof of Assumption~\ref{assump:generic:CLT:2} we only need to show that there exists a random variable $N < \infty$ such that $\epsilon_n = 0$ for $n \geq N$ (notice that the remaining two conditions are proven Theorem~\ref{thm:as:conv} and Lemma~\ref{lemma:assump:conv:CLT}). But this follows immediately from Proposition~\ref{prop:fin:time:ident}, as we now detail. Since there exists some random finite $N$ such that $A^\star_{a} \theta_n = b^\star_a$ and $C^\star_{a} \mu_n = d^\star_a$ for $n \geq N$, we have that $(\text{I}- \text{P}_{A^\star_a}) (\theta_{n} - \theta^\star) = 0$ and $(\text{I}- \text{P}_{C^\star_a})(\mu_{n}-\mu^\star) = 0$ for $n \geq N$. Moreover, since $A^\star_{i} \theta_n < b^\star_i$ and $C^\star_{i} \mu_n < d^\star_i$ for $n \geq N$, we have by complementary slackness that $\lambda_{A^\star_i,n} = \lambda_{C^\star_i,n} = 0$ for $n \geq N$. This concludes the proof of Assumption~\ref{assump:generic:CLT:2}, and with it the proof of Theorem~\ref{thm:asymp:norm}.
\end{proof}

Theorem~\ref{thm:asymp:norm} establishes the asymptotic normality of the coupled NDA procedure~\eqref{eq:NDA:SA+IS:iteration}. We highlight that the asymptotic covariance matrix~\eqref{eq:Sigma12} is block-diagonal, which is a natural consequence of the fact that the two stochastic gradients, $G_k$ and $H_k$, are sampled independently from the distributions $\mathbb P_{\mu_k}$ and $\mathbb P$, respectively. Importantly, from the asymptotic normality of the joint iterates, we can deduce the asymptotic optimality of the procedure for the iterates $\bar{\theta}_n$. Indeed, since the joint convergence in distribution implies the marginal convergence in distribution, we immediately recover the CLT
\begin{align*}
    \sqrt{n}\left( \bar{\theta}_n - \theta^\star \right) \overset{d}{\to} \mathcal N \left(0, \text{Q}^{\dagger}\, \text{Var}_{X^{(\mu^\star)} \sim \mathbb P_{\mu^\star}}\left[G_{\mu^\star}(\theta^\star, X^{(\mu^\star)})\right] \text{Q}^{\dagger} \right),
\end{align*}
with $\text{Q}:=\text{P}_{A^\star_{a}}\nabla^2  f(\theta^\star) \text{P}_{A^\star_{a}}$. Finally, invoking Assumption~\ref{assump:SO}\ref{assump:SO:twice:differentiable}, the delta method yields the projected gradient CLT as stated below.

\begin{corollary}[Asymptotic optimality~II]
\label{cor:asymp:norm}
Consider the setting of Theorem~\ref{thm:asymp:norm}. Then,
\begin{align*}
    \sqrt{n}\, \text{P}_{A^\star_{a}} \nabla f(\bar{\theta}_n)  \overset{d}{\to} \mathcal N \left(0, \text{Var}_{X^{(\mu^\star)} \sim \mathbb P_{\mu^\star}}\left[\text{P}_{A^\star_{a}} G_{\mu^\star}(\theta^\star, X^{(\mu^\star)})\right] \right).
\end{align*}
where $\theta^\star$ is the optimal solution in \eqref{eq:so} and $\mu^\star$ the optimal solution in (\ref{eq:optimal:IS:2}) satisfying
\begin{align*}
    \mu^\star = \argmin_{\mu \in \mathcal M} \trace\left(\text{Var}_{X^{(\mu)} \sim \mathbb P_\mu}\left[\text{P}_{A^\star_{a}} G_\mu(\theta^\star, X^{(\mu)})\right]\right).
\end{align*}
\end{corollary}
\begin{proof}
The result follows from Lemma~\ref{lemma:kkt-residual}, after noticing that Theorem~\ref{thm:asymp:norm} readily implies the CLT
\begin{align*}
    \sqrt{n}\,(\bar{\theta}_n - \theta^\star) \overset{d}{\to} \mathcal N (0, (\text{P}_{A^\star_{a}}\nabla^2 f(\theta^\star) \text{P}_{A^\star_{a}})^\dagger \text{Var}_{X^{(\mu^\star)} \sim \mathbb P_{\mu^\star}}\left[G_{\mu^\star}(\theta^\star, X^{(\mu^\star)})\right] (\text{P}_{A^\star_{a}}\nabla^2 f(\theta^\star) \text{P}_{A^\star_{a}})^\dagger).
\end{align*}
\end{proof}

This concludes the asymptotic analysis of the coupled NDA procedure. Together, Theorem~\ref{thm:asymp:norm} and Corollary~\ref{cor:asymp:norm} establish the asymptotic normality of the averaged iterates and the projected gradient, under the regularity conditions introduced above. These results verify the asymptotic efficiency of the proposed algorithm when combined with adaptive importance sampling.

% ------------------------------------------------------------------------------------
% ------------------------------------------------------------------------------------
% ------------------------------------------------------------------------------------

\section{Numerical Illustration}
\label{sec:numerics}

We illustrate the behavior of the proposed joint stochastic approximation and importance sampling scheme on a simple yet revealing instance of the quantile estimation problem from Example~\ref{example:nqe}. The goal is to highlight both adaptivity and asymptotic variance reduction in a rare-event regime.

\begin{figure}[t]
    \centering
    \begin{subfigure}[t]{0.49\linewidth}
        \centering
        \includegraphics[width=\linewidth]{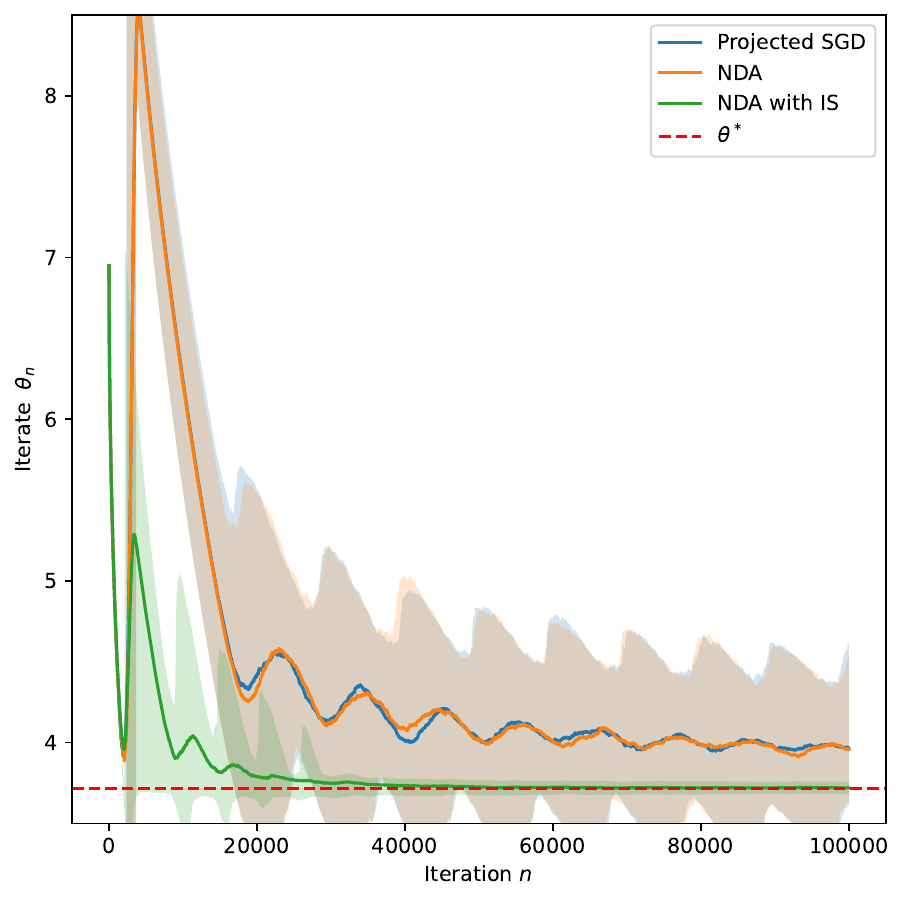}
        \caption{Raw iterates $\theta_n$}
    \end{subfigure}\hfill
    \begin{subfigure}[t]{0.49\linewidth}
        \centering
        \includegraphics[width=\linewidth]{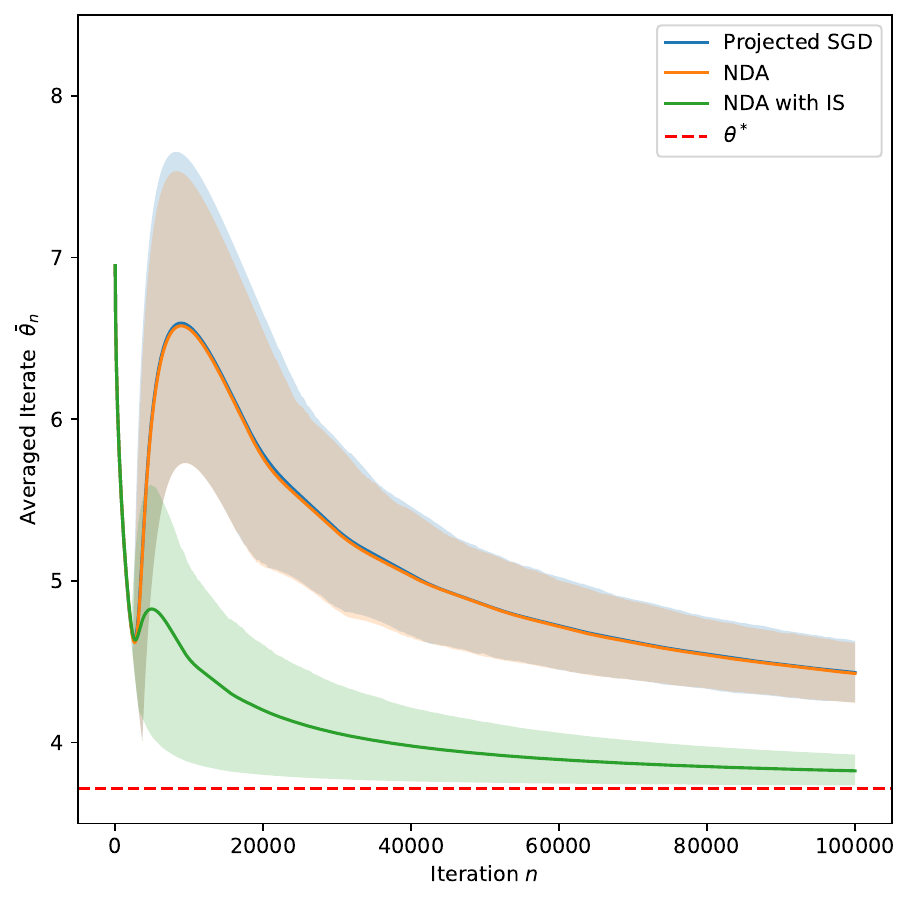}
        \caption{Averaged iterates $\bar\theta_n$}
    \end{subfigure}
    \caption{Evolution of the decision variable. The dashed red line indicates $\theta^\star$. NDA with IS exhibits rapid variance contraction once the IS parameter adapts, whereas Projected SGD and vanilla NDA do not exhibit comparable concentration.}
    \label{fig:theta}
\end{figure}

\medskip
\noindent
\textbf{Setup.}
We consider the estimation of the $0.9999$-quantile of a standard normal random variable $X \sim \mathcal N(0,1)$, formulated as the stochastic optimization problem in Example~\ref{example:nqe}. The unique minimizer is $\theta^\star \approx 3.72$. The decision variable is constrained to $\Theta=[-10,10]$, which is sufficiently large so that projection never becomes active in the simulation. In contrast, the importance sampling (IS) parameter is constrained to $\mathcal M=[-1.7,1.7]$. This choice is deliberate: for this problem, the unconstrained optimal exponential tilting parameter is larger (typically around $2$), but restricting to $\mathcal M$ forces the constrained optimum to lie at the boundary $\mu^\star=1.7$. This allows us to explicitly observe finite-time identification of the active constraint for the IS parameter, as predicted by Proposition~\ref{prop:fin:time:ident}. The IS family is exponential tilting, which in this setting corresponds to sampling
\[
    X_{n+1}^{(\mu_n)} \sim \mathcal N(\mu_n,1)
\]
at iteration $n$. The decision and IS parameters are updated using the NDA recursion \eqref{eq:NDA:SA+IS:iteration:extension}, with stepsizes of the form $\alpha_k=\alpha_0 k^{-\gamma}$, where $\gamma=0.55$, $\alpha_0=0.05$ for the $\theta$-update, and $\alpha_0=3\times 10^{-6}$ for the $\mu$-update. We initialize at $\theta_0=7$ and $\mu_0=0.2$. All plots are based on $1000$ independent trajectories. We display the empirical mean together with the $10\%$--$90\%$ quantile band.

\medskip
\noindent
\textbf{Decision variable dynamics.}
Figure~\ref{fig:theta} reports the evolution of the decision variable. Figure~\ref{fig:theta}(left) shows the raw iterates $\theta_n$, while Figure~\ref{fig:theta}(right) shows the averaged iterates $\bar\theta_n$. For reference, we include Projected SGD, vanilla NDA (without IS), and the proposed NDA with adaptive IS. For the raw iterates, NDA with IS initially exhibits relatively high variability, reflecting the exploratory phase during which the IS parameter is still adapting. Once the IS parameter approaches its optimal value, however, the variance of $\theta_n$ contracts rapidly and the iterates concentrate near $\theta^\star$ after roughly $4\times 10^4$ iterations. In contrast, both Projected SGD and vanilla NDA retain substantial variance and do not exhibit comparable concentration around $\theta^\star$ even after $10^5$ iterations (and similarly for longer runs up to $2\times 10^5$ iterations). Turning to the averaged iterates, we observe a similar trend. While the early variability of $\theta_n$ is naturally inherited by $\bar\theta_n$, the adaptive IS scheme nonetheless leads to significantly faster stabilization and lower dispersion. This behavior is consistent with Theorem~\ref{thm:as:conv}, which guarantees convergence of both $\theta_n$ and $\bar\theta_n$ to $\theta^\star$, and illustrates how variance reduction accelerates entry into the asymptotic regime.

\begin{figure}[t]
\centering
\begin{minipage}[t]{0.49\linewidth}
\centering
\includegraphics[width=\linewidth]{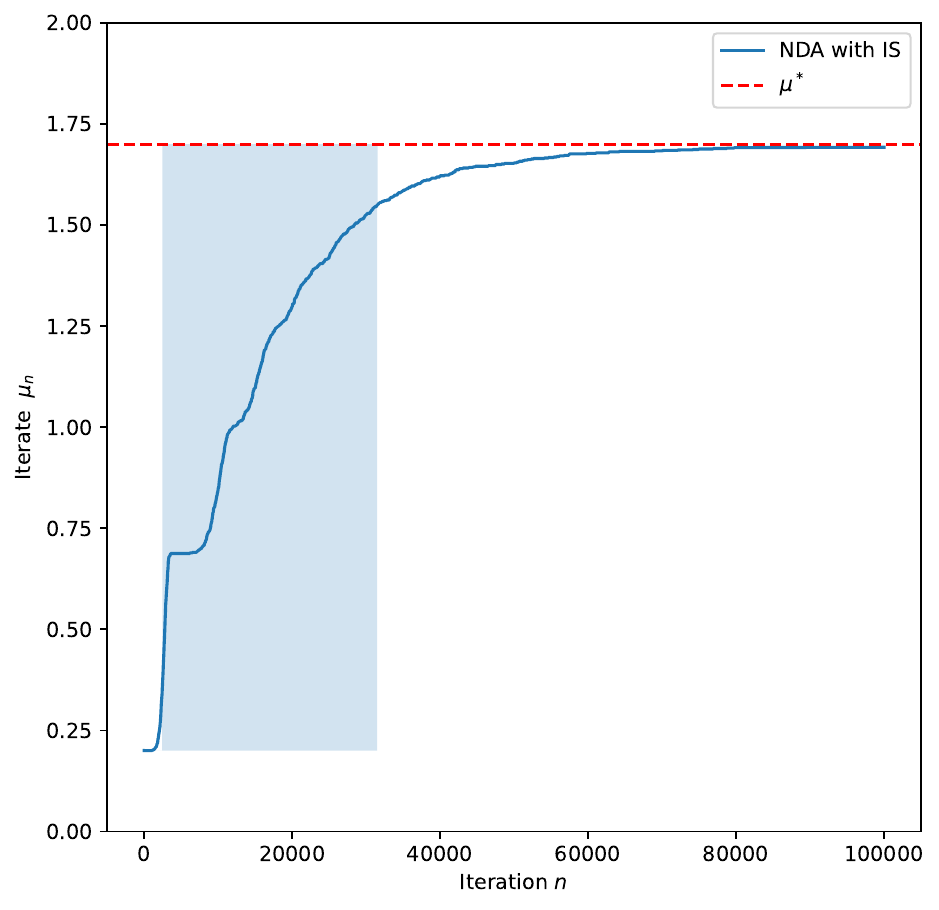}
\subcaption{Raw iterates $\mu_n$}
\end{minipage}\hfill
\begin{minipage}[t]{0.49\linewidth}
\centering
\includegraphics[width=\linewidth]{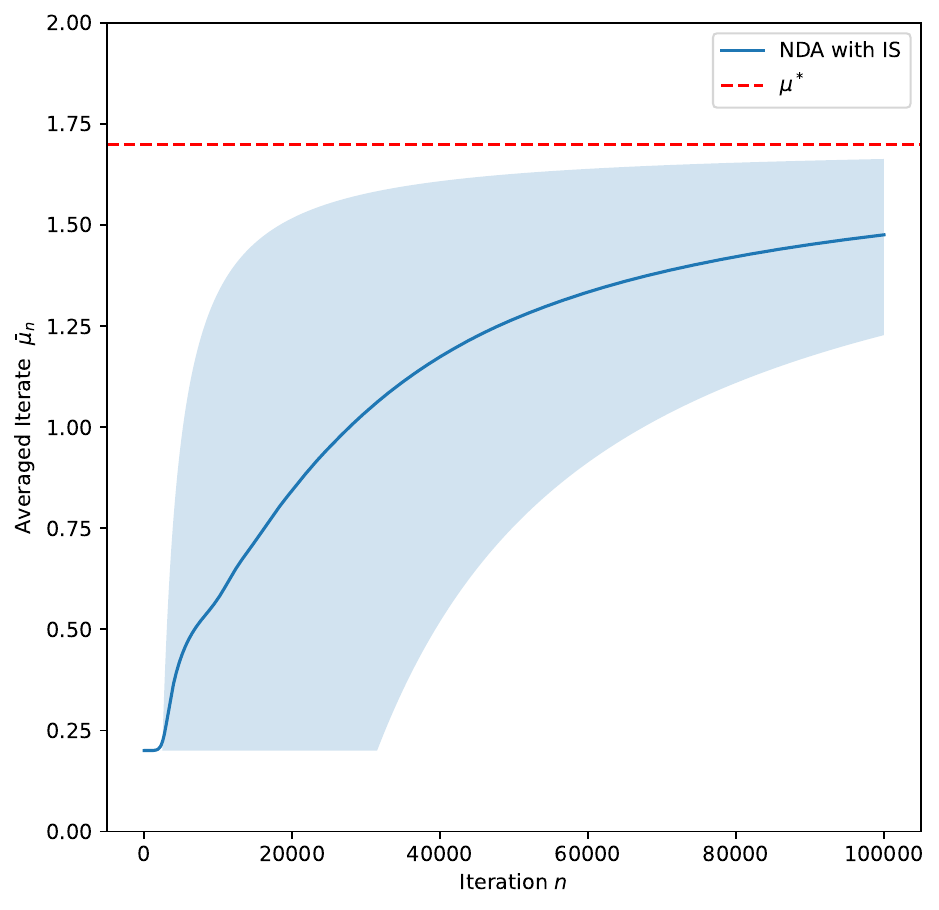}
\subcaption{Averaged iterates $\bar\mu_n$}
\end{minipage}
\caption{Evolution of the importance sampling parameter. The dashed red line indicates $\mu^\star=1.7$, corresponding to the constrained optimum over $\mathcal M$. The IS parameter identifies the active constraint in finite time, as predicted by Proposition~\ref{prop:fin:time:ident}.}
\label{fig:mu}
\end{figure}

\medskip
\noindent
\textbf{Importance sampling parameter dynamics and active constraint identification.}
Figure~\ref{fig:mu} reports the evolution of the importance sampling (IS) parameter. Figure~\ref{fig:mu}(left) shows the raw iterates $\mu_n$, while Figure~\ref{fig:mu}(right) shows the averaged iterates $\bar\mu_n$. Initially, the IS parameter exhibits substantial variability, reflecting the noise inherent in stochastic calibration of the tilting distribution. After a finite number of iterations, however, the iterates identify the active constraint $\mu=1.7$ (the constrained optimum over $\mathcal M$) and remain on this boundary thereafter. This finite-time active constraint identification is precisely the behavior predicted by Proposition~\ref{prop:fin:time:ident}. The averaged IS iterates display larger dispersion due to the influence of early iterations, but converge to the same limiting value. It is important to note that sampling at iteration $n$ is governed by the raw iterate $\mu_n$, since $X_{n+1}^{(\mu_n)} \sim \mathcal N(\mu_n,1)$, and therefore the rapid stabilization and active constraint identification of $\mu_n$ are the critical properties for effective importance sampling. Taken together with Figure~\ref{fig:theta}, these plots demonstrate the adaptivity of the proposed scheme: contraction of the decision iterates occurs simultaneously with stabilization and active constraint identification of the IS parameter, with both processes operating on the same time scale and without any form of time-scale separation, in line with Remark~\ref{remark:as:conv:identification:active}. We note that the residual variance in $\bar\mu_n$ inherited from early iterations could be further reduced by a secondary importance sampling scheme for the IS parameter itself, as discussed in Section~\ref{subsec:secondary:IS}.

\begin{figure}[t]
    \centering
    \begin{minipage}[t]{0.49\linewidth}
        \centering
        \includegraphics[width=\linewidth]{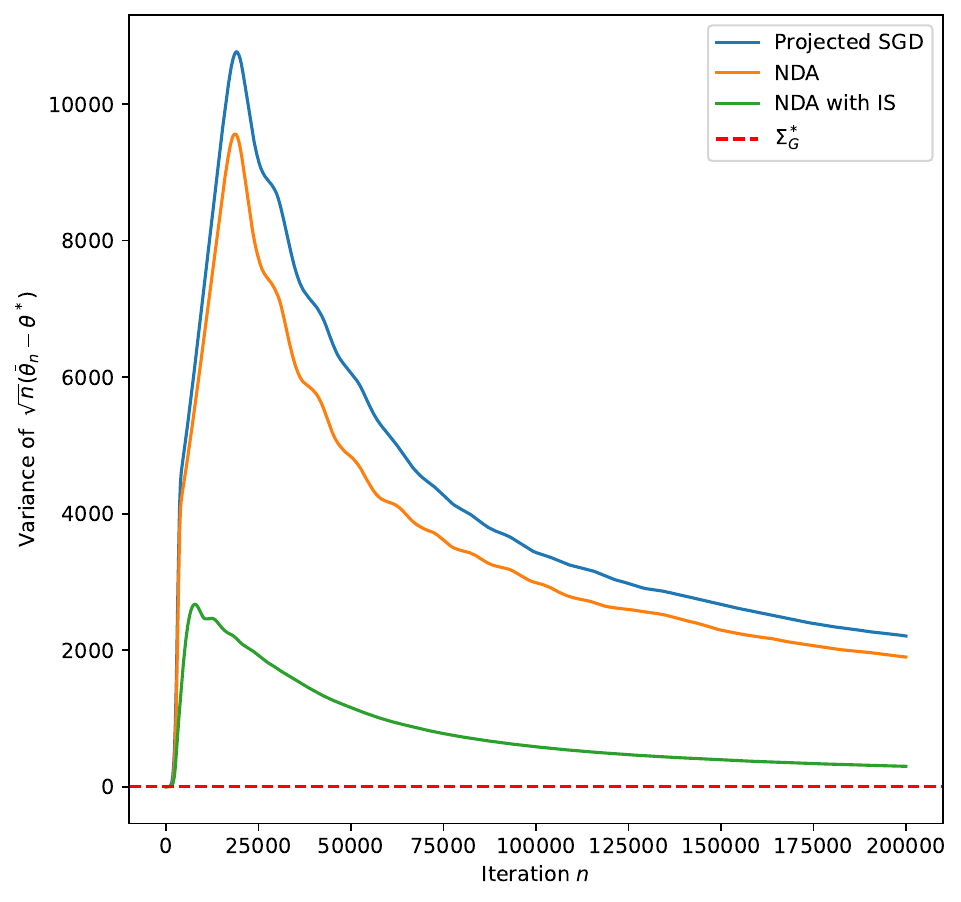}
        \subcaption{Without burn-in}
    \end{minipage}\hfill
    \begin{minipage}[t]{0.49\linewidth}
        \centering
        \includegraphics[width=\linewidth]{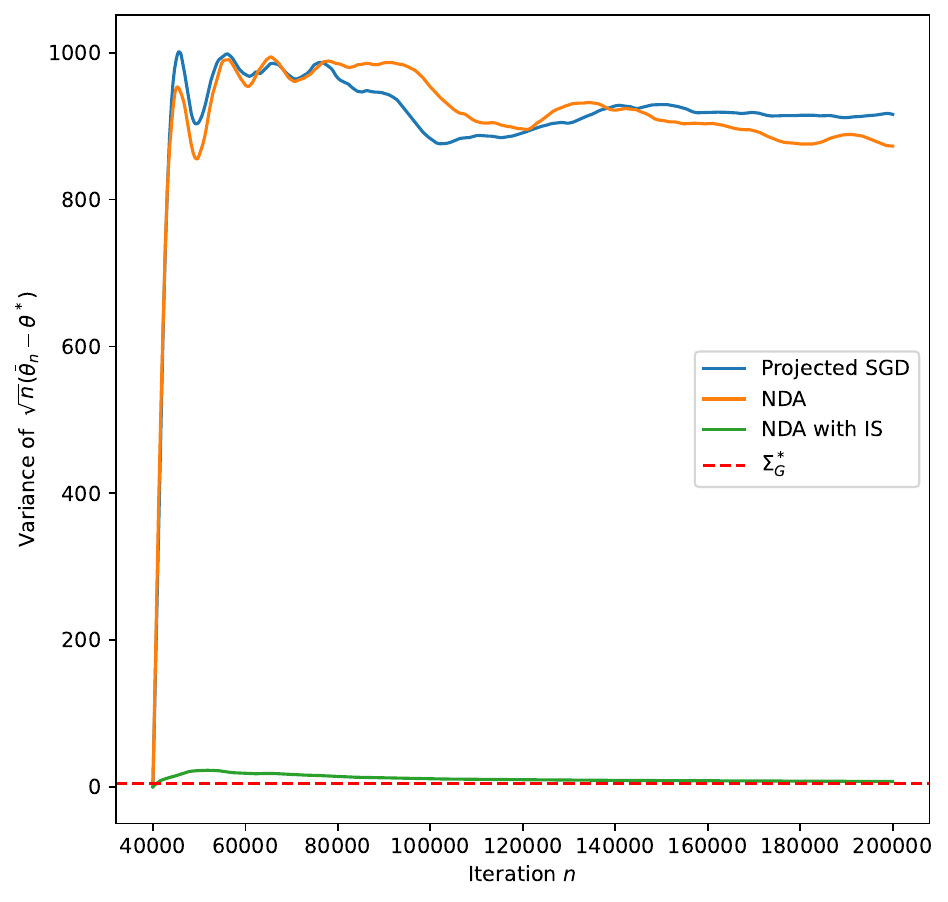}
        \subcaption{With burn-in}
    \end{minipage}
    \caption{Empirical variance of $\sqrt{n}(\bar\theta_n-\theta^\star)$. NDA with adaptive importance sampling attains a variance close to the theoretical optimum, while Projected SGD and vanilla NDA remain orders of magnitude larger.}
    \label{fig:variance}
\end{figure}

\medskip
\noindent
\textbf{Asymptotic variance comparison.}
Finally, Figure~\ref{fig:variance} compares the empirical variance of the scaled error process $\sqrt{n}(\bar\theta_n-\theta^\star)$ across methods. To isolate the asymptotic regime, Figure~\ref{fig:variance}(right) reports variance estimates computed after a burn-in period of $4\times 10^4$ iterations, by which time the importance sampling parameter has stabilized near its constrained optimum and the asymptotic normality regime described in Theorem~\ref{thm:asymp:norm} becomes relevant. In this regime, NDA with adaptive importance sampling empirically attains a variance that is orders of magnitude smaller than those of Projected SGD and vanilla NDA, closely matching the optimal asymptotic variance predicted by Theorem~\ref{thm:asymp:norm}. Even after $2\times 10^5$ iterations, the baseline methods exhibit variances that are roughly three orders of magnitude larger. For completeness, Figure~\ref{fig:variance}(left) reports the same variance estimates without burn-in. While the transient phase leads to slower convergence due to accumulated early variability, the adaptive importance sampling scheme still achieves at least a two-order-of-magnitude improvement over the baselines.

% ------------------------------------------------------------------------------------
% ------------------------------------------------------------------------------------
% ------------------------------------------------------------------------------------

%\section{Conclusion}
%\label{sec:conclusion}

%We proposed a single-loop stochastic approximation scheme that jointly updates the decision variable and the importance sampling (IS) distribution in constrained convex stochastic optimization problems. The method avoids time-scale separation, circumvents nested optimization, and achieves asymptotic optimality without requiring prior knowledge of the optimal IS parameters. Under mild regularity assumptions, we established global convergence and proved a central limit theorem for the averaged iterates, showing that the procedure attains the minimum possible asymptotic variance.

\section*{Acknowledgments}

Liviu Aolaritei acknowledges support from the Swiss National Science Foundation through the Postdoc.Mobility Fellowship (grant agreement P500PT\_222215). Henry Lam acknowledges support from the InnoHK initiative of the Innovation and Technology Commission of the Hong Kong Special Administrative Region Government, and Laboratory for AI-Powered Financial Technologies. Michael I. Jordan was funded in part by the European Union (ERC-2022-SYG-OCEAN-101071601). Views and opinions expressed are however those of the author(s) only and do not necessarily reflect those of the European Union or the European Research Council Executive Agency. Neither the European Union nor the granting authority can be held responsible for them.

We are grateful to Wasin Meesena for running the numerical simulations for this paper.

% ------------------------------------------------------------------------------------
% ------------------------------------------------------------------------------------
% ------------------------------------------------------------------------------------

\bibliographystyle{abbrvnat} 
\bibliography{bibfile.bib}

% ------------------------------------------------------------------------------------
% ------------------------------------------------------------------------------------
% ------------------------------------------------------------------------------------

\appendix

\section{Technical Preliminaries}

\begin{lemma}[{\citet[Theorem 1]{robbins1971convergence}}]
\label{lemma:robbins:siegmund}
Let $R_n$, $A_n$, $B_n$, $C_n$ be nonnegative random variables adapted to a filtration $\mathcal F_n$. Assume that 
\begin{align*}
    \mathbb E[R_{n+1}| \mathcal F_{n}] \leq (1+A_n)R_n + B_n - C_n.
\end{align*}
Then, on the event $\{\sum_n A_n < \infty, \sum_n B_n < \infty\}$, there exists a random variable $R_\infty < \infty$ such that $R_n \overset{\text{a.s.}}{\to} R_\infty$ and $\sum_n C_n < \infty$ almost surely.
\end{lemma}

\begin{lemma}[\citet{duchi2021asymptotic}, Lemma 4.2]
\label{lemma:duchi:4.2}
Let $x^\star \in \mathbb R^s$ satisfy $A_{a}^\star x^\star = b_a^\star$ and $A_{i}^\star x^\star < b_i^\star$, and define $g = - A_a^{\star\top} \lambda \in \mathbb R^s$ for some $\lambda >0$. Moreover, let $x_n$ be the unique minimizer of the following optimization problem
\begin{align}
\label{eq:opt:x_k:v_k:delta_k}
\begin{array}{cl}
    \min &\inner{g}{x} + \inner{v_n}{x} + \frac{\delta_n}{2} \|x - x_0\|^2 \\
    \text{s.t.} & x \in \mathbb R^s \\
    ~ & \begin{bmatrix} A_{a}^\star \\ A_{i}^\star \end{bmatrix}x \leq \begin{bmatrix} b_a^\star \\ b_i^\star \end{bmatrix},
\end{array}
\end{align}
with $(v_n,\delta_n) \in \mathbb R^s \times \mathbb R_{++}$, and $x^\star$ be a minimizer of \eqref{eq:opt:x_k:v_k:delta_k} for $v_n = 0$, $\delta_n = 0$. If $(\delta_n, v_n) \to 0$ and $x_n \to x^\star$ as $n \to \infty$, then there exists $N \in \mathbb N$ such that $A_{a}^\star x_n = b_a^\star$ for all $n \geq N$.
\end{lemma}

\begin{lemma}[\citet{dembo2016lecture}, Exercise 5.3.35]
\label{lemma:conv:martingale}
Let $M_n$ be a martingale adapted to the filtration $\mathcal F_n$, and let $\{d_n\}_{n\in \mathbb N}$ be a positive, non-random sequence satisfying $d_n \uparrow \infty$. If $\sum_{n=1}^\infty d_n^{-2} \mathbb E[\|M_{n}-M_{n-1}\|^2 | \mathcal F_{n-1}] < \infty$, then $d_n^{-1}M_n \overset{\text{a.s.}}{\to} 0$.
\end{lemma}

\begin{lemma}[\citet{hall2014martingale}, Corollary~3.1]
\label{lemma:clt:martingale}
Let $\{S_{n,i}, \mathcal F_{n,i}, 1\leq i \leq k_n, n \geq 1\}$ be a zero-mean, square-integrable martingale array with differences $Y_{n,i} = S_{n,i} - S_{n,i-1}$, and let $\eta^2$ be an a.s. finite random variable. It is assumed that $k_n \uparrow \infty$ as $n \to \infty$. Suppose that:
\begin{itemize}
    \item[(i)] for all $\varepsilon >0$, $\sum_{i=1}^{k_n} \mathbb E \left[ Y_{n,i}^2 \, \mathbf{1}_{\{|Y_{n,i}| \, > \, \varepsilon\}} | \mathcal F_{n,i-1}\right] \overset{P}{\to} 0$ as $n \to \infty$,

    \item[(ii)] $\sum_{i=1}^{k_n} \mathbb E \left[ Y_{n,i}^2 | \mathcal F_{n,i-1}\right] \overset{P}{\to} \eta^2$ as $n \to \infty$,

    \item[(ii)] the $\sigma$-fields are nested: $\mathcal F_{n,i} \subseteq \mathcal F_{n,i+1}$ for $1 \leq i \leq k_n$, $n \geq 1$.
\end{itemize}
Then, $S_{n,k_n} = \sum_{i=1}^{k_n} Y_{n,i} \overset{d}{\to} Z$ (stably), where the random variable $Z$ has characteristic function $\mathbb E \left[ \text{exp}(-\frac{1}{2}\eta^2 t^2) \right]$.
\end{lemma}

\subsection{A Generic Asymptotic Normality Result}
\label{subsec:generic:asym:nor}

To keep the paper self-contained, in what follows we recall the generic asymptotic normality result from Section~13 (pages~12-13) in the Supplementary material of \cite{duchi2021asymptotic}. There, Duchi and Ruan generalize Polyak and Juditsky’s results \cite{polyak1992acceleration} on asymptotic normality in averaged stochastic gradient methods restricted to an arbitrary subspace of $\mathbb R^s$. Before stating the result, we need to introduce some notation, as well as two assumptions.

Given $\mathcal T := \{x \in \mathbb R^s:\; A x = 0\}$ a subspace of $\mathbb R^s$, we denote by $\text{P} \in \mathbb R^{s \times s}$ the orthogonal projector onto $\mathcal T$. Moreover, let $\xi_n$ be a martingale difference process adapted to a filtration $\mathcal F_n$, and let $\{x_n\}_{n \in \mathbb N}$, $\{\zeta_n\}_{n \in \mathbb N}$, $\{\epsilon_n\}_{n \in \mathbb N}$, and $\{\Delta_n := x_n -x^\star\}_{n \in \mathbb N}$ be sequences of vectors in $\mathbb R^s$ adapted to the same filtration $\mathcal F_n$. Here, $x^\star$ denotes some vector in $\mathbb R^s$. Assume that for a matrix $H \in \mathbb R^{s \times s}$ we have the recursion
\begin{align}
\label{eq:generic:iteration:CLT}
    \Delta_{n+1} = \Delta_n - \alpha_{n+1} \text{P} H \text{P} \Delta_n - \alpha_{n+1} \text{P} (\xi_n + \zeta_n) + \epsilon_n,
\end{align}
where $\Delta_0 \in \mathcal T$ and $\epsilon_n \in \mathcal T$, for all $n \in \mathbb N$. The asymptotic normality result requires the following two assumptions.

\begin{assumption}[Generic Asymptotic Normality I]~
\label{assump:generic:CLT:1}
\begin{enumerate}[label=(\roman*)]
    \item \label{assump:generic:CLT:1:1} There exists $c>0$ such that for all $w \in \mathcal T$ we have $w^\top H w \geq c \|w\|^2$.

    \item \label{assump:generic:CLT:1:2} There exists $C<\infty$ such that $\mathbb E\left[ \|\xi_n\|^2 | \mathcal F_{n-1} \right] \leq C$. Moreover, for some $\Sigma \geq 0$,
    \begin{align*}
        \frac{1}{\sqrt{n}} \sum_{k=0}^{n-1} \xi_k \overset{d}{\to} \mathcal N(0,\Sigma).
    \end{align*}
\end{enumerate}
\end{assumption}

We would like to highlight that Assumption~\ref{assump:generic:CLT:1}\ref{assump:generic:CLT:1:2} is slightly weaker compared to Assumption~S.A in Section~13 in the Supplementary material of \cite{duchi2021asymptotic}. Under Assumption~\ref{assump:IS}\ref{ass:IS:dominated}, this weaker version will be enough for us to prove the CLT in Theorem~\ref{thm:asymp:norm}.

\begin{assumption}[Generic Asymptotic Normality II]~
\label{assump:generic:CLT:2}
\begin{enumerate}[label=(\roman*)]
    \item \label{assump:generic:CLT:2:1} The sequence $\{\zeta_n\}_{n \in \mathbb N}$ satisfies 
    \begin{align*}
        \frac{1}{\sqrt{n}} \sum_{k=0}^{n-1} \| \text{P} \zeta_k\| \overset{\text{a.s.}}{\to} 0.
    \end{align*}
    
    \item \label{assump:generic:CLT:2:2} There exists a random variable $N < \infty$ such that $\epsilon_n = 0$ for $n \geq N$.

    \item \label{assump:generic:CLT:2:3} The iterates $x_n$ satisfy $x_n \overset{\text{a.s.}}{\to} x^\star$ and
    \begin{align*}
        \frac{1}{\sqrt{n}} \sum_{k=0}^{n-1} \| x_k - x^\star \|^2 \overset{\text{a.s.}}{\to} 0.
    \end{align*}
\end{enumerate}
\end{assumption}

We are now ready to state the generic asymptotic normality result.

\begin{lemma}[\citet{duchi2021asymptotic}, Proposition S.1]
\label{lemma:duchi:CLT}
Let Assumptions~\ref{assump:generic:CLT:1} and \ref{assump:generic:CLT:2} hold for the recursion \eqref{eq:generic:iteration:CLT} with $\Delta_n := x_n -x^\star$. Then,
\begin{align}
\label{eq:lemma:duchi:CLT}
    \frac{1}{\sqrt{n}} \sum_{k=0}^{n-1} \Delta_k \overset{d}{\to} \mathcal N(0, (\text{P} H \text{P})^\dagger \Sigma (\text{P} H \text{P})^\dagger).
\end{align}
\end{lemma}

We highlight the difference between the asymptotic variance in \eqref{eq:lemma:duchi:CLT} and the asymptotic variance in \cite[Proposition S.1 in Supplementary material]{duchi2021asymptotic}, which is $(\text{P} H \text{P})^\dagger \text{P} \Sigma \text{P} (\text{P} H \text{P})^\dagger$. The two expressions are equal. This follows from the fact that $\text{P}$ is an orthogonal projection matrix, which guarantees that $\text{P} (\text{P} H \text{P})^\dagger = (\text{P} H \text{P})^\dagger = (\text{P} H \text{P})^\dagger \text{P}$.

% ------------------------------------------------------------------------------------
% ------------------------------------------------------------------------------------
% ------------------------------------------------------------------------------------

% ------------------------------------------------------------------------------------
% ------------------------------------------------------------------------------------
% ------------------------------------------------------------------------------------

\section{Supporting Lemmas for Theorem~\ref{thm:as:conv}}
\label{appendix:proof:as:conv}

\begin{lemma}
\label{lemma:as:conv:theta}
Consider the iteration
\begin{align}
\label{eq:NDA:theta}
    \theta_{n+1} &= \argmin_{\theta \in \Theta} \left\{ \inner{\sum_{k=0}^n \alpha_{k+1} G_k}{\theta} +\frac{1}{2} \| \theta\|^2\right\}
\end{align}
with $G_k\defn G_{\mu_k}(\theta_k, X_{k+1}^{(\mu_k)}) \defn G(\theta_k, X_{k+1}^{(\mu_k)})\ell(X_{k+1}^{(\mu_k)},\mu_k)$, for some arbitrary importance sampling parameter sequence $\{\mu_k\}_{k\in\mathbb N} \subset \mathcal M$. Moreover, let Assumptions~\ref{assump:as:conv}\ref{assump:as:conv:1}, \ref{assump:as:conv:2}, \ref{assump:as:conv:4} and Assumption~\ref{assump:IS}\ref{ass:IS:dominated} be satisfied, and let $b_{n+1} = \sum_{k=0}^n \alpha_{k+1}$. Then, 
\begin{enumerate}[label=(\roman*)]
    \item \label{lemma:as:conv:theta:1} $\sum_{n=0}^\infty \alpha_{n+1}\|\theta_n - \theta^\star\|^2 < \infty$;
    
    \item \label{lemma:as:conv:theta:2} $b_{n+1}^{-1/2}\sum_{k=0}^n \alpha_{k+1}(G_k - \nabla f(\theta_k)) \overset{\text{a.s.}}{\to} 0$;

    \item \label{lemma:as:conv:theta:3} $\left\|\sum_{k=0}^n \alpha_{k+1} (\nabla f(\theta_k) - \nabla f(\theta^\star))  \right\|^2 \leq C b_{n+1}$, with probability one, for some (random) finite $C$;

    \item \label{lemma:as:conv:theta:4} $\theta_n \overset{\text{a.s.}}{\to} \theta^\star$;

    \item \label{lemma:as:conv:theta:5} There exists some random $N < \infty$ such that $A_{a}^{\theta_n} = A_{a}^\star$ and $A_{i}^{\theta_n} = A_{i}^\star$, for all $n \geq N$.
\end{enumerate}
\end{lemma}
\begin{proof}
The proof follows along similar lines as \citet[Theorems~2 and 3]{duchi2021asymptotic}. For completeness, and since these results are fundamental for proving Theorem~\ref{thm:as:conv}, we provide a full proof for Lemma~\ref{lemma:as:conv:theta}. 

\medskip

\emph{Assertion \ref{lemma:as:conv:theta:1}.} We start by defining $R_{n+1}$ as
\begin{align*}
    R_{n+1} := \inner{\sum_{k=0}^n \alpha_{k+1} G_k + \theta_{n+1}}{\theta^\star-\theta_{n+1}} + \frac{1}{2}\|\theta_{n+1}-\theta^\star\|^2.
\end{align*}
Since $\theta_{n+1}$ is the optimal solution in the minimization problem \eqref{eq:NDA:theta}, the first-order optimality condition guarantees that 
\begin{align*}
    \inner{\sum_{k=0}^n \alpha_{k+1} G_k + \theta_{n+1}}{\theta-\theta_{n+1}} \geq 0,
\end{align*}
for all $\theta\in \Theta$. In particular, this holds true for $\theta^\star \in \Theta$, showing that $R_{n+1} \geq 0$. Standard algebraic manipulations show that $R_{n+1}$ can be rewritten as
\begin{align*}
    R_{n+1} &= \inner{\sum_{k=0}^n \alpha_{k+1} G_k}{\theta^\star} + \frac{1}{2}\|\theta^\star\|^2 + \inner{-\sum_{k=0}^n \alpha_{k+1} G_k}{\theta_{n+1}} - \frac{1}{2}\|\theta_{n+1}\|^2\\
    &= \inner{\sum_{k=0}^n \alpha_{k+1} G_k}{\theta^\star} + \frac{1}{2}\|\theta^\star\|^2 + \max_{\theta \in \Theta} \left\{ -\inner{\sum_{k=0}^n \alpha_{k+1} G_k}{\theta} - \frac{1}{2}\|\theta\|^2\right\}\\
    &= \inner{\sum_{k=0}^n \alpha_{k+1} G_k}{\theta^\star} + \frac{1}{2}\|\theta^\star\|^2 + \max_{\theta \in \mathbb R^s} \left\{ \inner{-\sum_{k=0}^n \alpha_{k+1} G_k}{\theta} - \frac{1}{2}\|\theta\|^2 - \delta_{\Theta}(\theta)\right\},
\end{align*}
with $\delta_\Theta$ the indicator function of the set $\Theta$. Defining $\ell(\theta):= \frac{1}{2}\|\theta\|^2 + \delta_{\Theta}(\theta)$, we have that the above maximum is precisely the convex conjugate $\ell^*(-\sum_{k=0}^n \alpha_{k+1} G_k)$. Now, since $\min_{\theta \in \mathbb R^s} \{ \inner{\sum_{k=0}^n \alpha_{k+1} G_k}{\theta} + \frac{1}{2}\|\theta\|^2 + \delta_{\Theta}(\theta)\}$ is the Moreau envelope of $\sum_{k=0}^n \alpha_{k+1} G_k + \delta_{\Theta}(\theta)$ evaluated at zero, we have that the gradient $\nabla \ell^*(-\sum_{k=0}^n \alpha_{k+1} G_k)$ is equal to proximal map evaluated at zero \cite[Theorem~2.26]{rockafellar2009variational}, i.e.,
\begin{align*}
    \nabla \ell^*\left(-\sum_{k=0}^n \alpha_{k+1} G_k\right) = \argmax_{\theta \in \mathbb R^s} \left\{ \inner{\sum_{k=0}^n \alpha_{k+1} G_k}{\theta} + \frac{1}{2}\|\theta\|^2 + \delta_{\Theta}(\theta) \right\}=\theta_{n+1}.
\end{align*}
Moreover, since $\ell$ is $1$-strongly convex, we have that $\ell^*$ is $1$-smooth. Therefore, we can upper-bound $\ell^*(-\sum_{k=0}^n \alpha_{k+1} G_k)$ as
\begin{align*}
    \ell^*\left( -\sum_{k=0}^n \alpha_{k+1} G_k \right) \leq \ell^*\left( -\sum_{k=0}^{n-1} \alpha_{k+1} G_k \right) - \alpha_{n+1} \inner{G_n}{\theta_n} + \frac{1}{2} \|\alpha_{n+1} G_n\|^2.
\end{align*}
Using this, we can upper-bound $R_{n+1}$ as
\begin{align*}
    R_{n+1} &\leq \inner{\sum_{k=0}^n \alpha_{k+1} G_k}{\theta^\star} + \frac{1}{2}\|\theta^\star\|^2 + \ell^*\left( -\sum_{k=0}^{n-1} \alpha_{k+1} G_k \right) - \alpha_{n+1} \inner{G_n}{\theta_n} + \frac{1}{2} \|\alpha_{n+1} G_n\|^2 \\
    &= R_n - \alpha_{n+1}\inner{G_n}{\theta_n - \theta^\star} + \frac{\alpha_{n+1}^2}{2}\|G_n\|^2.
\end{align*}
Since $R_n$ is adapted to the filtration $\mathcal F_n = \sigma(X_k^{(\mu_{k-1})} |\, k \leq n)$, we can take the conditional expectation $\mathbb E[\cdot | \mathcal F_n]$ on both sides and obtain
\begin{align}
\label{eq:R_n:RobbinsSiegmund}
    \mathbb E[R_{n+1} | \mathcal F_n] \leq R_n - \alpha_{n+1}\inner{\nabla f(\theta_n)}{\theta_n - \theta^\star} + \frac{\alpha_{n+1}^2}{2} \mathbb E[\|G_n\|^2 | \mathcal F_n],
\end{align}
where we have used the fact that
\begin{align*}
    \mathbb E[G_n | \mathcal F_n] = \mathbb E_{X \sim \mathbb P_{\mu_n}}\left[ G_{\mu_n}(\theta_n, X)  \ell(X,\mu_n) \right] = \mathbb E_{X \sim \mathbb P}\left[ G(\theta_n, X) \right] = \nabla f(\theta_n).
\end{align*}
Now, from the first-order optimality condition in \eqref{eq:so}, we have that $\inner{\nabla f(\theta_n)}{\theta_n - \theta^\star} \geq 0$. Using this and the facts $\mathbb E[\|G_n\|^2 | \mathcal F_n] \leq G_M^2$ (which follows from Assumption~\ref{assump:IS}\ref{ass:IS:dominated}) and $\sum_{n}\alpha_{n}^2 < \infty$, we have that \eqref{eq:R_n:RobbinsSiegmund} is as in Lemma~\ref{lemma:robbins:siegmund} with $A_n = 0$, $B_n = G_M^2 \alpha_{n+1}^2/2$, and $C_n = \alpha_{n+1}\inner{\nabla f(\theta_n)}{\theta_n - \theta^\star}$. In particular, notice that $C_n$ is $\mathcal F_n$-adapted. Therefore, Lemma~\ref{lemma:robbins:siegmund} guarantees that there is a random variable $R_\infty < \infty$ such that $R_n \overset{\text{a.s.}}{\to} R_\infty$, and that, with probability one,
\begin{align}
\label{eq:C_n}
    \sum_{n=0}^\infty \alpha_{n+1}\inner{\nabla f(\theta_n)}{\theta_n - \theta^\star} < \infty.
\end{align}
Using Inequality \eqref{eq:C_n} and Assumption \ref{assump:as:conv}\ref{assump:as:conv:1}, we obtain
\begin{align*}
    \sum_{n=0}^\infty \alpha_{n+1}\|\theta_n - \theta^\star\|^2 \leq \frac{1}{c_1}\sum_{n=0}^\infty \alpha_{n+1} ({f}(\theta_n) - {f}(\theta^\star)) \leq  \frac{1}{c_1} \sum_{n=0}^\infty \alpha_{n+1}\inner{\nabla {f}(\theta_n)}{\theta_n - \theta^\star} < \infty.
\end{align*}
This finishes the proof of Assertion \ref{lemma:as:conv:theta:1}.

\medskip

\emph{Assertion \ref{lemma:as:conv:theta:2}.} First notice that 
\begin{align*}
    M_{n+1} := \sum_{k=0}^{n} \alpha_{k+1}(G_k - \nabla f(\theta_k) )
\end{align*}
is a martingale adapted to the filtration $\mathcal F_{n+1} = \sigma(X_k^{(\mu_{k-1})} |\, k \leq n+1)$. Letting $d_{n} = \sqrt{b_{n}}$, we want to prove that $d_{n+1}^{-1} M_{n+1}  \overset{\text{a.s.}}{\to} 0$. From Lemma~\ref{lemma:conv:martingale}, we know that this holds if $\sum_{n=1}^\infty d_{n+1}^{-2} \mathbb E[\|M_{n+1}-M_{n}\|^2 | \mathcal F_{n}] < \infty$. In our notation, this is equivalent to showing that
\begin{align*}
    \sum_{n=0}^\infty \frac{1}{b_{n+1}} \mathbb E \left.\left[\left\|\alpha_{n+1} (G_n - \nabla f(\theta_n) ) \right\|^2 \right| \mathcal F_{n}  \right] < \infty.
\end{align*}
Due to Assumption~\ref{assump:IS}\ref{ass:IS:dominated}, we know that $\mathbb E[\|G_n\|^2 | \mathcal F_n] \leq G_M^2$, and using Jensen's inequality we have that $\|\nabla f(\theta_n)\|^2 \leq G_M^2$. Therefore, there exists some constant $c$ such that 
\begin{align*}
    \mathbb E \left.\left[\left\| G_n - \nabla f(\theta_n)\right\|^2 \right| \mathcal F_{n}  \right] \leq c,
\end{align*}
and consequently we only need to show that $\sum_{n=1}^\infty c {\alpha_{n+1}^2}/{b_{n+1}} < \infty$. This follows immediately from $\sum_{n=1}^\infty {\alpha_{n+1}^2} < \infty$, concluding the proof of assertion \ref{lemma:as:conv:theta:2}.

\medskip

\emph{Assertion \ref{lemma:as:conv:theta:3}.} We start by rewriting and upper bounding $\left\|\sum_{k=0}^n \alpha_{k+1} (\nabla {f}(\theta_k) - \nabla{f}(\theta^\star))  \right\|^2$ as
\begin{align*}
    b_{n+1}^2\left\|\sum_{k=0}^n \frac{\alpha_{k+1}}{b_{n+1}}  (\nabla {f}(\theta_k) - \nabla{f}(\theta^\star) ) \right\|^2 \leq b_{n+1}\sum_{k=0}^n \alpha_{k+1}\left\| \nabla {f}(\theta_k) - \nabla{f}(\theta^\star) \right\|^2,
\end{align*}
where the inequality follows from Jensen's inequality. Now, from Assumption~\ref{assump:as:conv}(ii.1) we have that
\begin{align*}
    \sum_{k=0}^n \alpha_{k+1}\|\nabla f(\theta_k) - \nabla f(\theta^\star)\|^2 \leq \sum_{k=0}^n \alpha_{k+1} c_2^2 \|\theta_k - \theta^\star\|^2 \leq \sum_{k=0}^\infty \alpha_{k+1} c_2^2 \|\theta_k - \theta^\star\|^2 < \infty,
\end{align*}
where the last inequality follows from assertion \ref{lemma:as:conv:theta:1}. This concludes the proof of assertion \ref{lemma:as:conv:theta:3}.

\medskip

\emph{Assertion \ref{lemma:as:conv:theta:4}.} For this, we will prove that $R_\infty = 0$, which implies that $\theta_n \overset{\text{a.s.}}{\to} \theta^\star$. We start by defining $b_{n+1} := \sum_{k=0}^n \alpha_{k+1}$. Since $\alpha_n = \alpha/n^\gamma$, for $\gamma \in (1/2,1)$, it can be shown that $\sum_{k=0}^\infty \left(\alpha_{k+1}/b_{k+1}\right) = \infty$. Moreover, using the inequality in the last display, we know that 
\begin{align*}
    \sum_{n=0}^\infty \frac{\alpha_{n+1}}{b_{n+1}} \left(b_{n+1}\left\|\theta_{n}-\theta^\star \right\|^2\right) < \infty.
\end{align*}
Therefore, there exists a subsequence $\{\theta_{n_i}\}_{i \in \mathbb N}$ for which, with probability one,
\begin{align}
\label{eq:conv:on:subsequence:1}
    \lim_{i \to \infty} b_{n_i+1}\left\|\theta_{n_i}-\theta^\star \right\|^2 = 0.
\end{align}
We are now ready to prove that $R_\infty = 0$. We start by bounding $R_{n+1}$ as
\begin{align}
\label{eq:bounds:R_n+1:1}
    R_{n+1} \leq \inner{\sum_{k=0}^n \alpha_{k+1} G_k}{\theta^\star-\theta_{n+1}}
    + \left\|\theta_{n+1}\right\|\left\|\theta^\star-\theta_{n+1}\right\|
    + \frac{1}{2}\left\| \theta_{n+1}-\theta^\star \right\|^2.
\end{align}
By restricting our attention to the subsequence $\{\theta_{n_i}\}_{i \in \mathbb N}$, and by using \eqref{eq:conv:on:subsequence:1} and the compactness of $\Theta$, we have that
\begin{align}
\label{eq:bounds:R_n+1:last:two:terms:1}
    0\leq \left\|\theta_{n_i+1}\right\|\left\|\theta^\star-\theta_{n_i+1}\right\|
    + \frac{1}{2}\left\| \theta_{n_i+1}-\theta^\star  \right\|^2 \overset{\text{a.s.}}{\to} 0.
\end{align}
We now focus on the first term on the right-hand side in \eqref{eq:bounds:R_n+1:1}, which we rewrite as
\begin{align*}
    \inner{\sum_{k=0}^n \alpha_{k+1} \left( (G_k - \nabla {f}(\theta_k)) + (\nabla {f}(\theta_k) - \nabla{f}(\theta^\star))  + \nabla{f}(\theta^\star)\right)}{\theta^\star-\theta_{n+1}}.
\end{align*}
By restricting our attention to the subsequence $\{\theta_{n_i}\}_{i \in \mathbb N}$, we have that
\begin{align*}
    \frac{1}{\sqrt{b_{n_i+1}}} \left\| \sum_{k=0}^{n_i} \alpha_{k+1} (G_k - \nabla {f}(\theta_k))  \right\| \sqrt{b_{n_i+1}} \left\| \theta^\star-\theta_{n_i+1}  \right\| \overset{\text{a.s.}}{\to} 0,
\end{align*}
using Lemma~\ref{lemma:noise:convergence} and \eqref{eq:conv:on:subsequence:1}. Moreover, 
\begin{align*}
    \left\| \sum_{k=0}^{n_i} \alpha_{k+1}(\nabla {f}(\theta_k) - \nabla{f}(\theta^\star)) \right\| \left\| \theta^\star-\theta_{n_i+1} \right\| \leq \sqrt{C} \sqrt{b_{n_i+1}} \left\| \theta^\star-\theta_{n_i+1}  \right\|\overset{\text{a.s.}}{\to} 0,
\end{align*}
where the inequality follows from Lemma~\ref{lemma:gradients:error:bound} and the convergence follows from \eqref{eq:conv:on:subsequence:1}. Finally, 
\begin{align*}
    \inner{\sum_{k=0}^n \alpha_{k+1} \nabla{f}(\theta^\star)}{\theta^\star-\theta_{n_i+1} } \leq 0
\end{align*}
follows from the first-order optimality conditions for $\theta^\star$. The last three display equations show that, with probability one,
\begin{align}
\label{eq:bounds:R_n+1:first:term:1}
    \limsup_{i \to \infty}\; \inner{\sum_{k=0}^{n_i} \alpha_{k+1} G_k}{\theta^\star-\theta_{n_i+1}} \leq 0.
\end{align}
Now, from \eqref{eq:bounds:R_n+1:last:two:terms:1} and \eqref{eq:bounds:R_n+1:first:term:1}, we have that $R_{n_i} \overset{\text{a.s.}}{\to} 0$, and since $R_n \overset{\text{a.s.}}{\to} R_\infty$, we obtain $R_\infty = 0$. This guarantees the desired convergence, and concludes the proof of assertion \ref{lemma:as:conv:theta:4}.

\medskip

\emph{Assertion \ref{lemma:as:conv:theta:5}.} First notice that since $\theta_n \overset{\text{a.s.}}{\to} \theta^\star$, and since $A_{i}^\star \theta^\star < b_i^\star$, there exists some random $N < \infty$ such that $A_{i}^\star \theta_n < b_i^\star$, for all $n \geq N$. Therefore, $A_{i}^{\theta_n} = A_{i}^\star$, for all $n \geq N$. We will now show that there exists some random $N < \infty$ such that $A_{a}^{\theta_n} = A_{a}^\star$, for all $n \geq N$. We start by rewriting iteration~\eqref{eq:NDA:theta} as
\begin{align*}
    \theta_{n+1} = \argmin_{\theta \in \Theta} \left\{ \inner{g}{\theta} + \inner{v_n}{\theta} +\frac{1}{2 b_{n+1}} \left\| \theta \right\|^2\right\},
\end{align*}
with 
\begin{align*}
    g = \nabla {f}(\theta^\star) \quad \text{and} \quad  
    v_n = \frac{1}{b_{n+1}} \left(\sum_{k=0}^n \alpha_{k+1} G_k - b_{n+1} \nabla {f}(\theta^\star) \right).
\end{align*}
From the KKT conditions for $\theta^\star$ we have that there exist $\lambda \in \mathbb R_{+}^{p_1}$ (with $p_1$ the dimension of $b_a^\star$) such that $\nabla {f}(\theta^\star) + A_a^{\star \top} \lambda = 0$. Moreover, using Assumptions~\ref{assump:as:conv}(v) we know that $\lambda$ can be chosen strictly positive. Therefore, $g = -{A_a^\star}^\top \lambda$, for some $\lambda \in \mathbb R_{++}^{p_1}$. We will now prove that $v_n \to 0$. For this, we first upper bound $v_n$ by
\begin{align*}
    \left\| \frac{1}{b_{n+1}} \sum_{k=0}^n \alpha_{k+1} \left(G_k - \nabla{f}(\theta_k) \right) \right\| + \left\| \frac{1}{b_{n+1}} \sum_{k=0}^n \alpha_{k+1} \left(\nabla{f}(\theta_k) - \nabla{f}(\theta^\star) \right) \right\|,
\end{align*}
and then notice that the two terms converge almost surely to zero using assertions \ref{lemma:as:conv:theta:2} and \ref{lemma:as:conv:theta:3}. Finally, since $b_{n} \to \infty$, we have that $1/b_{n+1} \to 0$ and the result now follows from Lemma~\ref{lemma:duchi:4.2}. This concludes the proof of assertion \ref{lemma:as:conv:theta:5}.
\end{proof}

\begin{lemma}
\label{lemma:noise:convergence}
Consider the setting of Theorem~\ref{thm:as:conv}, and let $b_{n+1} = \sum_{k=0}^n \alpha_{k+1}$. Then, 
\begin{align}
\label{eq:noise:convergence}
    \frac{1}{\sqrt{b_{n+1}}}\sum_{k=0}^n \alpha_{k+1}\begin{bmatrix}G_k - \nabla {f}(\theta_k) \\ H_k - \nabla_\mu v(\theta_k,\mu_k)\end{bmatrix} \overset{\text{a.s.}}{\to} 0.
\end{align}
\end{lemma}
\begin{proof}
First notice that 
\begin{align*}
    M_{n+1} := \sum_{k=0}^{n} \alpha_{k+1}\begin{bmatrix}G_k - \nabla {f}(\theta_k) \\ H_k - \nabla_\mu v(\theta_k,\mu_k)\end{bmatrix}
\end{align*}
is a martingale adapted to the filtration $\mathcal F_{n+1} = \sigma(X_k^{(\mu_{k-1})}, X_k^{(\nu_{k-1})} |\, k \leq n+1)$. Letting $d_{n} = \sqrt{b_{n}}$, \eqref{eq:noise:convergence} is equivalent to proving that $d_{n+1}^{-1} M_{n+1}  \overset{\text{a.s.}}{\to} 0$. From Lemma~\ref{lemma:conv:martingale}, we know that this holds if $\sum_{n=1}^\infty d_{n+1}^{-2} \mathbb E[\|M_{n+1}-M_{n}\|^2 | \mathcal F_{n}] < \infty$. In our notation, this is equivalent to showing that
\begin{align*}
    \sum_{n=0}^\infty \frac{1}{b_{n+1}} \mathbb E \left.\left[\left\|\alpha_{n+1} \begin{bmatrix}G_n - \nabla {f}(\theta_n) \\ H_n - \nabla_\mu v(\theta_n,\mu_n)\end{bmatrix} \right\|^2 \right| \mathcal F_{n}  \right] < \infty.
\end{align*}
Due to Assumption~\ref{assump:IS}\ref{ass:IS:dominated}, we know that $\mathbb E[\|G_n\|^2 | \mathcal F_n] \leq G_M^2$ and $\mathbb E[\|H_n\|^2 | \mathcal F_n] \leq H_M^2$. Moreover, using Jensen's inequality, we have also then that $\|\nabla {f}(\theta_n)\|^2 \leq G_M^2$ and $\|\nabla_\mu v(\theta_n,\mu_n)\|^2 \leq H_M^2$. Therefore, for $c:=4G_M^2 + 4H_M^2$ we have that 
\begin{align*}
    \mathbb E \left.\left[\left\| \begin{bmatrix}G_n - \nabla {f}(\theta_n) \\ H_n - \nabla_\mu v(\theta_n,\mu_n)\end{bmatrix} \right\|^2 \right| \mathcal F_{n}  \right] \leq c,
\end{align*}
and consequently we only need to show that $\sum_{n=1}^\infty c {\alpha_{n+1}^2}/{b_{n+1}} < \infty$. This follows immediately from $\sum_{n=1}^\infty {\alpha_{n+1}^2} < \infty$.
\end{proof}

\begin{lemma}
\label{lemma:gradients:error:bound}
Consider the setting of Theorem~\ref{thm:as:conv}, and let $b_{n+1} = \sum_{k=0}^n \alpha_{k+1}$. Then, 
\begin{align}
\label{eq:gradients:error:bound}
    \left\|\sum_{k=0}^n \alpha_{k+1} \begin{bmatrix}\nabla {f}(\theta_k) - \nabla {f}(\theta^\star) \\ \nabla_\mu v(\theta_k,\mu_k) - \nabla_\mu v(\theta^\star,\mu^\star)\end{bmatrix} \right\|^2 \leq C b_{n+1},
\end{align}
with probability one, for some almost surely finite $C$.
\end{lemma}
\begin{proof}
We start by rewriting the left-hand side of \eqref{eq:gradients:error:bound} as
\begin{align*}
    b_{n+1}^2\left\|\sum_{k=0}^n \frac{\alpha_{k+1}}{b_{n+1}}  \begin{bmatrix}\nabla {f}(\theta_k) - \nabla {f}(\theta^\star) \\ \nabla_\mu v(\theta_k,\mu_k) -\nabla_\mu v(\theta^\star,\mu_k) + \nabla_\mu v(\theta^\star,\mu_k) - \nabla_\mu v(\theta^\star,\mu^\star)\end{bmatrix} \right\|^2,
\end{align*}
which, using Jensen's inequality, can be upper bounded by
\begin{align*}
    b_{n+1}\sum_{k=0}^n \alpha_{k+1}\left\| \begin{bmatrix}\nabla {f}(\theta_k) - \nabla {f}(\theta^\star) \\ \nabla_\mu v(\theta_k,\mu_k) -\nabla_\mu v(\theta^\star,\mu_k) + \nabla_\mu v(\theta^\star,\mu_k) - \nabla_\mu v(\theta^\star,\mu^\star)\end{bmatrix} \right\|^2.
\end{align*}
The result now follows from Assumption~\ref{assump:as:conv}\ref{assump:as:conv:2}, as we now show. From the first bound in this assumption we have that 
\begin{align*}
    \sum_{k=0}^n \alpha_{k+1}\|\nabla {f}(\theta_k) - \nabla{f}(\theta^\star)\|^2 \leq \sum_{k=0}^n \alpha_{k+1} c_2^2 \|\theta_k - \theta^\star\|^2 \leq \sum_{k=0}^\infty \alpha_{k+1} c_2^2 \|\theta_k - \theta^\star\|^2 < \infty,
\end{align*}
where the last inequality follows from \eqref{eq:sum:alpha:theta^2:mu^2} in the proof of Theorem~\ref{thm:as:conv}. Moreover, from Assumption~\ref{assump:as:conv}\ref{assump:as:conv:3} we have that
\begin{align*}
    \sum_{k=0}^n \alpha_{k+1}\|\nabla_\mu v(\theta_k,\mu_k) -\nabla_\mu v(\theta^\star,\mu_k)\|^2 \leq \sum_{k=0}^n \alpha_{k+1} c_2^2 \|\theta_k - \theta^\star\|^4 \leq \sum_{k=0}^n \alpha_{k+1} \tilde{c} \|\theta_k - \theta^\star\|^2 < \infty,
\end{align*}
where the first inequality holds with probability one for large enough $n$ using Proposition~\ref{prop:fin:time:ident}, and the second inequality holds for some appropriate constant $\tilde{c}$, and follows from the compactness of $\Theta$. Finally, from the second bound in Assumption~\ref{assump:as:conv}\ref{assump:as:conv:2} we have that
\begin{align*}
    \sum_{k=0}^n \alpha_{k+1}\|\nabla_\mu v(\theta^\star,\mu_k) - \nabla_\mu v(\theta^\star,\mu^\star)\|^2 \leq \sum_{k=0}^n \alpha_{k+1} c_2^2 \|\mu_k - \mu^\star\|^2 \leq \sum_{k=0}^\infty \alpha_{k+1} c_2^2 \|\mu_k - \mu^\star\|^2 < \infty,
\end{align*}
where the last inequality follows from \eqref{eq:sum:alpha:theta^2:mu^2} in the proof of Theorem~\ref{thm:as:conv}.
Putting the last four display equations together, we obtain \eqref{eq:gradients:error:bound}.
\end{proof}

% ------------------------------------------------------------------------------------
% ------------------------------------------------------------------------------------
% ------------------------------------------------------------------------------------

\section{Supporting Lemmas for Proposition~\ref{prop:fin:time:ident}}
\label{appendix:proof:fin:time:ident}

\begin{lemma}
\label{lemma:conv:gradients}
Consider the setting of Proposition~\ref{prop:fin:time:ident}, and let $b_{n+1} = \sum_{k=0}^n \alpha_{k+1}$. Then, 
\begin{align*}
    \frac{1}{b_{n+1}} \sum_{k=0}^n \alpha_{k+1} \begin{bmatrix}G_k \\ H_k \end{bmatrix} \overset{\text{a.s.}}{\to} \begin{bmatrix}\nabla{f}(\theta^\star) \\ \nabla v(\theta^\star,\mu^\star)\end{bmatrix}.
\end{align*}
\end{lemma}
\begin{proof}
We start by upper-bounding
\begin{align*}
    \left\| \frac{1}{b_{n+1}} \sum_{k=0}^n \alpha_{k+1} \begin{bmatrix}G_k \\ H_k \end{bmatrix} - \begin{bmatrix}\nabla f(\theta^\star) \\ \nabla v(\theta^\star,\mu^\star)\end{bmatrix} \right\| 
\end{align*}
by
\begin{align*}
    \left\| \frac{1}{b_{n+1}} \sum_{k=0}^n \alpha_{k+1} \left(\begin{bmatrix}G_k \\ H_k \end{bmatrix} - \begin{bmatrix}\nabla f(\theta_k) \\ \nabla_\mu v(\theta_k,\mu_k)\end{bmatrix}\right) \right\| + \left\| \frac{1}{b_{n+1}} \sum_{k=0}^n \alpha_{k+1} \left(\begin{bmatrix}\nabla f(\theta_k) \\ \nabla_\mu v(\theta_k,\mu_k)\end{bmatrix} - \begin{bmatrix}\nabla f(\theta^\star) \\ \nabla_\mu v(\theta^\star,\mu^\star)\end{bmatrix}\right) \right\|.
\end{align*}
Then, the first term converges almost surely to zero using Lemma~\ref{lemma:noise:convergence}, and the second term converges almost surely to zero using Lemma~\ref{lemma:gradients:error:bound}.
\end{proof}

% ------------------------------------------------------------------------------------
% ------------------------------------------------------------------------------------
% ------------------------------------------------------------------------------------

\section{Supporting Lemmas for Theorem~\ref{thm:asymp:norm}}

\begin{lemma}
\label{lemma:assump:conv:CLT}
Consider the setting of Theorem~\ref{thm:asymp:norm}. Then, 
\begin{align*}
    \frac{1}{\sqrt{n}} \sum_{k=0}^{n-1} \left\| \begin{bmatrix}
        \theta_k - \theta^\star \\ \mu_k - \mu^\star
    \end{bmatrix} \right\|^2 \overset{\text{a.s.}}{\to} 0.
\end{align*}
\end{lemma}
\begin{proof}
Notice that, by Kronecker's lemma, it is enough to prove that with probability one,
\begin{align*}
    \sum_{k=0}^{\infty} \frac{1}{\sqrt{k}} \left\| \begin{bmatrix}
        \theta_k - \theta^\star \\ \mu_k - \mu^\star
    \end{bmatrix} \right\|^2 < \infty.
\end{align*}
The term on the left-hand side can be upper bounded as
\begin{align*}
    \sum_{k=0}^{\infty} \frac{1}{\sqrt{k}} \left\| \begin{bmatrix}
        \theta_k - \theta^\star \\ \mu_k - \mu^\star
    \end{bmatrix} \right\|^2 
    \leq 
    \left( \sum_{k=0}^{\infty} \frac{1}{\alpha_{k+1} \sqrt{k}} \right) 
    \left( \sum_{k=0}^{\infty} \alpha_{k+1} \left\| \begin{bmatrix}
        \theta_k - \theta^\star \\ \mu_k - \mu^\star
    \end{bmatrix} \right\|^2 \right)<\infty,
\end{align*}
with probability one, where the last inequality follows from equation~\eqref{eq:sum:alpha:theta^2:mu^2} in the proof of Theorem~\ref{thm:as:conv} and the fact that $\alpha_n = \alpha/n^\gamma$, for $\gamma \in (1/2,1)$, which guarantees that $\sum_{k=0}^\infty 1/ (\alpha_{k+1} \sqrt{k}) < \infty$.
\end{proof}

\begin{lemma}
\label{lemma:CLT:gradient:error}
Consider the setting of Theorem~\ref{thm:asymp:norm}. Then, 
\begin{align*}
    \frac{1}{\sqrt{n}} \sum_{k=0}^{n-1} \left(\begin{bmatrix}G_k \\ H_k \end{bmatrix} - \begin{bmatrix}\nabla{f}(\theta_k) \\ \nabla_\mu v(\theta_k,\mu_k)\end{bmatrix}\right) \overset{d}{\to} \mathcal N \left( 0, \Sigma^\star \right),
\end{align*}
with
\begin{align*}
    \Sigma^\star = \begin{bmatrix}
        \text{Var}_{X^{(\mu^\star)} \sim \mathbb P_{\mu^\star}}\left[ G_{\mu^\star}(\theta^\star,X^{(\mu^\star)})\right] & 0 \\ 0 & \text{Var}_{X \sim \mathbb P}\left[ H(\theta^\star, \mu^\star, X)\right]
    \end{bmatrix}.
\end{align*}
\end{lemma}
\begin{proof}
The proof follows from \cite[Corollary~3.1]{hall2014martingale}, using Assumption~\ref{assump:asymp:norm:unif:int} and the Cramér-Wold theorem. Fix $t \in \mathbb R^{s+m}$. Then, recalling the shorthand notation $\xi_k$ for the noise vector, we have that
\begin{align*}
    S_{n, i} := \frac{1}{\sqrt{n}} \sum_{k=0}^{i-1} t^\top \xi_k,
\end{align*}
with $i \leq n$, defines a zero-mean, square-integrable martingale array, adapted to the filtration $\mathcal F_{n,i} := \mathcal F_i = \sigma(X_k^{(\mu_{k-1})}, X_k |\, k \leq i)$, and with differences
\begin{align*}
    Y_{n,i} := S_{n,i} - S_{n,i-1} = \frac{1}{\sqrt{n}} t^\top \xi_{i-1}.
\end{align*}
The square-integrability follows immediately from Assumption~\ref{assump:IS}\ref{ass:IS:dominated}. We will now verify that the three conditions of Lemma~\ref{lemma:clt:martingale} are satisfied. For this, notice that Assumption~\ref{assump:asymp:norm:unif:int}\ref{assump:asymp:norm:1:1} is equivalent to the first condition. Moreover, the second condition follows immediately from Assumption~\ref{assump:asymp:norm:unif:int}\ref{assump:asymp:norm:1:2} with
\begin{align*}
     \eta^2 := t^\top \Sigma^\star t,
\end{align*}
using the fact that $Y_{n,i}^2 = t^\top \left(\xi_{i-1} \xi_{i-1}^\top\right) t$. Finally, the $\sigma$-fields $\mathcal F_{n,i}$ are clearly nested by definition. Therefore, Lemma~\ref{lemma:clt:martingale} implies that
\begin{align*}
    S_{n,n} =  t^\top \left( \frac{1}{\sqrt{n}} \sum_{k=0}^{n-1} \xi_k \right) \overset{d}{\to} \mathcal N \left (0, \eta^2 \right).
\end{align*}
Now notice that $\mathcal N \left (0, \eta^2 \right) = t^\top \mathcal N(0, \Sigma^\star)$. Therefore, by the Cramér-Wold theorem, we have that $\frac{1}{\sqrt{n}} \sum_{k=0}^{n-1} \xi_k \overset{d}{\to} \mathcal N(0, \Sigma^\star)$, which concludes the proof.
\end{proof}

\end{document}